%% file: arrabit.tex
\documentclass[final,leqno]{siamltex}

\usepackage[algo2e,linesnumbered,vlined,ruled]{algorithm2e}
\usepackage{graphicx,graphics,subfigure}
\usepackage{hyperref,url}
\usepackage{epsf,epstopdf}

\usepackage{comment}
\usepackage{times,hyperref}
\usepackage{amsmath,amsxtra,amsfonts,amscd,amssymb,bm}
\usepackage[usenames]{color}
\usepackage[normalem]{ulem}
\usepackage{exscale}

\newcommand{\be}{\begin{equation}}
\newcommand{\ee}{\end{equation}}
\newcommand{\bee}{\begin{equation*}}
\newcommand{\eee}{\end{equation*}}
\newcommand{\bea}{\begin{eqnarray}}
\newcommand{\eea}{\end{eqnarray}}
\newcommand{\beaa}{\begin{eqnarray*}}
\newcommand{\eeaa}{\end{eqnarray*}}

\newcommand{\st}{\,\mathbf{s.t.}\,}  
\newcommand{\R}{\mathbb{R}}  
  
\newcommand{\cC}{\mathcal{C}}  
  
\newcommand{\cS}{{\mathcal{S}}}

\newcommand{\RR}{\mathrm{RR}}
\newcommand{\tr}{\mathrm{tr}}
\newcommand{\zz}{^{\mathrm{T}}}

\newcommand{\fs}{^2_{\mathrm{F}}}

\newcommand{\ARR}{\textsc{Arr}} 
\newcommand{\spa}{\mathrm{span}} 
\newcommand{\tol}{\mathrm{tol}} 
\newcommand{\maxres}{\mathtt{maxres}} 
\newcommand{\res}{\mathtt{res}} 
\newcommand{\orth}{\mathbf{orth}}

\newcommand{\Xp}{\mathbf{X}_{p}}
\newcommand{\Yp}{\mathbf{Y}_{p}}
\newcommand{\cR}{{\cal{R}}}
\newcommand{\poly}{\textsc{poly}}
\newcommand{\gn}{\textsc{gn}}
\newcommand{\eigs}{\textsc{eigs}}

\newcommand{\mpm}{\textsc{mpm}}
\newcommand{\feast}{\textsc{feast}}

\newcommand{\arrabit}{\textsc{arrabit}}
\newcommand{\arpack}{\textsc{arpack}}

\DeclareMathOperator*{\argmax}{arg\,max}
\allowdisplaybreaks
\begin{document}

\title{
Block algorithms with augmented Rayleigh-Ritz projections
for large-scale eigenpair computation
} 

\author{
Zaiwen Wen\footnotemark[3]  
\and  
Yin Zhang\footnotemark[4]
}

\renewcommand{\thefootnote}{\fnsymbol{footnote}}

\footnotetext[3]{Beijing International Center for Mathematical
Research, Peking University, Beijing, CHINA (wenzw@pku.edu.cn).
Research supported in part by NSFC grants 11322109 and 11421101, and by the National
Basic Research Project under the grant 2015CB856000.}

\footnotetext[4]{Department of Computational and Applied Mathematics,
Rice University, Houston, UNITED STATES (yzhang@rice.edu). Research supported
in part by NSF DMS-1115950 and NSF DMS-1418724.}
\renewcommand{\thefootnote}{\arabic{footnote}}

%\date{\today} 
\maketitle

\begin{abstract}

Most iterative algorithms for eigenpair computation consist of two main steps: a subspace update (SU) step that generates bases for approximate eigenspaces, followed by a Rayleigh-Ritz (RR) projection step that extracts approximate eigenpairs.  So far the predominant methodology for the SU step is based on Krylov subspaces that builds orthonormal bases piece by piece in a sequential manner.  In this work, we investigate block methods in the SU step that allow a higher level of concurrency than what is reachable by Krylov subspace methods.  To achieve a competitive speed, we propose an augmented Rayleigh-Ritz (ARR) procedure and analyze its rate of convergence under realistic conditions.  Combining this ARR procedure with a set of polynomial accelerators, as well as utilizing a few other techniques such as continuation and deflation, we construct a block algorithm designed to reduce the number of RR steps and elevate concurrency in the SU steps.  Extensive computational experiments are conducted in Matlab on a representative set of test problems to evaluate the performance of two variants of our algorithm in comparison to two well-established, high-quality eigensolvers \arpack~ and \feast.  Numerical results, obtained on a many-core computer without explicit code parallelization, show that when computing a relatively large number of eigenpairs, the performance of our algorithms is competitive with, and frequently superior to, that of the two state-of-the-art eigensolvers.

\end{abstract}

%\begin{keywords} 
%Extreme eigenpairs, Principal SVD,
%Power method, Polynomial Acceleration,
%Orthogonalization, Rayleigh-Ritz procedure,
%Optimal Scalability.
%\end{keywords}
%\begin{AM} 15A18,  65F15,  65K05, 90C06 \end{AM}

\pagestyle{myheadings}
\thispagestyle{plain}
\markboth{Z. WEN, AND Y. ZHANG}
{Block algorithms with an ARR procedure
for large-scale exterior eigenpair computation
}

\section{Introduction}

For a given real symmetric matrix $A \in \R^{n\times n}$, let 
$\lambda_1, \lambda_2, \cdots, \lambda_n$ be the eigenvalues of $A$
sorted in an descending order: $ \lambda_1 \ge \lambda_2 \ge \cdots \ge
 \lambda_n$, and $q_1, \ldots, q_n\in \R^n$  be the corresponding eigenvectors 
such that $A q_i=\lambda_i q_i$, $\|q_i\|_2=1$, $i=1,\ldots,n$ and $q_i^T q_j=0$ 
for $i\neq j$.  The eigenvalue decomposition of $A$ is defined as 
$A = Q_n\Lambda_n Q_n^T$, where, for any integer $i \in [1,n]$, 
\begin{equation} \label{eq:Q}
Q_i = [q_1, \; q_2, \; \ldots, q_i] \in \R^{n\times i}, \; \mbox{  }  \; 
\Lambda_i = \diag(\lambda_1,\lambda_2, \ldots, \lambda_i) \in \R^{i \times i},
\end{equation}
where $\diag(\cdot)$ denotes a diagonal matrix with its arguments on the diagonal.
For simplicity, we also write $A = Q\Lambda Q^T$ 
where $Q=Q_n$ and $\Lambda=\Lambda_n$.
In this paper, we consider $A$ to be large-scale, which usually implies that $A$ is 
sparse.  Since eigenvectors are generally dense, in practical applications, instead 
of computing all $n$ eigenpairs of $A$,  it is only realistic to compute $k \ll n$ 
eigenpairs corresponding to $k$ largest or smallest eigenvalues of $A$.  
Fortunately, these so-called exterior (or extreme) eigenpairs of $A$ often contain 
the most relevant or valuable information about the underlying system or dataset 
represented by the matrix $A$.   As the problem size $n$ becomes ever larger, 
the scalability of algorithms with respect to $k$ has become a critical issue
even though $k$ remains a small portion of $n$.  

Most algorithms for computing a subset of eigenpairs of large matrices are
iterative in which each iteration consists of two main steps: 
a subspace update step and a projection step.  The subspace update step 
varies from method to method but with a common goal in finding a matrix 
$X\in \R^{n\times k}$ so that its column space is a good approximation to 
the $k$-dimensional eigenspace spanned by $k$ desired eigenvectors.
Once $X$ is obtained and orthonormalized, the projection step, 
often referred to as the Rayleigh-Ritz (RR) procedure,
aims to extract from $X$ a set of approximate eigenpairs 
(see more details in Section~\ref{sec:overview})
that are optimal in a sense.
More complete treatments of iterative algorithms for computing subsets 
of eigenpairs can be found, for example, in the books
\cite{Baibook2000,parlett,Saad:1992:PWS,Demmel:1997,Stewart:1998:SIAM}.

At present, the predominant methodology for subspace updating 
is still Krylov subspace methods, as represented by Lanczos 
type methods~\cite{Lanczos:1950,Sorensenetall1997} for real 
symmetric matrices.   These methods generate an orthonormal matrix 
$X$ one (or a few) column at a time in a sequential mode.   
Along the way, each column is multiplies by the matrix $A$ and made 
orthogonal to all the previous columns. In contrast to Krylov subspace 
methods, block methods, as represented by the classic simultaneous 
subspace iteration method~\cite{Rutishauser1969}, carry out the 
multiplications of $A$ to all columns of $X$ at the same time in a batch 
mode.   As such, block methods generally demand a lower level of
communication intensity.

The operation of the sparse matrix $A$ multiplying a vector, 
or SpMV, used to be the most relevant complexity measure for 
algorithm efficiency.   As Krylov subspace methods generally 
tend to require considerably fewer SpMVs than block methods do, 
they had naturally become the methodology of choice for 
the past a few decades even up to date.
However, the evolution of modern computer architectures, 
particularly the emergence of multi/many-core architectures, 
has seriously eroded the relevance of SpMV (and arithmetic
operations in general) as a leading complexity measure, as
communication costs have, gradually but surely, become 
more and more predominant.

The purpose of this work is to construct, analyze and test a 
framework for block algorithms that can efficiently, reliably and 
accurately compute a relatively large number of exterior eigenpairs of 
large-scale matrices.  The algorithm framework is constructed to take
advantages of multi/many-core or parallel computers, although
a study of parallel scalability itself will be left as a future topic.  
It appears widely accepted that a key property hindering the 
competitiveness of block methods is that their convergence 
can become intolerably slow when decay rates in relevant 
eigenvalues are excessively flat.  A central task of our algorithm 
construction is to rectify this issue of slow convergence.

Our framework starts with an outer iteration loop that features
an enhanced RR step called the augmented Rayleigh-Ritz (ARR) 
projection which can provably accelerate convergence under mild 
conditions.  For the SU step, we consider two block iteration schemes 
whose computational cost is dominated by block SpMVs: 
(i) the classic power method applied to multiple vectors without periodic 
orthogonalization, and (ii) a recently proposed Gauss-Newton method.  
For further acceleration, we apply our block SU schemes to a set of 
polynomial accelerators, say $\rho(A)$, aiming to suppress the magnitudes 
of $\rho(\lambda_{j})$ where $\lambda_{j}$'s are the unwanted eigenvalue 
of $A$ for $j > k$.  In addition, a deflation scheme is utilized 
to enhance the algorithm's efficiency.  
Some of these techniques have been studied in the literature over the years 
(e.g. \cite{Saad:1984,chebydav} on polynomial filters), and are relatively well 
understood.  In practice, however, it is still a nontrivial task to integrate all the 
aforementioned components into an efficient and robust eigensolver.  
For example, an effective use of a set of polynomial filters involves the choice 
of polynomial types and degrees, and the estimations of intervals in which
eigenvalues are to be promoted or suppressed.   There are quite a number of
choices to be made and parameters to be chosen that can significantly impact 
algorithm performance.

Specifically, our main contributions are summarized as follows. 
\begin{enumerate}
\item
An augmented Rayleigh-Ritz (ARR) procedure is proposed and analyzed
that provably speeds up convergence without increasing the block size of 
the iterate matrix $X$ in the SU step (thus without increasing the cost of 
SU steps).  This ARR procedure can significantly reduce the number of 
RR projections needed, at the cost of increasing the size of 
a few RR calls.

\item
A versatile and efficient algorithmic framework is constructed that 
can accommodate different block methods for subspace updating.
In particular, we revitalize the power method as an exceptionally 
competitive choice for a high level of concurrency.  Besides ARR, 
our framework features several important components, including
\begin{itemize}
\item
a set of low-degree, non-Chebyshev polynomial accelerators that 
seem less sensitive to erroneous intervals than the classic 
Chebyshev polynomials;
 
\item 
a bold stoping rule for SU steps that demands no periodic 
orthogonalizations and welcomes a (near) loss of numerical rank.
\end{itemize}

\end{enumerate}

With regard to the issue of basis orthogonalization, we recall
that in traditional block methods such as the classic subspace 
iteration, orthogonalization is performed either at every iteration 
or frequently enough to prevent the iterate matrix $X$ from losing 
rank.  On the contrary, our algorithms aim to make $X$ numerically 
rank-deficient right before performing an RR projection.

%\subsection{Organization and Notation}
The rest of this paper is organized as follows. An overview of relevant iterative 
algorithms for eigenpair computation is presented in Section \ref{sec:overview}.
The ARR procedure and our algorithm framework are proposed in
Section \ref{sec:MPM}. We analyze the ARR procedure in Section \ref{sec:ARR}.
The polynomial accelerators used by us are given in Section \ref{sec:poly}.
A detailed pseudocode for our algorithm is outlined in Section \ref{sec:alg}. 
Numerical results are presented in Section \ref{sec:num}.  
Finally, we conclude the paper in Section \ref{sec:con}.

%%%%
\section{Overview of Iterative Algorithms for Eigenpair Computation}
\label{sec:overview}

Algorithms for eigenvalue problem have been extensively studied for 
decades.  We will only briefly review a small subset of them that are
most closely related to the present work.

%Recall that by our notation the eigen-decomposition of a symmetric
%$A \in\R^{n\times n}$ is $A = Q_n\Lambda_n Q_n^{T}$, where 
%$Q_n$ and $\Lambda_n$ are defined in \eqref{eq:Q}.
Without loss of generality,  we assume for convenience
that $A$ is positive definite (after a shift if necessary).
Our task is to compute $k$ largest eigenpairs $(Q_k, \Lambda_k)$ 
for some $k\ll n$ where by definition  
$A Q_k = Q_k \Lambda_k$ and $Q_k^T Q_k = I\in\R^{k\times k}$. 
Replacing $A$ by a suitable function of $A$, say $\lambda_{1}I-A$, 
one can also in principle apply the same algorithms to finding $k$ 
smallest eigenpairs as well.

An RR step is to extract approximate eigenpairs, called Ritz-pairs, 
from a given matrix $Z\in\R^{n\times m}$ whose range space, 
$\cR(Z)$, is supposedly an approximation to a desired 
$m$-dimensional eigenspace of $A$.  Let $\orth(Z)$ be the 
set of orthonormal bases for the range space of $Z$.  The RR procedure is 
described as Algorithm~\ref{RR} below, which is also denoted by a map 
$(Y, \Sigma)=\RR(A,Z)$ where the output $(Y, \Sigma)$ is a Ritz pair block.
\begin{algorithm2e}[htb!]
\label{RR}
%\footnotesize
\caption{ Rayleigh-Ritz procedure: $(Y, \Sigma)=\RR(A,Z)$}
\SetKwInOut{Input}{input}\SetKwInOut{Output}{output}
\SetKwComment{Comment}{}{}
\BlankLine
Given $Z\in\R^{n\times m}$, orthonormalize $Z$ to obtain $U \in\orth(Z)$.\\
Compute $H = U^{T}AU\in\R^{m\times m}$, the projection of $A$ 
onto $\orth(Z)$.\\ 
Compute the eigen-decomposition $H=V^T \Sigma V$,
where $V^{T}V=I$ and $\Sigma$ is diagonal.\\
Assemble the Ritz pairs $(Y,\Sigma)$ where $Y = UV\in\R^{n\times m}$ 
satisfies $Y^{T}Y=I$.
\end{algorithm2e}

It is known (see \cite{parlett}, for example) that Ritz pairs are, in a certain sense, 
optimal approximations to eigenpairs in $\cR(Z)$, the column space of $Z$.

%%%%%%%
\subsection{Krylov Subspace Methods}

Krylov subspaces are the foundation of several state-of-the-art solvers for
large-scale eigenvalue calculations. By definition, for given matrix 
$A \in \R^{n\times n}$ and vector $v \in \R^n$, the Krylov subspace of 
order $k$ is $\spa\{v, Av, A^2v, \ldots, A^{k-1}v\}$.
Typical Krylov subspace methods include  Arnoldi algorithm
for general matrices (e.g., \cite{Sorensenetall1997,Lehoucq2001}) and 
Lanczos algorithm for symmetric (or Hermitian) matrices (e.g.,
\cite{Sorensen1996,Larsen1998}).  In either algorithm,  orthonormal bases 
for Krylov subspaces are generated through a Gram-Schmidt type process.
Jacobi-Davidson methods (e.g., \cite{Notay2007,Stath1994}) are based on 
a different framework, but they too rely on Krylov subspace methodologies 
to solve linear systems at every iteration.

As is mentioned in the introduction, Krylov-subspace type methods are 
generally most efficient in terms of the number of SpMVs (sparse 
matrix-dense vector multiplications).  Indeed, they remain the method of 
choice for computing a small number eigenpairs.  However, due to the 
sequential process of generating orthonormal bases, Krylov-subspace 
type methods incur a low degree of concurrency, especially as the dimension 
$k$ becomes relatively large. To improve concurrency, multiple-vector 
versions of these algorithms have been developed where each single 
vector in matrix-vector multiplication is replaced by a small number of 
multiple vectors.  Nevertheless, such a remedy can only provide a 
limited relief in the face of the inherent scalability barrier as $k$ grows.
%Nevertheless, the computational bottlenecks are (i) the construction and
%maintenance of orthonormal bases for Krylov subspaces and (ii) the repeated 
%RR steps to compute approximate eigenpairs, since BLAS2 may account for a 
%large proportion of computational cost and the amount of parallelism is still 
%fundamentally limited.
Another well-known limitation of Krylov subspace methods is the difficulty to 
warm-start them from a given subspace.  Warm-starting is important in an 
iterative setting in order to take advantages of available information 
computed at previous iterations.

%%%%%%
\subsection{Classic Subspace Iteration}

The simple (or simultaneous) subspace iteration (SSI) method (see
\cite{Rutishauser1969,Rutishauser1970,Stewart1976,StewartJennings1981},
for example) %\cite{Saad:1992:PWS} 
extends the idea of the power method which computes a single eigenpair
corresponding to the largest eigenvalue (in magnitude).  Starting from 
an initial (random) matrix $U$, SSI performs repeated matrix multiplications 
$AU$, followed by periodic orthogonalizations and RR projections.  
The main purpose of orthogonalization is to prevent the iterate matrix $U$
from losing rank numerically.  In addition, since the rates of convergence
for different eigenpairs are uneven, numerically converged eigenvectors 
can be deflated after each RR projection.  A version of SSI algorithm
is presented as Algorithm \ref{SI-2} below, following the description in
\cite{Stewart:1998:SIAM}.

\begin{algorithm2e}[htb!]
\label{SI-2}
%\small
\caption{Subspace Iteration}% (more elaborate version)
\SetKwInOut{Input}{input}\SetKwInOut{Output}{output}
\SetKwComment{Comment}{}{}
\BlankLine
Initialize orthonormal matrix $U \in \R^{n \times m}$ with $m = k + q \ge k$.\\
\While{the number of converged eigenpairs is less than $k$,}{
\While{convergence is not expected,}{
\While{the columns of $U$ are sufficiently independent,}{
Compute $U = AU$}
Orthogonalize the columns of $U$.}
Perform an RR step using $U$. \\
Check convergence and deflate.
}
\end{algorithm2e}

In the above SSI framework, $q$ extra vectors, often called guard vectors,
are added into iterations to help improve convergence at the price of 
increasing the iteration cost.

A  main advantage of SSI is the use of simultaneous matrix-block multiplications
instead of individual matrix-vector multiplications. It enables fast
memory access and highly parallelizable computation on modern computer
architectures. Furthermore, SSI method has a guaranteed convergence to the 
largest $k$ eigenpairs from any generic starting point as long as there is a 
gap between the $k$-th and the $(k+1)$-th eigenvalues of $A$.  
As is points out in \cite{Stewart:1998:SIAM}, 
``combined with shift-and-invert enhancement or Chebyshev 
acceleration, it sometimes wins the race''.
However, a severe shortcoming of the SSI method is that its convergence 
speed depends critically on eigenvalue distributions that can, and often does, 
become intolerably slow in the face of unfavorable eigenvalue distributions. 
Thus far, this drawback has essentially prevented the SSI method from 
being used as a computational engine to build robust, reliable and 
efficient general-purpose eigensolvers. 

\subsection{Trace Maximization Methods}\label{sec:tracemax}

Computing a $k$-dimensional eigenspace associated with $k$
largest eigenvalues of $A$ is equivalent to solving an orthogonality 
constrained trace maximization problem:
\begin{equation}\label{tracemax}
\max_{X\in \R^{n \times k}} \tr(X^TAX), ~\st~ X^TX = I.
\end{equation}
%The first-order necessary conditions of \eqref{tracemax} for optimality are
%$AX = X\Lambda$ and $X^TX = I$,
%where $\Lambda=X^TAX \in \R^{k\times k}$ is the matrix of Lagrangian 
%multiplier.  Once the $k\times k$ matrix $\Lambda$ is diagonalized, the 
%matrix pair $(\Lambda,X)$ will deliver $k$ eigenpairs of $A$.
This formulation can be easily extended to solving the {\em generalized eigenvalue problem} where $X^{T}X=I$ is replace by $X^TBX=I$ for a symmetric positive definite matrix $B \in \R^{n\times n}$.  When maximization is changed to minimization, one computes an eigenspace associated with $k$ smallest eigenvalues.  The algorithm TraceMin~\cite{tracemin} solves
the trace minimization problem using a Newton type method.

Some block algorithms have been developed based on solving \eqref{tracemax}, 
include the locally optimal block preconditioned conjugate gradient method 
(LOBPCG)~\cite{LOBPCG} and more recently the limited memory block Krylov 
subspace optimization method (LMSVD)~\cite{LMSVD}.   At each iteration, 
these methods solve a subspace trace maximization problem of the form
\be \label{eq:sub2}
 Y = \argmax_{X \in \R^{n\times k}}
 \left\{\tr(X^T AX):  X\zz X = I, \; X \in \cS\right\},
\ee
where $X \in\cS$ means that each column of $X$ is in the
given subspace $\cS$ which varies from method to method.  
LOBPCG constructs $\cS$ as the span of the two most recent 
iterates $X^{(i-1)}$ and $X^{(i)}$, and the residual at $X^{(i)}$, 
which is essentially equivalent to
\begin{equation}\label{S:lobpcg}
 \cS = \spa \left\{X^{(i-1)},X^{(i)},A X^{(i)}\right\},
\end{equation}
where the term $AX^{(i)}$ may be pre-multiplied by a pre-conditioning matrix. 
In the LMSVD method, on the other hand, the subspace $\cS$ is spanned by 
the current $i$-th iterate and the previous $p$ iterates; i.e.,
\begin{equation}\label{S:lmsvd}
\cS = \spa \left\{X^{(i)},X^{(i-1)},...,X^{(i-p)} \right\},
\end{equation}
In general, the subspace $\cS$ should be constructed such that the cost of
solving \eqref{eq:sub2} can be kept relatively low.
The parallel scalability of these algorithms, although improved from that of Krylov
subspace methods, is now limited by the frequent use of basis orthogonalizations 
and RR projections involving $m\times m$ matrices where $m$ is the dimension
of the subspace $\cS$ (for example, $m=3k$ in LOBPCG). 

%%%%%%
\subsection{Polynomial Acceleration}

Polynomial filtering has been used in eigenvalue computation in various ways
(see, for example, \cite{Saad:1984,Stewart:1998:SIAM,chebydav,FangSaad2012}).
For a polynomial function $\rho(t): \R \to \R$ and a symmetric matrix with
eigenvalue decomposition $A=Q\Lambda Q^T$, it holds that 
\begin{equation}\label{eq:poly-A}
  \rho(A) = Q \rho(\Lambda) Q^{T} = \sum_{i=1}^n \rho(\lambda_i) q_i q_i^T,
\end{equation}
where $\rho(\Lambda)= \diag(\rho(\lambda_1),\rho(\lambda_2), \ldots, \rho(\lambda_n))$.
By choosing a suitable polynomial function $\rho(t)$ and replacing $A$ by $\rho(A)$, 
we can change the original eigenvalue distribution into a more favorable one at 
a cost.  To illustrate the idea of polynomial filtering,
suppose that $\rho(t)$ is a good approximation to the step function that is
one on the interval $[\lambda_k,\lambda_1]$ and zero otherwise.
%\[\psi(t) = \begin{cases}1, & t \in [\lambda_k,\lambda_1],\\
%   0, &\mbox{ otherwise. }
% \end{cases} 
%\]
For a generic initial matrix $X \in \R^{n\times k}$, it follows from \eqref{eq:poly-A} 
that $\rho(A)X \approx Q_kQ_k^T X$, which would be an approximate basis for
the desired eigenspace. In practice, however, approximating a non-smooth step 
function by polynomials is an intricate and demanding task which does not always
lead to efficient algorithms.

For the purpose of convergence acceleration, the most often used polynomials 
are the Chebyshev polynomials (of the first kind), defined by the three-term recursion: 
\begin{equation}\label{eq:cheby-poly} 
  \rho_{d+1}(t)=2t  \rho_{d}(t)-\rho_{d-1}(t), \;\; d \ge 1,
\end{equation}
where $\rho_0(t)=1$ and $\rho_1(t)=t$.    Some recent works that use Chebyshev 
polynomials include \cite{chebydav,FangSaad2012}, for example.

%%%%%%
\subsection{FEAST}

The \feast~ algorithm~\cite{Eric2009,PingEric2014} is based on complex
contour integrals for computing all eigenvalues in a given interval
$[a,b] \subset \R$ and their corresponding eigenvectors.  It is equivalent
to using a rational function filter in subspace iteration.

Let $\cC$ be the circle on the complex plane centered at $c=\frac{a+b}{2}$ 
with radius $r=\frac{b-a}{2}$, which can be parameterized by the function
$\phi(t) = c + r e^{\iota\frac{\pi}{2}(1+t)}$ for $t \in [-1,3]$ where $\iota^2 = -1$
is the imaginary unit.  By the Cauchy integral theorem, for any $\mu\notin \cC$
%the contour integral (in the counterclockwise direction) satisfies 
\begin{equation*}
  \frac{1}{2\pi \iota} \oint_{\cC} \frac{1}{z-\mu} dz 
 =  \frac{1}{2\pi \iota} \int_{-1}^1 \left[ \frac{
  \phi'(t)}{\phi(t)-\mu} - \frac{
  \overline{\phi'(t)}}{\overline{\phi(t)}-\mu}\right] dt 
 = \begin{cases} 1, & \mbox{ if } |\mu- c|<r \\
  0, &   \mbox{ if } |\mu-c|> r
  \end{cases}, 
\end{equation*}
where the integral on $[1,3]$ has been equivalently transformed into $[-1,1]$.
Applying a $q$-point Gauss-Legendre quadrature formula with weight-node pairs
$(w_l,t_l)$, $l=1,2,\ldots,q$, such that $w_l > 0$ and $t_l \in (-1,1)$, the above
integral can be approximated by the rational function
\begin{equation*}
\rho(\mu)=\sum_{l=1}^q \left( \frac{\sigma_l}{\phi_l-\mu} -
  \frac{\overline{\sigma_l}}{\overline{\phi_l}-\mu}  \right),
\end{equation*}
where $\phi_l=\phi(t_l)$ and $\sigma_l=w_l\phi'(t_l)/(2\pi\iota)$. 
Since none of $\phi_l$'s is real and $A$ is symmetric,
the matrices $\phi_l I - A$ and $\overline{\phi_l} I - A$ are all
invertible for $l=1,2,\ldots,q$.  Therefore, 
\begin{equation}\label{eq:feast-rho}
	\rho(A)=\sum_{l=1}^q  \sigma_l (\phi_l I - A)^{-1} - \sum_{l=1}^q
	\overline{\sigma_l} (\overline{\phi_l} I - A)^{-1}
\end{equation}
is a rational function filter approximating a desired step function on the real line. 
The application of this filter to $X\in \R^{n\times m}$, i.e., computing $\rho(A)X$, 
will require solving $q$ (since all quantities involved are real) 
linear systems of equations with $m$ right-hand sides each.
It is notable that these linear systems could be solved independently in parallel.

In order to compute all eigenpairs in an interval $[a,b]$, \feast~ need to
estimate the number of eigenvalues in the interval $[a,b]$. It repeatedly 
applies the rational filter $X = \rho(A)X$, followed by an RR projection. 
A high-level summary of the \feast~ algorithm is presented as Algorithm
 \ref{FEAST}.  

\begin{algorithm2e}[htb!]
\label{FEAST}
%\small
\caption{A abstract version of \feast}
\SetKwInOut{Input}{input}\SetKwInOut{Output}{output}
\SetKwComment{Comment}{}{}
\BlankLine
Input $[a, b]$ and $m$ -- estimated number of eigenvalues in $[a,b]$. \\ 
Choose a Gauss-Legendre quadrature formula with $q$ nodes. \\
Initialize a matrix $X \in \R^{n \times m}$.\\
\While{not ``converged'',}{
Compute $X = \rho(A)X$ with $\rho(\cdot)$ given in \eqref{eq:feast-rho}. \\
Do RR projection using $X$ to extract Ritz pairs. 
}
\end{algorithm2e}

It should be clear that the performance of \feast~ depends strongly on the 
efficiency of solving the linear systems of equations involved in applying
the rational filter $\rho(A)$ to $X$.   In addition, in order to compute the $k$ 
largest eigenpairs, for example, one need to supply \feast~ with an interval 
$[a,b] \supseteq [\lambda_{k},\lambda_{1}]$.  The quality of this interval $[a,b]$
could have a significant effect on the performance of \feast.

%%%%%%
\subsection{A Gauss-Newton Algorithm}\label{sec:GN}
A Gauss-Newton (GN) algorithm is recently proposed in 
\cite{SLRPGN} to compute the eigenspace associated with 
$k$ largest eigenvalues of $A$ based on solving the
nonlinear least squares problem:
$%\begin{equation} \label{NLS}
\min\|XX\zz - A\|\fs,
$%\end{equation}
where $X \in \R^{n \times k}$, $\|\cdot\|\fs$ is the Frobenius norm squared 
and $A$ is assumed to have at least $k$ positive eigenvalues.  
If the eigenpairs of $A$ are required, then an RR projection must 
be performed afterwards.

It is shown in \cite{SLRPGN} that at any full-rank iterate 
$X \in \R^{n \times k}$, the \gn~ method takes the simple closed form
\begin{equation*}
X^+ = X+ \alpha  \left(I - \frac{1}{2}X(X\zz X)^{-1}X\zz \right)\left(AX(X\zz X)^{-1}-X\right),
\end{equation*}
where the parameter $\alpha > 0$ is a step size.  Notably, this method 
requires to solve a small $k\times k$ linear system at each iteration.  
It is also shown in \cite{SLRPGN} that the fixed step $\alpha \equiv 1$ is justifiable 
from either a theoretical or an empirical viewpoint, which leads to a parameter-free 
algorithm given as Algorithm~\ref{alg:GN}, named simply as \gn. 
For more theoretical and numerical results on this \gn~ algorithm, we
refer readers to \cite{SLRPGN}.

\begin{algorithm2e}[ht]
\caption{A GN Algorithm: $X = \gn(A,X)$}
\label{alg:GN}
\SetKwInOut{Input}{input}\SetKwInOut{Output}{output}
\SetKwComment{Comment}{}{}
\BlankLine %\dontprintsemicolon
%Input $A\in\R^{m\times n}$ and $k<\min(m,n)$. \\
Initialize $X\in\R^{n\times k}$ to a rank-$k$ matrix.\\
\While{``the termination criterion'' is not met,}{ 
Compute $Y =  X\left(X^T X\right)^{-1}$ and $Z = AY$. \\
Compute $X = Z - X(Y\zz Z - I)/2$.%\\
%Increment $i$ and continue.
}
Perform an RR step using $X$ if Ritz-pairs are needed. 
\end{algorithm2e}

\section{Augmented Rayleigh-Ritz Projection and Our Algorithm Framework}
\label{sec:MPM}

We first introduce the augmented Rayleigh-Ritz or ARR procedure.  
It is easy to see that the RR map $(Y, \Sigma) = \RR(A,Z)$ is equivalent to 
solving the trace-maximization subproblem \eqref{eq:sub2} with the subspace
$\cS = \cR(Z)$, while requiring $Y\zz AY$ to be a diagonal matrix $\Sigma$. 
For a fixed number $k$, the larger the subspace $\cR(Z)$ is, the 
greater chance there is to extract better Ritz pairs.  
The classic SSI always sets $Z$ to the current iterate $X^{(i)}$, 
while both LOBPCG~\cite{LOBPCG} and LMSVD~\cite{LMSVD} 
augment $X^{(i)}$ by additional blocks (see \eqref{S:lobpcg} and 
\eqref{S:lmsvd}, respectively).  Not surprisingly, such augmentations 
are the main reason why algorithms like LOGPCG and LMSVD
generally achieve faster convergence than that of the classic SSI.

In this work, we define our augmentation based on a block Krylov 
subspace structure.  That is, for some integer $p\ge 0$ we define
\begin{equation} \label{S:ARR} 
\cS = \spa\{X, AX, A^2X, \ldots, A^p X\}.% = \cK_{p}(A,X).
\end{equation}
This choice \eqref{S:ARR} of augmentation is made mainly because 
it enables us to conveniently analyze the acceleration rates induced
by such an augmentation (see the next Section).  It is more than likely 
that some other choices of $\cS$ may be equally effective as well.

The optimal solution of the trace maximization problem \eqref{eq:sub2},
restricted in the subspace $\cS$ in \eqref{S:ARR}, can be computed via 
the RR procedure, i.e., Algorithm~\ref{RR}.   We formalize our augmented 
RR procedure as Algorithm~\ref{alg:ARR}, which will often be referred to 
simply as ARR.  

\begin{algorithm2e}[htb!]
\label{alg:ARR}
%\small
\caption{ARR: $(Y,\Sigma) = \ARR(A,X,p)$}
\SetKwInOut{Input}{input}\SetKwInOut{Output}{output}
\SetKwComment{Comment}{}{}
\BlankLine
Input $X \in \R^{n \times k}$ and $p \ge 0$ so that $(p+1)k < n$. \\
Construct augmentation $\Xp=[X\,\; AX\,\; A^2X\,\; \cdots\,\; A^p X]$.\\ 
Perform an RR step using $(\hat{Y},\hat\Sigma)=\RR(A,\Xp)$.\\
Extract $k$ leading Ritz pairs $(Y,\Sigma)$ from $(\hat{Y},\hat\Sigma)$. 
\end{algorithm2e}

We next introduce an abstract version of our algorithmic framework with
ARR projections.  It will be named \arrabit~(standing for ARR and block iteration).
A set of polynomial functions $\{\rho_d(t)\}$, where $d$ is the polynomial degree,
and an integer $p \ge 0$ are chosen at the beginning of the algorithm.  
At each outer iteration, we perform the two main steps: subspace update (SU) 
step and augmented RR (ARR) step.   There are two sets of stopping criteria:
inner criteria for the SU step, and outer criteria for detecting the convergence
of the whole process.

In principle, the SU step can be fulfilled by any reasonable updating scheme 
and it does not require orthogonalizations. In this paper, we consider the 
classic power iteration as our main updating scheme, i.e., 
for $X = [x_1\; x_2\; \cdots, x_m] \in \R^{n\times m}$, we do
\begin{equation*}
 x_i = \rho(A)x_i 
 \quad\mbox{ ~and~ }\quad 
 x_i = \frac{x_i}{\|x_i\|_2},
 \;\; j = 1, 2, \cdots, m.
\end{equation*}
Since the power iteration is applied individually to all columns of the iterate 
matrix $X$, we call this scheme
{\em multi-power method} or \mpm.  Here we intensionally avoid to use the 
term {\em subspace iteration} because, unlike in the classic SSI, we do not 
perform any orthogonalization during the entire inner iteration process.  

To examine the versatility of the \arrabit~framework, we also use the 
Gauss-Newton (\gn) method, presented in Algorithm~\ref{alg:GN},  as a 
second updating scheme.  Since the \gn~ variant requires solving $k \times k$ 
linear systems, its scalability with respect to $k$ may be somewhat lower 
than that of the \mpm~ variant.  Together, we present our \arrabit~algorithmic 
framework in Algorithm~\ref{alg:Arrabit1}.  The two variants, corresponding to
``inner solvers'' \mpm~ and \gn,  will be named \arrabit-\mpm~ and \arrabit-\gn, 
or simply \mpm~ and \gn.

\begin{algorithm2e}[htb!]
\label{alg:Arrabit1}
%\small
\caption{Algorithm \arrabit~(abstract version)}
\SetKwInOut{Input}{input}\SetKwInOut{Output}{output}
\SetKwComment{Comment}{}{}
\BlankLine
Input $A \in\R^{n\times n}$, $k$, $p$ and $\rho(t)$.
Initialize $X \in \R^{n \times k}$. \\
\While{not ``converged'',}{
\While{``inner criteria'' are not met,}{ 
\If{\mpm~ is the inner solver,}{
$X = \rho(A)X$, then normalize columns individually.}
\If{\gn~ is the inner solver,}{
$X =  \gn(\rho(A),X)$, as is given by Algorithm~\ref{alg:GN}.}
}
ARR projection: $(X,\Sigma) = \ARR(A,X,p)$, as in Algorithm~\ref{alg:ARR}.\\
Possibly adjust $p$, the degree of $\rho(t)$, and perform deflation.
}
\end{algorithm2e}

It is worth mentioning that the ``inner criteria'' in the \arrabit~framework 
can have a significant impact on the efficiency of Algorithm~\ref{alg:Arrabit1}.  
Against the conventional wisdom, we do not attempt to keep $X$ numerically
full rank by periodic orthogonalizations which can be quite costly.  Instead, 
we keep iterating until we detect that $X$ is about to lose, or has just lost, 
numerical rank.  More details on this issue will be given in 
Algorithm \ref{alg:Arrabit2} in Section \ref{sec:alg}.

%%%%%%%%%%%%%%%%%%%%%%%%%%%%%%
\section{Analysis of the Augmented Rayleigh-Ritz Procedure}
\label{sec:ARR}

%%%%%%%%%%%%%%%%%%%%%
\subsection{Notation}

Recall that the eigen-decomposition of $A \in\R^{n\times n}$ is 
$A = Q\Lambda Q^{T}$.
In anticipation of later usage, for integer $h\in[1,n)$ 
we introduce the partition $Q = [Q_h \; Q_{h+}]$ where, as
previously defined, $Q_h=[q_{1}\;\; q_{2}\;\; \cdots q_{h}]$ and
\begin{equation}\label{Qh+}
Q_{h+}=[q_{h+1}\;\; q_{h+2}\; \cdots\; q_{n}].
\end{equation}

Let $X \in \R^{n\times k}$ be an approximate basis for $\cR(Q_k)$, 
the range space of $Q_{k}$ or the eigenspace spanned by the first $k$ 
eigenvectors of $A$.  It is desirable for $X$ to have a large projection
$Q_kQ_k^{T}X=\sum_{i=1}^{k}q_iq_i^{T}X$ onto $\cR(Q_k)$ 
relative to that onto $\cR(Q_{k+})$.
Therefore, a good measure for the relative accuracy of $X$ 
is the following ratio
\begin{equation}\label{def:deltak}
	\delta_{k}(X) \triangleq
\frac{\max_{i>k}\|q_{i}^{T}X\|}{\min_{i\le k}\|q_{i}^{T}X\|},
\end{equation}
where $\|q_{i}^{T}X\| = \|(q_{i}q_{i}^{T})X\|$ measures the size of 
the projection of $X$ onto the span of the $i$-th eigenvector $q_{i}$.
Clearly, the smaller $\delta_{k}(X)$ is, the better is $X$ as an 
approximate basis for $\cR(Q_k)$.

Let $Y \in \R^{n\times k}$ be another approximate basis for the 
eigenspace $\cR(Q_k)$ which is constructed from $X$. To 
compare $Y$ with $X$, we naturally compare $\delta_{k}(Y)$ 
with $\delta_{k}(X)$.   More precisely, we will try to estimate the
ratio ${\delta_{k}(Y)}/{\delta_{k}(X)}$ and show that under 
reasonable conditions, it can be made much less than the unity.

To facilitate presentation, we introduce the following
Vandermonte matrix constructed from the spectrum of $A$:
\begin{equation}\label{def:V}
V = \left(\begin{array}{ccccc}
1 & \lambda_{1} & \lambda_{1}^{2} & \cdots & \lambda_{1}^{p}\\
1 & \lambda_{2} & \lambda_{2}^{2} & \cdots & \lambda_{2}^{p}\\
\vdots & \vdots & \vdots & \vdots & \vdots \\
1 & \lambda_{n} & \lambda_{n}^{2} & \cdots & \lambda_{n}^{p}
\end{array}\right) \in\R^{n\times (p+1)},
\end{equation}
where $\lambda_{1}, \cdots, \lambda_{n}$ are the eigenvalues of $A$.

%%%%%%%%%%%%%%%%%%%%%
\subsection{Technical Results}

Before calling the ARR procedure, we have an iterate matrix 
$X \in\R^{n\times k}$.   From $X$, we construct the augmented 
matrix $[X \;\; AX\;\; \cdots \;\; A^{p}X] \in \R^{n \times (p+1)k}$ 
which we call $\Xp$ for a given $p\ge 0$.
In view of the eigen-decomposition $A=Q\Lambda Q^{T}$,  
we have the expression $\Xp=Q\hat{G}$
%\begin{equation*}\label{eqn:Xp=QGh}
%  \Xp = Q[Q^{T}X \;\; \Lambda Q^{T}X \;\; \cdots \;\; \Lambda^{p}Q^{T}X] 
%  = Q\hat{G},
%\end{equation*}
where 
\begin{equation}\label{def:Gh}
 \hat{G} = [Q^{T}X \;\; \Lambda Q^{T}X \;\; \cdots \;\; \Lambda^{p}Q^{T}X].
\end{equation}
We next normalize the rows of $\hat{G}$.  Let $D$ be the diagonal matrix whose
diagonal consists of the row norms of $\hat{G}$.   From
the structure of $\hat{G}$ in \eqref{def:Gh}, it is easy to see that
\begin{equation}\label{def:Dii}
	D_{ii} = \|e_{i}^{T}\hat{G}\| = \|q_{i}^{T}X\| \|e_{i}^{T}V\|, 
	\;\; i = 1, 2, \cdots, n,
\end{equation}
where $e_{i}$ is the $i$-th column of the $n \times n$ identity matrix and 
$V$ is defined in \eqref{def:V}.
Let $D^{\dagger}$ be the pseudo-inverse of $D$, that is,  $D^{\dagger}$
is a diagonal matrix with
\begin{equation}\label{def:D+ii}
	(D^{\dagger})_{ii} = \left\{\begin{array}{cc} 
	1/D_{ii}, & \mbox{ if } D_{ii} \ne 0, \\ 
	0, & \mbox{ otherwise}.
	\end{array}\right.
\end{equation}
The normalization of the rows of $\hat{G}$ in \eqref{def:Gh} defines
another matrix
\begin{equation}\label{def:G}
	G = D^{\dagger}\hat{G} = 
	[C \;\; \Lambda C \;\; \cdots \;\; \Lambda^{p}C],
\end{equation}
where $C=D^{\dagger}Q^{T}X$ and the nonzero rows of $G$ all have unit norm.
Now we rewrite
\begin{equation}\label{eq:Xp}
  \Xp = QDD^{\dagger}\hat{G}  = QDG.
\end{equation}

Let $m$ be a parameter varying in the following range:
for $p \ge 0$ such that $k+pk < n$,
\begin{equation}
	m \in [k,k+pk].
\end{equation}
We perform the partition 
\begin{equation}\label{eq:Xp2}
	\Xp =  [Q_{m} \; Q_{m+}]
 \left[\begin{array}{cc} D_{1} & 0 \\0 & D_{2}\end{array}\right]
 \left[\begin{array}{c} G_{1} \\ G_{2}\end{array}\right] = [Q_{m} \; Q_{m+}]
 \left[\begin{array}{c} D_{1}G_{1} \\ D_{2}G_{2}\end{array}\right],
\end{equation}
where $D$ and $G$ are partitioned following that of $Q$.  In particular,
$G_{1}$ consists of the first $m$ rows of $G$ and $G_{2}$ the
last $n-m$ rows of $G$.

In the sequel, we will make use of an important assumption on 
$G_{1} \in \R^{m\times (p+1)k}$ which we formally name as the 
$G_{1}$-Assumption:
\begin{equation}\label{G1-Assumption}
\mbox{$G_{1}$-Assumption: the first $m$ rows of 
$G$ (or $\hat{G}$) are linearly independent}.
\end{equation}
The $G_{1}$-Assumption implies that
(i) $D_{1} > 0$, and (ii) the pseudo-inverse $G_{1}^{\dagger}$ exists 
such that $G_{1}G_{1}^{\dagger}=I_{m\times m}$.  Let
\begin{equation}\label{def:Yp}
\Yp = \Xp G_{1}^{\dagger}D_{1}^{-1} =  [Q_{m} \; Q_{m+}]
 \left[\begin{array}{c} I \\ D_{2}G_{2}G_{1}^{\dagger}D^{-1}\end{array}\right].
 \end{equation}
In particular, we are interested in the first $k$ columns of $\Yp$, i.e., 
by Matlab notation,
\begin{equation}\label{def:Y}
	Y = \Yp(: , 1\!:\!k)  \in \R^{n\times k}.
\end{equation}
We summarize what we already have for $Y$ into the following lemma.
\begin{lemma}\label{lem:Y=}
Let $A = Q\Lambda Q^{T}$ be the eigen-decomposition of 
$A = A^{T} \in \R^{n\times n}$.   For integers $k>0$ and $p \ge 0$ 
satisfying $(p+1)k < n$, and $m\in [k,k+pk]$, let $G$, $\Xp$, 
$\Yp$ and $Y$ be defined as in \eqref{def:G}, \eqref{eq:Xp}, 
\eqref{def:Yp} and \eqref{def:Y}, respectively.  Under the
$G_{1}$-Assumption,
\begin{equation}\label{eqn:Y}
Y = Q_{m}E_{k} + Q_{m+}SE_{k},
\end{equation}
where $S=D_{2}G_{2}G_{1}^{\dagger}D_{1}^{-1}$ and
$E_{k} \in\R^{m\times k}$ consists of the first $k$ columns 
of the $m\times m$ identity matrix. 
\end{lemma}
\begin{proof}
The equality directly follows from \eqref{def:Yp} and \eqref{def:Y}.
\end{proof}

Since $Y$ is extracted from the subspace $\cR(\Xp)$ constructed 
from $X$, a central question is how much improvement $Y$ can 
provide over $X$ as an approximate basis for $\cR(Q_k)$.  
We study this question by comparing the accuracy measure 
$\delta_{k}(Y)$ relative to $\delta_{k}(X)$. First, we estimate
$\delta_{k}(Y)$.

%%%%%%%%%%%%%%%
\begin{lemma}\label{lem:esti1}
%%%%%%%%%%%%%%%
Under the conditions of Lemma~\ref{lem:Y=},
\begin{equation}\label{ineq:deltaY}
\delta_{k}(Y) \le
\frac{\max_{i>m}d_{i}}{\min_{i\le k}d_{i}}
\max_{1 \le i \le n-m}\|e_{i}^{T}G_{2}G_{1}^{\dagger}E_{k}\|.
\end{equation}
where $d=\diag(D)$ with $D_{ii}$ defined in \eqref{def:Dii}.
\end{lemma}
\begin{proof}
It follows from \eqref{eqn:Y} that
\begin{eqnarray*}
q_{i}^{T}Y = \left\{\begin{array}{cc}
e_{i}^{T}, & i \in [1,k] \\
{\bf 0}^{T}, & i \in (k,m] \\
e_{i-m}^{T}SE_{k}, & i \in (m, n] 
\end{array}\right.  
\end{eqnarray*}
where $e_{i}\in\R^{k}$, ${\bf 0}\in\R^{k}$ and $e_{i-m}\in\R^{n-m}$.
These formulas imply that in the definition \eqref{def:deltak} 
the denominator term $\min_{i\le k}\|q_{i}^{T}Y\|=1$; thus
\begin{equation}\label{eqn:del(Y)}
\delta_{k}(Y) = \max_{i>k}\|q_{i}^{T}Y\| = \max_{i>m}\|q_{i}^{T}Y\|.
\end{equation}
In view of the formula $S=D_{2}G_{2}G_{1}^{\dagger}D_{1}^{-1}$, 
and the definition of $D$ in \eqref{def:Dii}, we have
\begin{eqnarray*}
q_{i}^{T}Y = d_{i}e_{i-m}^{T}G_{2}G_{1}^{\dagger}D_{1}^{-1}E_{k},
\;\;  i \in (m,n].
\end{eqnarray*}
Therefore, for $i \in (m,n]$, 
$%\begin{eqnarray*}
\|q_{i}^{T}Y\| \le \frac{d_{i}}{\min_{j\le k}d_{j}}
\|e_{i-m}^{T}G_{2}G_{1}^{\dagger}E_{k}\|
$. %\end{eqnarray*}
 It follows that
\begin{eqnarray*}
\max_{i>m}\|q_{i}^{T}Y\| \le \frac{\max_{i>m}d_{i}}{\min_{i\le k}d_{i}}
\max_{1 \le i \le n-m}\|e_{i}^{T}G_{2}G_{1}^{\dagger}E_{k}\|,
\end{eqnarray*}
which, together with \eqref{eqn:del(Y)}, establishes \eqref{ineq:deltaY}.
\end{proof}

%%%%%%%%%%%%%%
\subsection{Main Results}

We first extend the definition \eqref{def:deltak} for $\delta_{k}(X)$ 
into a more general form.   For any matrix $M$ of $n$ rows, we define
\begin{equation}\label{def:rho}
\Gamma_{k,m}(M) \triangleq \frac 
{\max_{i>m}\|e_{i}^{T}M\|}
{\min_{i\le k}\|e_{i}^{T}M\|}.
\end{equation}
By this definition, $\delta_{k}(X)=\Gamma_{k,k}(Q^{T}X)$. 

It is worth observing that 
(i) $\Gamma_{k,m}(M)$ is monotonically non-increasing with respect 
to $m$ for fixed $k$ and $M$; 
(ii) 
$\Gamma_{k,m}(M)$ is small if the first $k$ rows of $M$ are much 
larger in magnitude than the last $n-m$; 
(iii)
if $\{\|e_{i}^{T}M\|\}$ is non-increasing, then $\Gamma_{k,m}(M) \le 1$.

Specifically, since the eigenvalues of $A$ are ordered in a 
descending order, for the matrix $V$ in \eqref{def:V} we have
\begin{equation}\label{eqn:GamV}
\Gamma_{k,m}(V) = \frac{\|e_{m+1}^{T}V\|}{\|e_{k}^{T}V\|}
= \left(\frac{1+\lambda_{m+1}^{2}+\cdots+\lambda_{m+1}^{2p}}
{1+\lambda_{k}^{2}+\cdots+\lambda_{k}^{2p}}\right)^{\frac{1}{2}} \le 1,
\end{equation}
Evidently, the faster the decay is between $\lambda_{k}$ and 
$\lambda_{m+1}$, the smaller is $\Gamma_{k,m}(V)$.

Moreover, when $M=z\in\R^{n}$ is a vector which is in turn the 
element-wise multiplication of two other vectors, say $x \in\R^{n}$ 
and $y \in\R^{n}$ so that $z_{i}=x_{i}y_{i}$ for $i=1,\cdots,n$, 
then it holds that
\begin{equation}\label{ineq:xy}
\Gamma_{k,m}(z) \le \Gamma_{k,m}(x)\,\Gamma_{k,m}(y).
\end{equation}

In our first main result, we refine the estimation of $\delta_{k}(Y)$ 
and compare it to $\delta_{k}(X)$.  

%%%%%%%%%%%%%%%%%
\begin{theorem}\label{thm:ARR1}
Under the conditions of Lemma~\ref{lem:Y=},
\begin{equation}\label{accelerate1}
	\delta_{k}(Y) \le 
	\Gamma_{k,m}(Q^{T}X) \Gamma_{k,m}(V) 
	\left\|G_{1}^{\dagger}E_{k}\right\|_{2}.
\end{equation}
Furthermore,
%%%%%%%%%%%%%%%%%
\begin{equation}\label{accelerate2}
	\frac{\delta_{k}(Y)}{\delta_{k}(X)} \le 
	\frac{\max_{j>m}\|q_{j}^{T}X\|}{\max_{j>k}\|q_{j}^{T}X\|}
	\Gamma_{k,m}(V) \left\|G_{1}^{\dagger}E_{k}\right\|_{2}.
\end{equation}
\end{theorem}
\begin{proof}
Observe that the ratio in the right-hand side of \eqref{ineq:deltaY} 
is none other than $\Gamma_{k,m}(d)$.  
Applying \eqref{ineq:xy} to $M=d$ where 
$d=\diag(D)$ with $D_{ii}$ defined in \eqref{def:Dii}, 
$x_{i}=\|q_{i}^{T}X\|$ and $y_{i}=\|e_{i}^{T}V\|$, we derive 
$%\begin{equation*}
\Gamma_{k,m}(d) \le \Gamma_{k,m}(Q^{T}X) \Gamma_{k,m}(V).
$%\end{equation*}
We observe that 
$
\|e_{i}^{T}G_{2}G_{1}^{\dagger}E_{k}\| \le
\|G_{1}^{\dagger}E_{k}\|_{2}
$
for all $i \in [1,n-m]$, 
since the row vectors $e_{i}^{T}G_{2}$ are all unit vectors.
Substituting the above two inequalities into \eqref{ineq:deltaY}, 
we arrive at \eqref{accelerate1}.
To derive \eqref{accelerate2}, we simply observe that
\begin{eqnarray*}
\Gamma_{k,m}(Q^{T}X)  
= \frac{\max_{j>m}\|q_{j}^{T}X\|}{\min_{j\le k}\|q_{j}^{T}X\|} 
%&=& \frac{\max_{j>k}\|q_{j}^{T}X\|}{\min_{j\le k}\|q_{j}^{T}X\|}
%\frac{\max_{j>m}\|q_{j}^{T}X\|}{\max_{j>k}\|q_{j}^{T}X\|} \\
= \delta_{k}(X) \frac{\max_{j>m}\|q_{j}^{T}X\|}{\max_{j>k}\|q_{j}^{T}X\|}.
\end{eqnarray*}
Substituting the above into \eqref{accelerate1} and dividing both
sides by $\delta_{k}(X)$, we obtain \eqref{accelerate2}.
\end{proof}

To put the above results into perspective, let us examine the 
right-hand side of \eqref{accelerate2}.   Clearly, the first term, 
the ratio involving $\|q_{j}^{T}X\|$'s, is always less than or equal 
to one since $k \le m$, and it decreases as $m$ increases.  In
particular, when $m = k+1+pk$ with $p > 0$ and a large $k$,
then $m \gg k$ and the ratio can be tiny as long as there is a 
significant decay in $\{\|q_{j}^{T}X\|\}_{j=1}^{n}$ between indices 
$k$ and $m$.
In addition, from \eqref{eqn:GamV}, we know that the second term 
$\Gamma_{k,m}(V)\le 1$ and can be far less than one if there is a 
large decay between $\lambda_{k}$ and $\lambda_{m+1}$.
The third term $\|G_{1}^{\dagger}E_{k}\|_{2}$, however, presents a 
complicating factor.  How this term behaves as $p$ increases requires 
a scrutiny which will be the topic of Section~\ref{sec:G1}.

Similarly, we can examine the right-hand side of \eqref{accelerate1}
in which only the first term is different.  Given a good approximate 
basis $X$ for which the row norms of $Q^{T}X$ have a nontrivial decay, 
we can also have $\Gamma_{k,m}(Q^{T}X) \ll 1$; and the faster the 
decay is, the smaller is the term $\Gamma_{k,m}(Q^{T}X)$. 
Therefore, with the exception of the term $\|G_{1}^{\dagger}E_{k}\|_{2}$, 
all the terms in the right-hand sizes of \eqref{accelerate1} and 
\eqref{accelerate2} are small under reasonable conditions.

Next we consider the case where $X\in\R^{n\times k}$ is the result of 
applying a block power iteration $q$ times to an initial random matrix 
$X_{0} \in \R^{n\times k}$,
\begin{equation}\label{def:Xi}
 X = \rho(A)^{q}X_{0} = Q\rho(\Lambda)^{q}Q^{T}X_{0},
\end{equation}
where $\rho(A)$ is a polynomial or rational matrix function accelerator
(or filter) such that 
\begin{equation}\label{order1}
\min_{1\le j\le k}|\rho(\lambda_{j})| = |\rho(\lambda_{k})| 
\ge |\rho(\lambda_{k+1})| \ge \cdots \ge
|\rho(\lambda_{m+1})| = \max_{m< j\le n}|\rho(\lambda_{j})|.
\end{equation}
\begin{theorem}\label{thm:ARR2}
Let $X$ be defined in (\ref{def:Xi}) from an initial matrix 
$X_{0} \in \R^{n\times k}$.
Assume that the conditions of Lemma~\ref{lem:Y=} hold. Then 
there exists a constant $c_{m}$ such that
%%%%%%%%%%%%%%%%%
\begin{equation}\label{accelerate3}
	\delta_{k}(Y) \le c_{m}
	\left|\frac{\rho(\lambda_{m+1})}{\rho(\lambda_{k})}\right|^{q},
\end{equation}
where
\begin{equation}
c_{m} = 	\Gamma_{k,m}(Q^{T}\!X_{0})\Gamma_{k,m}(V)
	\left\|G_{1}^{\dagger}E_{k}\right\|_{2}.
\end{equation}
Moreover, there exists a constant $c'_{m}$ such that
%%%%%%%%%%%%%%%%%
\begin{equation}\label{accelerate4}
	\frac{\delta_{k}(Y)}{\delta_{k}(X)} \le c'_{m}
	\left|\frac{\rho(\lambda_{m+1})}{\rho(\lambda_{k+1})}\right|^{q},
\end{equation}
where
\begin{equation}
c'_{m} = \frac{\max_{j>m}\|q_{j}^{T}X_{0}\|}{\min_{j>k}\|q_{j}^{T}X_{0}\|}
	\Gamma_{k,m}(V)\left\|G_{1}^{\dagger}E_{k}\right\|_{2}.
\end{equation}
\end{theorem}

\begin{proof}
It follows from $Q^{T}X = \rho(\Lambda)^{q}Q^{T}X_{0}$ that
\begin{equation}\label{eqn:xqx0}
\|q_{i}^{T}X\| = |\rho(\lambda_{i})|^{q}\|q_{i}^{T}X_{0}\|, \;\; 
i = 1, \cdots, n.
\end{equation}
Applying \eqref{ineq:xy} to \eqref{eqn:xqx0}, we obtain
$ %\begin{equation*}
\Gamma_{k,m}(Q^{T}X) \le \Gamma_{k,m}
(\rho(\Lambda)^{q}) \Gamma_{k,m}(Q^{T}X_{0})
$ %\end{equation*}
which establishes \eqref{accelerate3}, 
upon substituting into \eqref{accelerate1}.

To prove \eqref{accelerate4}, we first use \eqref{eqn:xqx0} to calculate
\begin{eqnarray*}
\frac{\max_{j>m}\|q_{j}^{T}X\|}{\max_{j>k}\|q_{j}^{T}X\|} 
= \frac{\max_{j>m} |\rho(\lambda_{j})|^{q}\|q_{j}^{T}X_{0}\|}
{\max_{j>k} |\rho(\lambda_{j})|^{q}\|q_{j}^{T}X_{0}\|}     \le
\left|\frac{\rho(\lambda_{m+1})}{\rho(\lambda_{k+1})}\right|^{q}
\frac{\max_{j>m}\|q_{j}^{T}X_{0}\|}{\min_{j>k}\|q_{j}^{T}X_{0}\|}.
\end{eqnarray*}
Then substituting the above into \eqref{accelerate2} yields
\eqref{accelerate4}.
\end{proof}

Let us also state a couple of special cases of \eqref{accelerate3}.
\begin{corollary}
If the $G_{1}$-Assumption holds for $m = k+pk$, 
then there exist constants $C_{p}$ and $C'_{p}$ such that
\begin{equation*}\label{accelerate5}
	\delta_{k}(Y) \le C_{p}
	\left|\frac{\rho(\lambda_{k+1+pk})}{\rho(\lambda_{k})}\right|^{q}
\qquad\mbox{ and }\qquad
	\frac{\delta_{k}(Y)}{\delta_{k}(X)} \le C'_{p}
	\left|\frac{\rho(\lambda_{k+1+pk})}{\rho(\lambda_{k+1})}\right|^{q}.
\end{equation*}
In particular, when there is no augmentation ($p=0$) and no acceleration
($\rho(t)=t$), the convergence rate reduces to
$ %\begin{equation*}
\delta_{k}(Y) \le C_0 \left|{\lambda_{k+1}}/{\lambda_{k}}\right|^{q}.
$ %\end{equation*}
\end{corollary}

Finally, we remark that all of our results point out that there exists a matrix 
$Y \in \R^{n\times k}$ in the augmented subspace $\cR(\Xp)$ (which is 
constructed from the matrix $X$) that is a better approximate basis for 
$\cR(Q_{k})$ than $X$ is, under reasonable conditions.  It is known that
the Ritz pairs produced by the RR procedure are optimal approximations 
to the eigenpairs of $A$ from the input subspace (see \cite{parlett} for 
example).   Therefore, the derived bounds in this section should be 
attainable by the Ritz pairs generated by the ARR procedure.

%%%%%%%%%%%%%%%%%%
\subsection{Validity of $G_{1}$-Assumption}
\label{sec:G1}

A key condition for our results is the $G_{1}$-Assumption, given
in \eqref{G1-Assumption}, that requires the first $m$ rows of $G$
in \eqref{def:G} to be linearly independent.   Under this assumption, 
the larger $m$ is, the better the convergence rate could be. 
  
Let us examine the matrix $G_{1}$ consisting of the first $m$ rows of 
$G$ in \eqref{def:G}.  To simplify notation, we use $H$ for $G_{1}$, 
redefine $C$ as the first $m$ row of $C$ in \eqref{def:G}, and consider 
the matrix 
\begin{equation}\label{def:H}
  H = [C \;\; \Lambda_{m}C \;\; \cdots \;\; \Lambda_{m}^{p}C] 
  \in \R^{m \times (p+1)k},
\end{equation}
where $\Lambda_{m}$ is the $m \times m$ leading block of 
$\Lambda$ whose disgonal is assumed to be positive. 

We first give a necessary condition for the $m$ rows 
of $H$ to be linearly independent. 
\begin{proposition}\label{necessity}
Let $m \in (k,k+pk]$ for $p>0$.  The matrix $H\in\R^{m \times (p+1)k}$ 
defined in \eqref{def:H} has full rank $m$ only if $\Lambda_{m}$ has 
no more than $k$ equal diagonal elements (i.e., $\Lambda_{m}$ contains 
no eigenvalue of multiplicity greater than $k$).
\end{proposition}
\begin{proof}
Without loss of generality, suppose that the first $k+1$ diagonal elements 
of $\Lambda_{m}$ are all equal, i.e., 
$\lambda_{1}=\lambda_{2}=\cdots=\lambda_{k+1}=\alpha$.
Then the first $k+1$ rows of $H$, say $H'$, is of the form
$H' = [C' \;\; \alpha C' \;\; \cdots \;\; \alpha^{p}C']$,
where $C'$ consists of the first $k+1$ rows of $C$.
Since all column blocks are scalar multiples of $C'$
which has $k$ columns, the rank of $H$ is at most $k$.  
independent of $m$.
\end{proof}

The fact that $H$ is built from $C$ which has only $k$ columns
dictates that to have $\rank(H)$ greater than $k$, it is necessary
that the maximum multiplicity of $\Lambda_{m}$ must not exceed $k$.   

On the other hand, the next result says that when $p=1$ and $m$ 
reaches its upper bound $2k$, a multiplicity equal to $k$ is sufficient 
for $H$ to attain the full rank $2k$ (i.e., to be nonsingular) in a generic case.

First, let us do the partitioning
\begin{equation}\label{part:CLH}
C = \left[\begin{array}{c} C_{1} \\ C_{2}\end{array}\right],
\;\;\;
\Lambda_{m} = \left[\begin{array}{cc} 
\Lambda_{1} &  \\ & \Lambda_{2}
\end{array}\right],
\;\;\;
H = \left[\begin{array}{cc} 
C_{1} & \Lambda_{1}C_{1} \\ C_{2} & \Lambda_{2}C_{2}
\end{array}\right].
\end{equation}
where $m=2k$, and $C_{j}, \Lambda_{j}$, $j=1,2$, are all $k\times k$ submatrices.
Recall that $\Lambda_{1}$ consists of the first $k$ eigenvalues of 
$A$ and $\Lambda_{2}$ the next $k$ eigenvalues.

\begin{proposition}\label{prop:p=1}
Let $p=1$, $m=2k$, and $C$, $\Lambda_{m}$ and $H$ be defined as in 
\eqref{part:CLH}.   Let $r$ be the maximum multiplicity of $\Lambda_{m}$. 
Assume that any $k\times k$ submatrix of $C$ is nonsingular.  Then $H$ 
is nonsingular for $r=k$.
\end{proposition}

\begin{proof}
We will show that when $\lambda_{1}$ or $\lambda_{k+1}$ has multiplicity
$k$, then $H$ is nonsingular.  All the other cases can be similarly proven 
with appropriate permutations before partitioning \eqref{part:CLH} is done.

First, the nonsingularity of $H$ is equivalent to that of
\begin{equation*}\label{def:H'}
\left[\begin{array}{cc} 
C_{1} & \Lambda_{1}C_{1} \\ C_{2} & \Lambda_{2}C_{2}
\end{array}\right]
\left[\begin{array}{cc} C_{1}^{-1} &  \\ & C_{1}^{-1} \end{array}\right]
= \left[\begin{array}{cc} 
I & \Lambda_{1} \\ C_{2}C^{-1} & \Lambda_{2}C_{2}C^{-1}
\end{array}\right]
= \left[\begin{array}{cc} 
I & \Lambda_{1} \\ F & \Lambda_{2}F
\end{array}\right],
\end{equation*}
where $F\triangleq C_{2}C_{1}^{-1}$ is nonsingular by our assumption.
Eliminating the (2,1) block, we obtain
\[
\left[\begin{array}{cc} 
I & \Lambda_{1} \\ F & \Lambda_{2}F
\end{array}\right]  
~~\longrightarrow~~
\left[\begin{array}{cc} 
I & \Lambda_{1} \\ 0 & \Lambda_{2}F-F\Lambda_{1}
\end{array}\right] 
\]
Hence, the nonsingularity of $H$ is equivalent to 
that of $F\Lambda_{1}-\Lambda_{2}F$, or in turn
equivalent to that of the following matrix:
\begin{equation}\label{def:K}
K = \Lambda_{1}-F^{-1}\Lambda_{2}F.
\end{equation}
If the multiplicity of $\lambda_{1}$ is $k$ (implying that 
$\Lambda_{1}=\lambda_{k}I$), \eqref{def:K} reduces to
$K = F^{-1}(\lambda_{k}I-\Lambda_{2})F$.
On the other hand, if the multiplicity of $\lambda_{k+1}$ is $k$
(implying that $\Lambda_{2}=\lambda_{k+1}I$), then 
$K = \Lambda_{1}-\lambda_{k+1}I$.
In either case, $K$ is nonsingular since $\lambda_{k+1}<\lambda_{k}$; 
hence, so is $H$.
(Also in either case, $K$ becomes singular for multiplicity $r > k$ 
which implies $\lambda_{k+1}=\lambda_{k}$.)
\end{proof}

In Proposition~\ref{prop:p=1}, we assume that every $k \times k$ 
submatrix of $C$ is nonsingular.   It is well-known that for a generic
random matrix $C$, this assumption holds with high probability.  
Therefore, in a generic setting Proposition~\ref{prop:p=1} holds 
with high probability.  

Now the unproven case is for maximum multiplicity $r<k$.  Let us
rewrite $K$ in \eqref{def:K} into a sum of two matrices,
\begin{equation}
K = (\Lambda_{1}-\lambda_{k}I) + F^{-1}(\lambda_{k}I-\Lambda_{2})F.
\end{equation}
The first is diagonal and positive 
semidefinite, and the second has positive eigenvalues when 
$\lambda_{k}>\lambda_{k+1}$, but is generally asymmetric. So far,
we have not been able to find a result that guarantees nonsingularity
for such a matrix $K$.  However, in a generic setting where $K$ 
comes from random matrices, nonsingularity should be expected 
with high probability (which has been empirically confirmed by 
our numerical experiments).   
%Nevertheless, this topic remains to be theoretically resolved.

It should be noted that $G_{1}$ being nonsingular with 
$m=k+kp$ represents the best scenario where the acceleration 
potential of $p$-block augmentation is fully realized.  However, 
$m<k+kp$ does not represent a failure, considering the fact that 
as long as $m>k$, an acceleration is still realized to some extent.

Once it is established for $p=1$ and $m=2k$ that in a generic setting 
$H$ is nonsingular whenever the maximum multiplicity of $\Lambda_{m}$ 
is less than or equal to $k$, the same result can in principle be extended 
to the case of $p=3$ by considering
\[
H = \left[C\;\; \Lambda C\;\; \Lambda^{2}C\;\; \Lambda^{3}C\right] = 
\left[[C\;\; \Lambda C]\;\; \Lambda^{2}[C\;\; \Lambda C]\right]
= [\hat{C} \;\; \hat{\Lambda}\hat{C}],
\]
where $\hat{C}=[C\;\; \Lambda C]$ and $\hat{\Lambda}=\Lambda^{2}$,
which has the same form as for the case $p=1$.   It will also cover 
the case of $p=2$ where the matrix involved is a submatrix of the 
one for $p=3$.

It is worth noting that $m=(p+1)k$ could be kept constant if $k$ is 
decreased while $p$ is increased.  Is it sensible to use fewer 
vectors in power iterations but to compensate it with an 
augmentation of more blocks?  Although in some cases this
strategy works well, in general it seems to be a risky approach 
for two reasons.  First, the smaller $k$ is, the more likely 
it is to encounter matrices that have eigenvalues of multiplicity 
greater than $k$.  In this case, by Proposition~\ref{necessity}, the
benefit of augmentation could become limited.  Secondly, we 
have observed in numerical experiments that the condition number 
of $G_{1}$ tends to increase as $p$ increases, which would in turn 
increase the constants $c_{m}$ and $c'_{m}$ in 
\eqref{accelerate1}-\eqref{accelerate2}.  
These facts suggest that using a small $k$ and a large $p$ to compute
more than $k$ eigenpairs could be numerically problematic.  In our 
implementation, we choose to be conservative by using the default 
value of $p=1$, while setting $k$ to be slightly bigger than the 
number of eigenpairs to be computed.

%\section{Practical Issues}
\section{Polynomial Accelerators} \label{sec:poly}
To construct polynomial accelerators (or filters) $\rho(t)$, we use 
Chebyshev interpolants on highly smooth functions.  
Chebyshev interpolants are polynomial interpolants on
Chebyshev points of the second kind, defined by
\begin{equation}
t_{j}=-\cos(j\pi/N),  0 \le j \le N,
\end{equation}
where $N \ge 1$ is an integer.  Obviously, this set of $N+1$ points are 
in the interval $[-1,1]$ inclusive of the two end-points. 
Through any given data values $f_{j}, j=0,1,\cdots,N$, at these 
$N+1$ Chebyshev points, the resulting unique polynomial interpolant 
of degree $N$ or less is a Chebyshev interpolant.  It is known
that Chebyshev interpolants are ``near-best'' \cite{Ehlich-Zeller1966}.

Our choices of functions to be interpolated are
\begin{equation}\label{def:fN}
f_{d}(t) = (f_{1}(t))^{d} 
\quad\mbox{ where }\quad
f_{1}(t) = \max(0,t)^{10},
\end{equation}
and $d$ is a positive integer.   
Obviously, $f_{d}(t) \equiv 0$ for $t \le 0$ and $f_{d}(1) \equiv 1$.  
The power 10 is rather arbitrary and exchangeable with other 
numbers of similar magnitude without making notable differences.

The functions in \eqref{def:fN} are many times differentiable so that their 
Chebyshev interpolants converge relatively fast, %(at the algebraic rate $1/N^{k}$) 
see \cite{Mastroianni-Szabados1995}.   Interpolating such smooth functions
on Chebyshev points helps reducing the effect of the Gibbs phenomenon and
allows us to use relatively low-degree polynomials. 

There is a well-developed open-source Matlab package called {\tt Chebfun}
\cite{Driscoll2014} for doing Chebyshev interpolations, among many other 
functionalities\footnote
{Also see the website \url{http://www.chebfun.org/docs/guide/guide04.html}}.
In this work, we have used {\tt Chebfun} to construct Chebyshev interpolants 
as our polynomial accelerators. Specifically, we interpolate the function
$f_{d}(t)$ by the $d$-th degree Chebyshev interpolant polynomial, say,
\be \label{eq:chebfun-f} 
\psi_d(t) = \gamma_1 t^d + \gamma_2 t^{d-1} + \ldots + \gamma_d t +
\gamma_{d+1}.
\ee

Suppose that we want to dampen the eigenvalues in an interval $[a, b]$, 
where $a\le \lambda_n$ and $b< \lambda_k$, while magnifying eigenvalues
to the right of $[a, b]$.   Then we map the interval $[a, b]$ onto $[-1, 1]$ 
by an affine transformation and then apply $\psi_d(\cdot)$ to $A$. That is, 
we apply the following polynomial function to $A$, 
\be \label{eq:rhod}
\rho_d(t)= \psi_d\left( \frac{2t - a - b}{b-a}\right).
\ee 
Let $\Gamma_d = (\gamma_{1},\gamma_{2},\cdots,\gamma_{d+1})$ 
denote the coefficients of the polynomial $\psi_d(t)$ in \eqref{eq:chebfun-f}.
The corresponding matrix operation $Y=  \rho_d(A)X$ can be 
implemented by Algorithm~\ref{alg:poly} below.
\begin{algorithm2e}[htb!]
\label{alg:poly}
%\small
\caption{Polynomial function: $Y = \poly(A,X,a,b,\Gamma_d)$}
\SetKwInOut{Input}{input}\SetKwInOut{Output}{output}
\SetKwComment{Comment}{}{}
\BlankLine
\DontPrintSemicolon
Compute $c_0=\frac{a+b}{a-b}$ and $c_1=\frac{2}{b-a}$. Set $Y=\gamma_1X$.\\
\lFor{$j=1,2,\ldots,d$}{
$Y=c_0 Y + c_1 A Y+\gamma_{j+1}X$.
}
\end{algorithm2e}

%In Figure~\ref{fig:poly}, the function $f_{1}(t)$ in \eqref{def:fN} is 
%plotted on the left, and $\psi_d(t)$ are plotted on the right for $d=2$ to $7$.
%Among these polynomials, the higher the degree is, the closer 
%the curve is to the vertical line $t=1$.
%Before a monotone ascent towards 1, all these polynomials 
%keep a low profile with a maximum magnitude less 
%than or around $0.2$.

For a quick comparison, we plot our Chebyshev interpolates of
degrees 2 to 7 and the Chebyshev polynomials of degrees 2 to 7
side by side in Figure~\ref{fig:poly}.  For both kinds of polynomials, 
the higher the degree is, the closer the curve is to the vertical line $t=1$.
We observe that inside the interval $[-1,1]$, our Chebyshev interpolates
have lower profiles (with magnitude less than or around 0.2 except near 1) 
than the Chebyshev polynomials which oscillate between $\pm 1$, 
while outside $[-1,1]$ the Chebyshev polynomials grow faster.  

\begin{figure}[htb]
\centering
  \hfill
    \subfigure[Chebyshev interpolants of degrees 2 to 7]{
    \includegraphics[width=0.45\textwidth,height=0.33\textwidth]
    {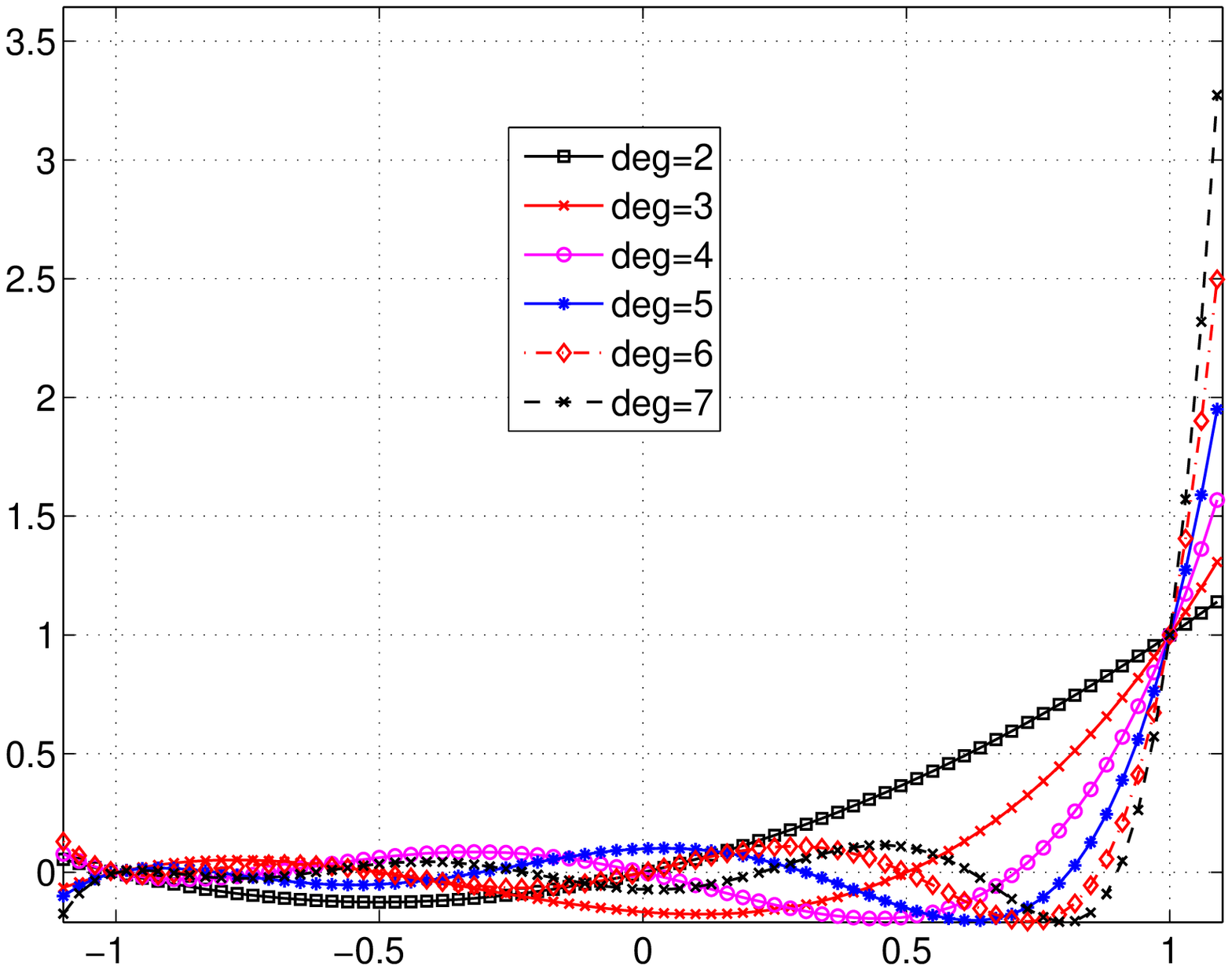}
  }
  \hfill
    \subfigure[Chebyshev polynomials of degrees 2 to 7]{
    \includegraphics[width=0.45\textwidth,height=0.33\textwidth]
    {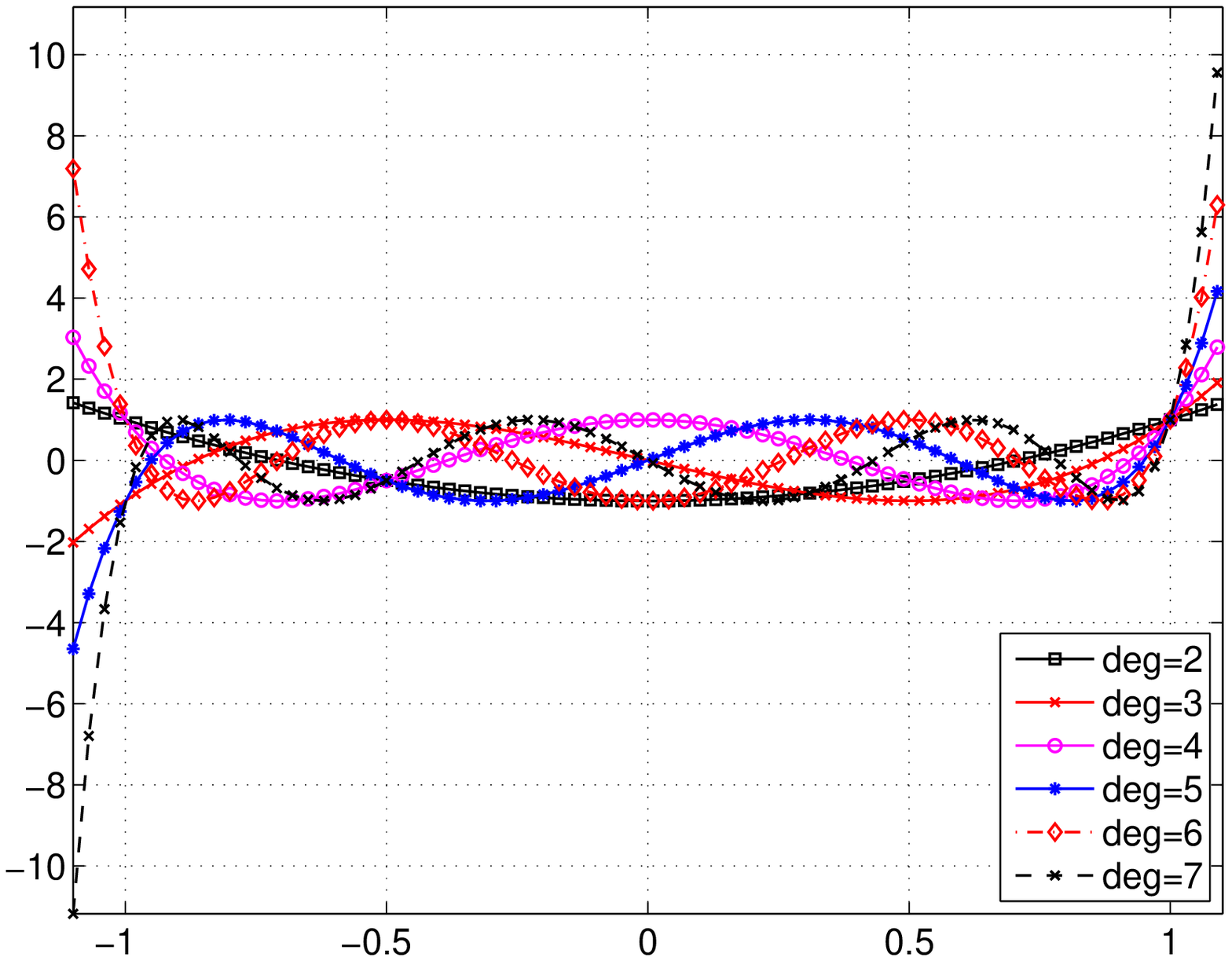}
  }
\caption{illustration of polynomial functions}
\label{fig:poly}
\end{figure}

The idea of polynomial acceleration is straightforward and old, but its 
success is far from foolproof, largely due to inevitable errors in estimating 
intervals within which eigenvalues are supposed to be suppressed or 
promoted.  The main reason for us to prefer our Chebyshev interpolates 
over the classic Chebyshev polynomials is that their lower profiles tend 
to make them less sensitive to erroneous intervals, hence easier to control.  
Indeed, our numerical comparison, albeit limited, appears to justify our choice.

%%%%%%%%%%%%%%%%%%%%%%%%%%%
\section{Details of \arrabit~ Algorithms} \label{sec:alg}

In this section, we describe technical details and give parameter choices 
for our \arrabit~algorithm which computes $k$ eigenpairs corresponding to
$k$ algebraically largest eigenvalues of a given symmetric matrix $A$.   

\textbf{Guard vectors}. 
When computing $k$ eigenpairs, it is a common practice to compute a few
extra eigenpairs to help guard against possible slow convergence.  For
this purpose, a small number of ``guard vectors" are added to the iterate
matrix $X$.  In general, the more guard vectors are used, the less iterations 
are needed for convergence, but at a higher cost per iteration on memory 
and computing time.  In our implementation, we set the number of columns
in iterate matrix $X$ to $k+q$, where by default $q$ is set to $0.1k$ 
(rounded to the nearest integer). 

\textbf{Estimation of $\lambda_n$ and $\lambda_{k+q}$}.
To apply polynomial accelerators, we need to estimate the interval 
$[a,b]=[\lambda_n,\lambda_{k+q}]$ which contains unwanted eigenvalues.   
The smallest eigenvalue $\lambda_n$ is computed by calling the the 
Matlab built-in solver \eigs~ (i.e., \arpack~\cite{Sorensenetall1997}).  
Given an initial matrix $X \in \R^{n\times (k+q)}$ whose columns are
orthogonalized, an under-estimation of $\lambda_{k+q}$ can be taken
as the smallest eigenvalue of the projected matrix $X^TAX$ 
(which requires an RR projection).  
As the iterations progress, more accurate estimates of $\lambda_{k+q}$ 
will becomes available after each later ARR projection.

\textbf{Outer loop stop rule}.  
Let $(x_{i},\mu_{i})$, $i=1,2,\cdots,k$, be computed Ritz pairs where 
$x_{i}^{T}x_{j}=\delta_{ij}$.  
We terminate the algorithm when the following maximum relative residual norm 
becomes smaller than a prescribed tolerance $\tol$, i.e.,
\begin{equation}\label{eq:out-stop}
\maxres:=\max_{i=1,\ldots,k} \{\res_i\} \le \tol, 
%\quad \res_i := \frac{\|Ax_i - \mu_ix_i\|_2}{\max(1, |\mu_i|)}, \; i=1,\cdots,k, 
\end{equation}
where
\begin{equation}\label{eq:resi}
\res_i := \frac{\|Ax_i - \mu_ix_i\|_2}{\max(1, |\mu_i|)}, 
\; i=1,\cdots,k.
\end{equation}
The algorithm is also stopped in the following three cases:
(i) if a maximum number of iterations, denoted by ``maxit'', is reached
(by default maxit = 30); or
(ii) if the maximum relative residual norm has not been reduced after 
three consecutive outer iterations; or
(iii) if most Ritz pairs have residuals considerably smaller than 
$\tol$ and the remaining have residuals slightly larger than $\tol$; 
specifically, $\maxres <(1 + 9 h/ k) \tol$  $(< 10*\tol)$, where $h$ 
is the number of Ritz pairs with residuals less than $0.1*\tol$.  
In our experiments we also monitor the computed 
partial trace $\sum_{i=1}^{k}\mu_{i}$ at the end for all solvers 
as a check for correctness.

\textbf{Continuation}. 
When a high accuracy (say, $\tol\le 10^{-8}$) is requested, we use a continuation 
procedure to compute Ritz-pairs satisfying a sequence of tolerances: 
$\tol_1 > \tol_2 > \cdots \ge \tol$, and use the computed Ritz-pairs for 
$\tol_t$ as the starting point to compute the next solution for $\tol_{t+1}$. 
In our implementation, we use the update scheme 
\begin{equation}\label{eq:contin}
\tol_{t+1} = \max(10^{-2}\,\tol_t, \,\tol), 
%\;\;\color{blue}\tol_1 = 10^{-6}.
\end{equation}
where $\tol_1$ is chosen to be considerably larger than $\tol$.
A main reason for doing such a continuation is that our deflation
procedure (see below) is tolerance-dependent.   At the early stages 
of the algorithm, a stringent tolerance would delay the activation of deflation
and likely cause missed opportunities in reducing computational costs.

\textbf{Inner loop parameters and stop rule}.   
Both \mpm~ and \gn~are tested as inner solvers to update $X$. These inner 
solvers are applied to the shifted matrix $A-a I$ which is supposedly positive 
semidefinite since $a$ is a good approximation to $\lambda_{n}$ (computed
by \eigs~in our implementation).  We check inner stopping criteria every 
$\mathrm{maxit}_2$ iterations and check them at most $\mathrm{maxit}_1$ 
times.  In the present version, the default values for these two parameters 
are $\mathrm{maxit}_1=10$ and $\mathrm{maxit}_2=5$ 
Therefore, the maximum number of inner iterations allowed is 
$\mathrm{maxit}_1 \times \mathrm{maxit}_2 = 50$. 

The inner loop stopping criteria are either
\begin{equation} \label{eq:inn-stop}
  \mathtt{rc}=\mathtt{rcond}(X^{T}X) \le \tol_{t} 
  \mbox{\;\;  or \; \; }
  \mathtt{rc}/\mathtt{rcp}>0.99,
\end{equation}
where
$\tol_t$ is the current tolerance (in a continuation sequence) 
and rcp is the previously computed $\mathrm{rcond}(X)$.   
In \eqref{eq:inn-stop}, we use the {\tt rcond} subroutine in LAPACK 
(also used by Matlab) to estimate the reciprocal 1-norm condition 
number of $X^{T}X$, which we find to be relatively inexpensive. 
The first condition in \eqref{eq:inn-stop} indicates that $X$ is about 
to lose (or have just lost) rank numerically, which implies that we 
achieve the goal of eliminating the unwanted eigenspace numerically.  
However, it is probable that a part of the desired eigenspace is also 
sacrificed, especially when there are clusters among the desired 
eigenvalues.  Fortunately, this problem can be corrected, at a cost, 
in later iterations after deflation.  
On the other hand, the second condition is used to deal with the
situation where the conditioning of $X$ does not deteriorate, which 
occurs from time to time in later iterations when there exists little 
or practically no decay in the relevant eigenvalues.

\textbf{Deflation}. 
Since Ritz pairs normally have uneven convergence rates, a procedure
of detecting and setting aside Ritz pairs that have ``converged'' is called
deflation or locking, which is regularly used in eigensolvers because it
not only reduces the problem size but also facilitates the convergence
of the remaining pairs.   In our algorithm, a Ritz pair $(x_i, \mu_i)$ is 
considered to have ``converged'' with respect to a tolerance $\tol_t$ 
if its residual (see~\eqref{eq:resi} for definition) satisfies
\begin{equation}\label{eq:defl}
\res_{i} \le \max(10^{-14},\tol_t^2).
\end{equation}
After each ARR projection, we collect the converged Ritz vectors into 
a matrix $Q_c$, and start the next iteration from those Ritz vectors 
``not yet converged'', which we continue to call $X$.  
Obviously, whenever $Q_c$ is nonempty
$X$ is orthogonal to $Q_c$.  Each time we check the stopping rule 
in the inner loop, we also perform a projection $X=X -Q_c(Q_c^T X)$ 
to ensure that $X$ stays orthogonal to $Q_c$.   
In addition, the next ARR projection will also be performed
in the orthogonal complement of $\cR(Q_c)$.  That is, we apply an ARR 
projection to the matrix $Y-Q_c(Q_c^TY)$ for $Y = [X\; AZ\; \cdots\; A^{p}X]$. 
At the end, we always collect and keep $k+q$ leading Ritz pairs from both 
the ``converged'' and the ``not yet converged'' sets.

\textbf{Augmentation blocks}.
The default value for the number of augmentation blocks is $p=1$, but
this value may be adjusted after each ARR projection.  We increase $p$ 
by one when we find that the relevant Ritz values show a small decay 
and at the same time the latest decrease in residuals is not particularly 
impressive.  Specifically, we set $p=p+1$ if
\begin{equation}\label{eq:block}
\frac{\mu_{k+q}}{\mu_k} > 0.95 
\quad \mbox{ and } \quad
\frac{\maxres }{\mathtt{maxresp}} > 0.1,
\end{equation}
where $\mathtt{maxresp}$ is the maximum relative residual norm at the 
previous iteration.  The values $0.95$ and $0.1$ are set after some limited
experimentation and by no means optimal.  For $k$ relatively large, since 
the memory demand grows significantly as $p$ increases, we also limit the 
maximum value of $p$ to $p_{\max}=3$.

\textbf{Polynomial degree}.
Under normal conditions, the higher degree is used in a polynomial accelerator,
the fewer number of iterations will be required for convergence, but at a higher
cost per iteration.  A good balance is needed.
Let $d$ and $d_{\max}$ be the initial and the largest polynomial degrees,
respectively.    We use the default values $d=3$ and $d_{\max}=15$.
Let  $\rho_d(t)$ be the polynomial function defined in \eqref{eq:rhod}.
After each ARR step, we adjust the degree based on estimated spectral 
information of $\rho_d(A)$ computable using the current Ritz values.   
We know that the convergence rate of the inner solvers would be satisfactory
if the eigenvalue ratio ${\rho_d(\lambda_{k+q})}/{\rho_d(\lambda_{k})}$ 
is small.  Based on this consideration, we calculate
\begin{equation}\label{eq:hat-d}
	\hat{d} = \min_{d \ge 3}\left\{d \in \mathbb{Z}: 
	\frac{\rho_{d}(\mu_{k+q}^*)}{\rho_{d}(\mu_{k}^*)} < 0.9\right\}, 
\end{equation}
and then apply the cap $d_{\max}$ by setting
\begin{equation}\label{eq:degree}
	d = \min(\hat{d}, d_{\max})
\end{equation}
where $\mu_{k}^*$ and $\mu_{k+q}^*$ are a pair of Ritz values 
corresponding to the iteration with the smallest residual ``$\maxres$'' 
defined in \eqref{eq:out-stop}  (therefore the most accurate so far). 
The value of $0.9$ is of course adjustable.

Finally, a pseudocode for our \arrabit~ algorithm with all the above features
is presented as Algorithm \ref{alg:Arrabit2}.  This is the version used to 
produce the numerical results of this paper.  As one can see, 
\arrabit~ algorithm uses $A$ only in matrix multiplications.

\begin{algorithm2e}[htb!]
\label{alg:Arrabit2}
%\small
\caption{Algorithm \arrabit~(detailed version)}
\SetKwInOut{Input}{input}\SetKwInOut{Output}{output}
\SetKwComment{Comment}{}{}
\BlankLine
\DontPrintSemicolon
Input $A \in \R^{n\times n}$, integer $k \in (0,n)$ and tolerance $\tol>0$.\\
Choose $d$ and $d_{\max}$, the initial and maximum polynomial degrees. 
\tcc*[f]{initialize}\\
Choose $p$ and $p_{\max}$, the initial and maximum number of augmentation blocks.\\
Choose $q\ge0$, the number of guard vectors, so that $(p+1)(k+q) < n$.\\
Set tolerance parameters: $t=1$, $\tol_t\ge \tol$ and
$\tol_d=\max(10^{-14},\tol_t^2)$.\\
Initialize converged Ritz pairs $(Q_c,\Sigma_c) = \emptyset$ 
for deflation purposes.\\
Initialize an {\it i.i.d.}~Gaussian random matrix $X \in \R^{n \times (k+q)}$.\\
Estimate the interval $[\lambda_n,\lambda_{k+q}] \approx [a,b]$.\\ 
%such that $a\le \lambda_n$ and $b < \lambda_{k+q}$.\\
\For(\tcc*[f]{outer loop}){$j=1,\ldots,\mathrm{maxit}$}{
Initialize rc to infinity.\\% and inc = 1. \\
\For(\tcc*[f]{inner loop}){$i_{1} = 1, 2, \cdots, \mathrm{maxit}_{1}$,}{
\For(\tcc*[f]{call inner solvers}){$i_{2} = 1, 2, \cdots, \mathrm{maxit}_{2}$,}{
\If(\tcc*[f]{MPM}){\mpm~ is the inner solver,}{
Call $X = \poly(A-aI,X,0,b-a,\Gamma_d)$. \tcc*[f]{accelerator}\\
Normalize the columns of $X$ individually. 
}
\If(\tcc*[f]{GN}){\gn~ is the inner solver,}{
Compute $Y =  X\left(X^T X\right)^{-1}$. \\
Call $Z = \poly(A-aI,Y,0,b-a,\Gamma_d)$.\tcc*[f]{accelerator}\\
Compute $X = Z - X(Y\zz Z - I)/2$.
}
Compute $X = X - Q_c (Q_c^T X)$ if $Q_c\ne\emptyset$.
\tcc*[f]{projection}
}
Set $\mathrm{rcp} = \mathrm{rc}$ and compute 
$\mathrm{rc} = \mathrm{rcond}(X^TX)$.\\
\lIf(\tcc*[f]{end inner loop})
{the inner stop rule \eqref{eq:inn-stop} is met,}{break.} 
}
Compute $Y=[X, AX, \ldots, A^p X]$. \tcc*[f]{augmentation}\\
$Y=Y-Q_c(Q_c^T Y)$ if $Q_c\ne\emptyset$.\tcc*[f]{projection}\\
Perform ARR step: $(X,\Sigma)=\RR(A,Y)$.\tcc*[f]{ARR}\\
Extract $k+q$ leading Ritz pairs $(x_{i},\mu_{i})$ from 
$(Q_c, \Sigma_c)$ and $(X,\Sigma)$.\\ 
Overwrite $(X,\Sigma)$ by the $k+q$ Ritz pairs.
Compute residuals by \eqref{eq:resi}.\\
\If{the outer stop rule \eqref{eq:out-stop} is met for $\tol$,}
   {output the Ritz pairs $(X,\Sigma)$ and exit. 
   \tcc*[f]{output and exit}\\
 }
\If(\tcc*[f]{continuation})
   {the outer stop rule \eqref{eq:out-stop} is met for $\tol_t$}{ 
   Set $\tol_{t+1} = \max\left(10^{-2}\tol_t, \tol\right)$,
   $b = \mu_{k+q}$ and $t=t+1$.
} 
Collect converged Ritz pairs in $(Q_c,\Sigma_c)$ that satisfy \eqref{eq:defl}.
\tcc*[f]{deflation}\\
Overwrite $(X, \Sigma)$ by the remaining not yet converged Ritz pairs. \\
\lIf(\tcc*[f]{update p})
  {rules in \eqref{eq:block} are met,}
  {set $p = \min(p+1, p_{\max})$.}
Update the polynomial degree by rules \eqref{eq:hat-d}-\eqref{eq:degree}.
\tcc*[f]{update degree}
}
\end{algorithm2e}

\section{Numerical Results} \label{sec:num}

In this section, we evaluate the performance of \arrabit~ on a set 
of sixteen sparse matrixes.   Although we have constructed the algorithm 
with parallel scalability in mind as a major motivating factor, a study of 
scalability issues in a massively parallel environment is beyond the 
scope of the current paper.   

As a first step, we test the algorithm in Matlab environment, on a single 
computing node (2 processors) and without explicit code parallelization, 
to determine how it performs in comparison to established solvers.   
We have implemented our \arrabit~algorithm, as is described by the 
pseudocode Algorithm~\ref{alg:Arrabit2}, in MATLAB.   For brevity, 
the two variants, corresponding to the two choices of inner solvers, 
will be called \mpm~ and \gn, respectively.
  
We test two levels of accuracy in our experiments: 
$\tol=10^{-6}$ or $\tol=10^{-12}$.  By our stoping rule, upon successful 
termination the largest eigenpair residual will not exceed $10^{-5}$ or 
$10^{-11}$, respectively.  Since our algorithm checks the termination 
rule only after each ARR call, it often returns solutions of higher 
accuracies than what is prescribed by the $\tol$ value.

%%%%%%
\subsection{Solvers, Platform and Test Matrices}

Since it is impractical to carry out numerical experiments with a large 
number of solvers, we have carefully chosen two high-quality packages
to compare with our \arrabit~code.  One package is \arpack\footnote
{See \url{http://www.caam.rice.edu/software/ARPACK/}}
\cite{Sorensenetall1997}, which is behind the Matlab built-in iterative 
eigensolver \eigs, and will naturally serve as the benchmark solver.  
Another is a more recent package called \feast~\cite{PingEric2014}
which has been integrated into Intel's Math Kernel Library (MKL) 
under the name ``Intel MKL Extended Eigensolver"\footnote
{See \url{http://software.intel.com/en-us/intel-mkl} 
(version 11.0.2 on our workstation)}. 
Both \arpack~ and \feast~ are written in Fortran.  While 
\arpack~ can be directly accessed through \eigs~ in Matlab, 
we call \feast~ from Intel's MKL Library via Matlab's \texttt{MEX} 
external interfaces.  
In our experiments, all parameters in \eigs~ and \feast~ are set to their default 
values, and each solver terminates with its own stopping rules using either 
$\tol=10^{-6}$ or $\tol=10^{-12}$.

We have also examined a few other solvers as potential candidates but decided 
not to use them in this paper, including but not limited to the filtered Lanczos
algorithm\footnote{See \url{http://www-users.cs.umn.edu/~saad/software/filtlan}} \cite{FangSaad2012} and the Chebyshev-Davidson algorithm\footnote
{See \url{http://faculty.smu.edu/yzhou/code.htm}} \cite{chebydav}. 
Our initial tests indicated that, for various reasons, these solvers'
overall performance could not measure up with that of the commercial-grade 
software packages \arpack~ and \feast~on a number of test problems.
This fact may be more of a reflection on the current status of software 
development for these solvers than on the merits of the algorithms behind.

It is important to note that \feast~ is designed to compute all eigenvalues 
(and their eigenvectors) in an interval, which is given as an input along with 
an estimated number of eigenvalues inside the interval.   When computing 
$k$ largest eigenpairs, we have observed that the performance of \feast~ 
is affected greatly by the quality of the two estimations: the interval itself
and the number of eigenvalues inside the interval.  When calling \feast, 
we set (i) the interval to be $[\lambda_k^*, \lambda_1^*]$ where 
$\lambda_k^*$ and $\lambda_1^*$ are computed eigenvalues by 
\eigs~ using the same tolerance $\tol$; and (ii) the estimated number 
of eigenvalues in the interval to $1.2k$ rounded to the nearest integer.  
We consider this setting to be fair, if not overly favorable, to \feast. 

Our numerical experiments are preformed on a single computing 
node of Edison\footnote{See 
\url{http://www.nersc.gov/users/computational-systems/edison/}}, 
a Cray XC30 supercomputer maintained at the National Energy Research
Scientific Computer Center (NERSC) in Berkeley.  The node consists of
two twelve-core Intel ``Ivy Bridge'' processors at 2.4 GHz  with a total of
64 GB shared memory.   Each core has its own L1 and L2 caches
of 64 KB %(32 KB instruction, 32 KB data) 
and 256 KB, respectively; A 30-MB L3 cache shared between 12 cores on
the ``Ivy Bridge'' processor.  We generate Matlab standalone executable programs 
and submit them as batch jobs to Edison.  The reported runtimes are wall-clock times.

On a multi/many-core computer, memory access patterns and 
communication overheads have a notable impact on computing time.  
In Matlab, dense linear algebra operations are generally well optimized by 
using BLAS and LAPACK tuned to the CPU processors in use.  On the other 
hand, we have observed that some sparse linear algebra operations in Matlab 
seem to have not been as highly optimized (at least in version 2013b).  
In particular, when doing multiplications between a large sparse matrix and a 
dense matrix (like $AX$), Matlab is often slower than a routine in Intel's Math 
Kernel Library (MKL) named ``{\tt mkl\_dcscmm}'' when it is invoked through 
Matlab's MEX external interfaces in our experiments.   For this reason, 
we use this MKL routine in our Matlab code to perform the operation $AX$.

Our test matrices are selected from the University of Florida Sparse Matrix 
Collection\footnote{See \url{http://www.cise.ufl.edu/research/sparse/matrices}}.
For each matrix, we compute both $k$ eigenpairs corresponding to $k$ 
largest eigenvalues and those corresponding to $k$ smallest eigenvalues.  
Many of the selected matrices are produced by PARSEC \cite{parsec}, 
a real space density functional theory (DFT) based code for electronic 
structure calculation in which the Hamiltonian is discretized by a finite 
difference method.   We do not take into account any background 
information for these matrices; instead, we simply treat them 
algebraically as matrices.

Table~\ref{tab:matrix} lists, for each matrix $A$, the dimensionality $n$, 
the number of nonzeros $nnz(A)$ and the density of $A$, i.e., the ratio 
$({nnz(A)}/{n^2})100\%$.   The number of eigenpairs to be computed 
is set either to 1\% of $n$ rounded to the nearest integer or to $k=1000$
whichever is smaller.
Table~\ref{tab:matrix} also reports the number of the nonzeros in
the Cholesky factor $L$ of matrix $A - \alpha I$ where 
$\alpha = \max(2\lambda_n(A),0)$.
The factorization is carried out after an ``approximate minimum degree'' 
permutation performed by the Matlab function ``\texttt{amd}'', as is 
done by the following MATLAB line:
$
\mathtt{t = amd(B); ~L = chol(B(t,t), 'lower')}.
$ 
We have also tested the ``symmetric 
approximate minimum degree'' permutation (``{\tt symamd}'' in Matlab), 
but the corresponding density of $L$ is slightly larger on most matrices. 
The density of factor $L$ and the computing time in seconds used by 
Cholesky factorization are also given in Table~\ref{tab:matrix}. 
Although all matrices $A$ are very sparse, the Cholesky factors of 
some matrices, such as Ga10As10H30, Ga3As3H12 and Ge87H76, are
quite dense.   As a result, the Cholesky factorization time varies greatly 
from matrix to matrix.
We mention that the spectral distributions of the test matrices can 
behave quite differently from matrix to matrix.  Even for the same matrix,
the spectrum of a matrix can change behavior drastically from region to region.
Most notably, computing $k$ smallest eigenpairs of many matrices in this set
turns out to be more difficult than computing $k$ largest ones.

The largest matrix size in this set is more than a quarter of million.  Relative
to the computing resources in use, we consider these selected matrices to 
be fairly large scale.
Overall, we consider this test set reasonably diverse and representative, 
fully aware that there always exist instances out there that are more 
challenging to one solver or another.

\begin{table}[htbp!]
  \setlength{\tabcolsep}{3pt}
  \centering
  \caption{Information of Test Matrices} 
  \label{tab:matrix}
{\footnotesize
  \begin{tabular}{|c|c|c|c|c|c|c|c|}
 \hline%\hline
 matrix name  & $n$ & $k$ & $nnz(A)$ & density of A & nnz(L) & density of L & time \\ \hline\hline

    Andrews & 60000 & 600 & 760154 & 0.021\% & 117039940 & 6.502\% &    7.18 \\ \hline
                C60 & 17576 & 176 & 407204 & 0.132\% & 34144169 & 22.105\% &    1.62 \\ \hline
               cfd1 & 70656 & 707 & 1825580 & 0.037\% & 35877440 & 1.437\% &    1.81 \\ \hline
            finance & 74752 & 748 & 596992 & 0.011\% & 2837714 & 0.102\% &    0.28 \\ \hline
        Ga10As10H30 & 113081 & 1000 & 6115633 & 0.048\% & 1562547805 & 24.439\% &  127.12 \\ \hline
          Ga3As3H12 & 61349 & 613 & 5970947 & 0.159\% & 596645077 & 31.705\% &   42.00 \\ \hline
   shallow\_water1s & 81920 & 819 & 327680 & 0.005\% & 2357535 & 0.070\% &    0.21 \\ \hline
            Si10H16 & 17077 & 171 & 875923 & 0.300\% & 56103003 & 38.474\% &    2.60 \\ \hline
             Si5H12 & 19896 & 199 & 738598 & 0.187\% & 78918573 & 39.871\% &    3.80 \\ \hline
                SiO & 33401 & 334 & 1317655 & 0.118\% & 186085449 & 33.359\% &   10.01 \\ \hline
          wathen100 & 30401 & 304 & 471601 & 0.051\% & 1490209 & 0.322\% &    0.32 \\ \hline
            Ge87H76 & 112985 & 1000 & 7892195 & 0.062\% & 1403571238 & 21.990\% &  109.64 \\ \hline
           Ge99H100 & 112985 & 1000 & 8451395 & 0.066\% & 1477089634 & 23.141\% &  120.08 \\ \hline
        Si41Ge41H72 & 185639 & 1000 & 15011265 & 0.044\% & 3457063398 & 20.063\% &  358.53 \\ \hline
            Si87H76 & 240369 & 1000 & 10661631 & 0.018\% & 5568995364 & 19.277\% & 1499.80 \\ \hline
        Ga41As41H72 & 268096 & 1000 & 18488476 & 0.026\% & 6998257446 & 19.473\% & 2498.43 \\ \hline
  \end{tabular}
  }
\end{table}

%For example, Figure~\ref{fig:spectral} illustrates the spectral distribution of the matrices cfd1 and finance.

%\begin{figure}[htb]
%\centering
%  \hfill
%%  \subfigure[matrix: Andrews]{
%%    \includegraphics[width=0.48\textwidth,height=0.4\textwidth]
%%    {Andrews-ev.eps}
%%  }
%  \subfigure[matrix: cfd1]{
%    \includegraphics[width=0.48\textwidth,height=0.4\textwidth]
%    {cfd1-log-ev.eps}
%  }
%  \hfill
%  \subfigure[matrix: finance]{
%    \includegraphics[width=0.48\textwidth,height=0.4\textwidth]
%    {finance-ev.eps}
%  }
%%  \hfill
%%    \subfigure[matrix: Ga10As10H30]{
%%    \includegraphics[width=0.48\textwidth,height=0.4\textwidth]
%%    {Ga10As10H30-log-ev.eps}
%%  }
%\caption{spectral distribution $\lambda(A)$}
%\label{fig:spectral}
%\end{figure}

\subsection{Comparison between RR and ARR}
\label{sec:comp-ARR}
We first evaluate the performance difference between ARR and RR for both
\mpm~ and \gn.   Table \ref{tab:ARR} gives results for computing both $k$ 
largest and smallest eigenpairs on the first six matrices in Table \ref{tab:matrix} 
to the accuracy of $\tol=10^{-12}$.  We note that RR and ARR
correspond to $p=0$ and $p>0$, respectively, in Algorithm \ref{alg:Arrabit2}. 
In order to differentiate the effect of changing $p$ from that of changing the
polynomial degree, we also test a variant of Algorithm~\ref{alg:Arrabit2} with 
a fixed polynomial degree at $d=8$ (by skipping line 34). 
In Table \ref{tab:ARR}, ``$\maxres$'' denotes the maximum relative 
residual norm in \eqref{eq:out-stop},  ``time'' is the runtime measured 
in seconds, ``RR'' is the total number of the outer iterations, i.e.,
the total number of the RR or ARR calls made (excluding the one
called in preprocessing for estimating $\lambda_{k+q}$), 
and ``$p$'' and ``$d$'' are the number of augmentation blocks 
and the polynomial degree, respectively, used at the final outer iteration. 
In addition, on the matrices cfd1 and finance we plot the (outer) 
iteration history of $\maxres$ in 
Figures \ref{fig:resi-ARR-largest} and \ref{fig:resi-ARR-smallest}
for computing $k$ largest and smallest eigenpairs, respectively. 

The following observations can be drawn from the table and figures.
\begin{itemize}

\item The performances of \mpm~ and \gn~ are similar.
  For both of them, ARR can accelerate convergence, reduce the number 
  of outer iterations needed, and improve the accuracy, often to a great extent. 
  
\item The scheme of adaptive polynomial degree generally works better
  than a fixed polynomial degree. A more detailed look at the effect of
  polynomial degrees is presented in Section~\ref{sec:comp-deg}. 

\item The default value $p=1$ for the number of augmentation blocks 
  in ARR is generally kept unchanged 
  (recall that it can be increased by the algorithm).

\item The total number of ARR called is mostly very small, especially in the 
  cases where the adaptive polynomial degree scheme is used and the $k$ 
  largest eigenpairs are computed (which tend to be easier than the $k$ 
  smallest ones).  We observe from Figure \ref{fig:resi-ARR-largest} 
  that in several cases a single ARR is sufficient to reach the accuracy 
  of $\tol$=1e-6 (even of $\tol$=1e-12 in one case).
  
\end{itemize}

\begin{table}[htbp]
  \setlength{\tabcolsep}{2pt}
  \centering
  \caption{Comparison results between RR and ARR with tol=1e-12} 
  \label{tab:ARR}
{\footnotesize
  \begin{tabular}{rrrr rrrr rrrr rrrr rrrr}
\hline
& \multicolumn{4}{c}{\mpm~ with RR} &&
 \multicolumn{4}{c}{\mpm~ with ARR} && \multicolumn{4}{c}{\gn~ with RR}  &&
 \multicolumn{4}{c}{\gn~ with ARR}  \\
 \cline{2-5} \cline{7-10} \cline{12-15} \cline{17-20}
\input{tab-ARR.tex}

\end{tabular}
}
\end{table}

\begin{figure}[htb]
\centering
 \hfill
    \subfigure[matrix: cfd1]{
    \includegraphics[width=0.48\textwidth,height=0.4\textwidth]
    {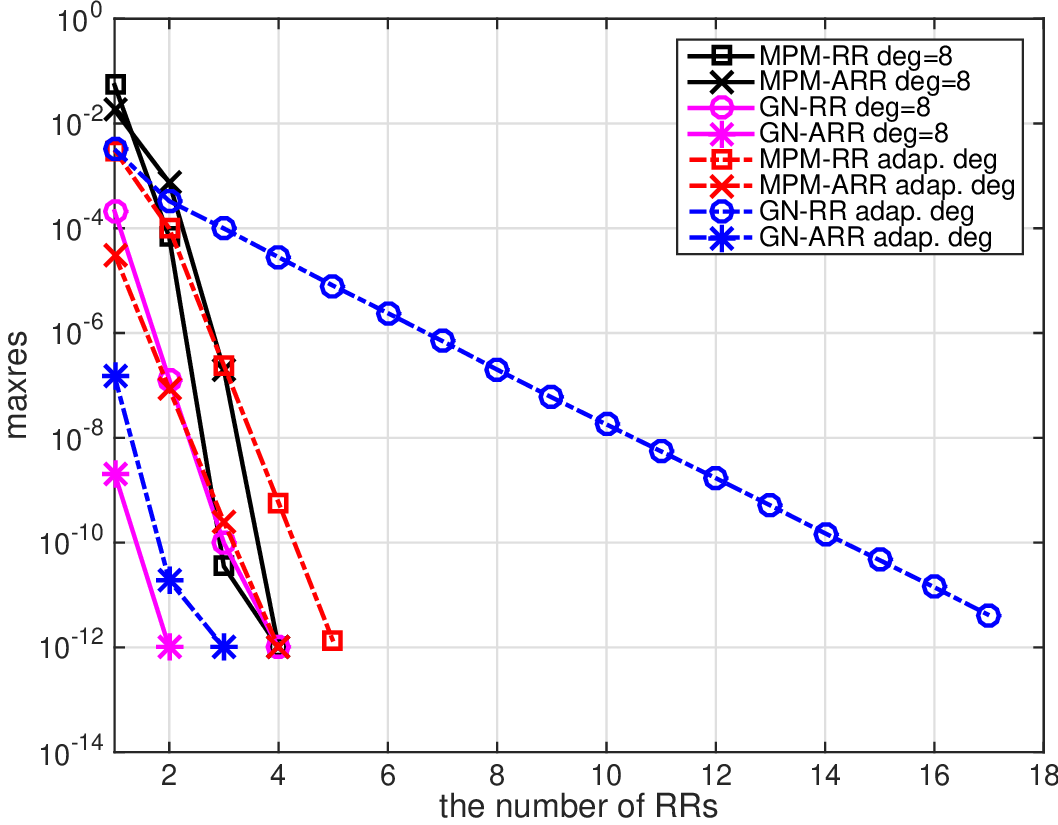}
  }
  \hfill
    \subfigure[matrix: finance]{
    \includegraphics[width=0.48\textwidth,height=0.4\textwidth]
    {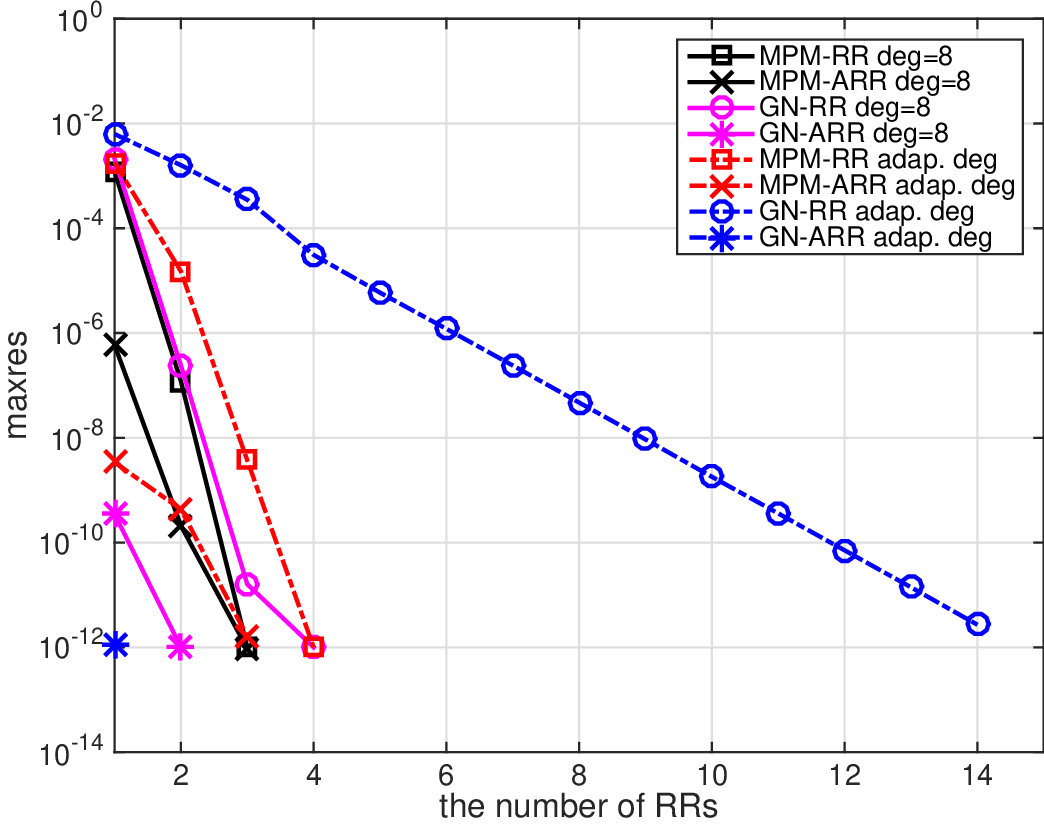}
  }
\caption{ARR vs RR: Iteration history of $\maxres$ for computing $k$ largest
eigenpairs}
\label{fig:resi-ARR-largest}

%\end{figure}
\vspace{.75cm}
%\begin{figure}

\centering
 \hfill
    \subfigure[matrix: cfd1]{
    \includegraphics[width=0.48\textwidth,height=0.4\textwidth]
    {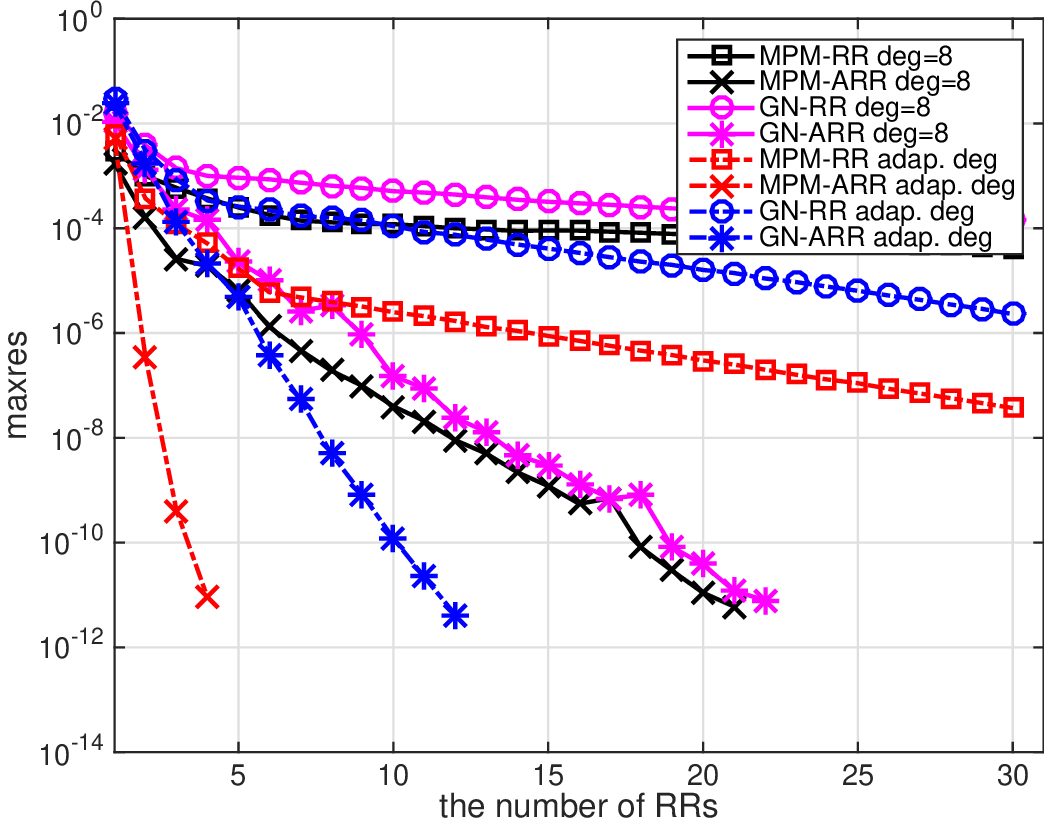}
  }
  \hfill
    \subfigure[matrix: finance]{
    \includegraphics[width=0.48\textwidth,height=0.4\textwidth]
    {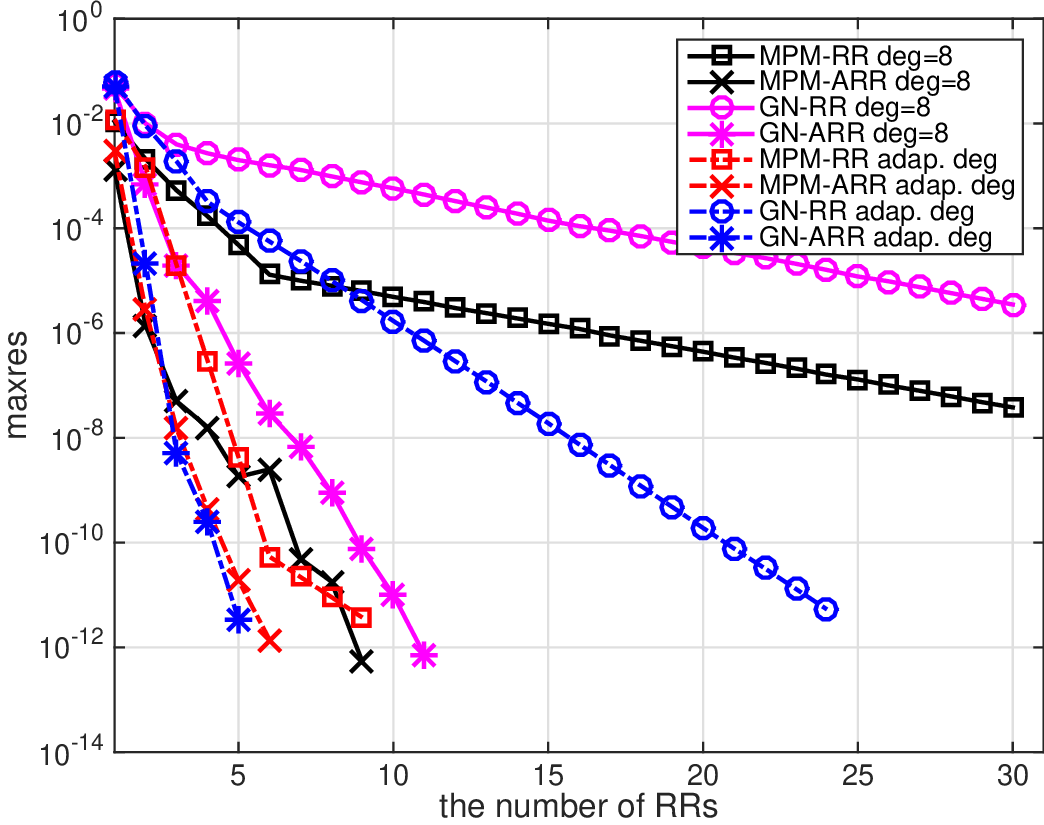}
  }
\caption{ARR vs RR: Iteration history of $\maxres$ for computing $k$ smallest
eigenpairs}
\label{fig:resi-ARR-smallest}
\end{figure}

\subsection{Comparison on Polynomials}
\label{sec:comp-deg}

We next examine the effect of polynomial degrees on the convergence behavior 
of \mpm~ and \gn, again on the first six matrices in Table \ref{tab:matrix}. 
We compare two schemes: the first is to use a fix degree among $\{4, 8, 15\}$ 
and skip line 34 of Algorithm~\ref{alg:Arrabit2}, and the second is the adaptive
scheme in Algorithm~\ref{alg:Arrabit2}.  The computational results are 
summarized in Table \ref{tab:polydeg}. We also plot the iteration history of 
$\maxres$, for computing both $k$ largest and smallest eigenpairs on 
the matrices cfd1 and finance in
Figures \ref{fig:resi-polydeg-largest} and \ref{fig:resi-polydeg-smallest},
respectively.  The numerical results lead to the following observations:
\begin{itemize}

\item 
Again the performances of \mpm~ and \gn~ are similar, and
the default value $p=1$ for augmentation is mostly unchanged.

\item
In general, the number of outer iterations is decreased as the 
polynomial degree is increased, but the runtime time is not necessarily 
reduced because of the extra cost in using higher-degree polynomials.
Overall, our adaptive strategy seems to have achieved a reasonable 
balance.

\item 
With fixed polynomial degrees, in a small number of test case
\mpm~ and \gn~  fail to reach the required accuracy.
\end{itemize}
 
\begin{figure}[htb]
\centering
  \hfill
      \subfigure[matrix: cfd1]{
    \includegraphics[width=0.48\textwidth,height=0.4\textwidth]
    {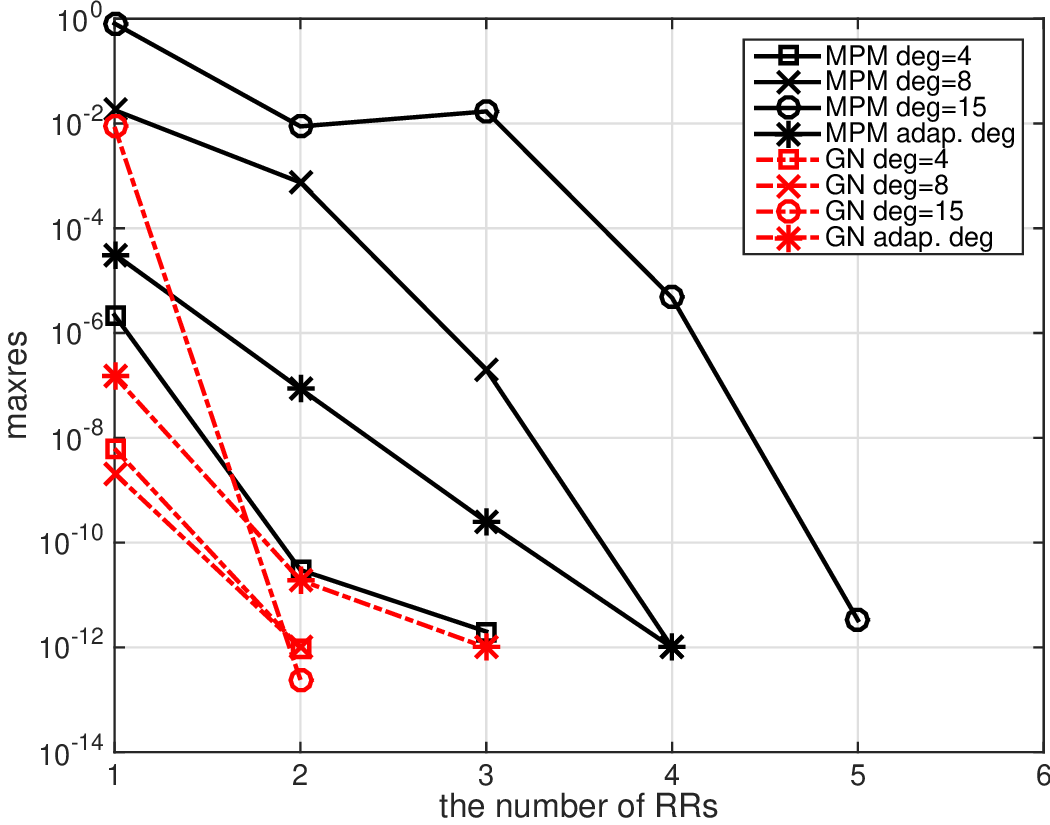}
  }
  \hfill
    \subfigure[matrix: finance]{
    \includegraphics[width=0.48\textwidth,height=0.4\textwidth]
    {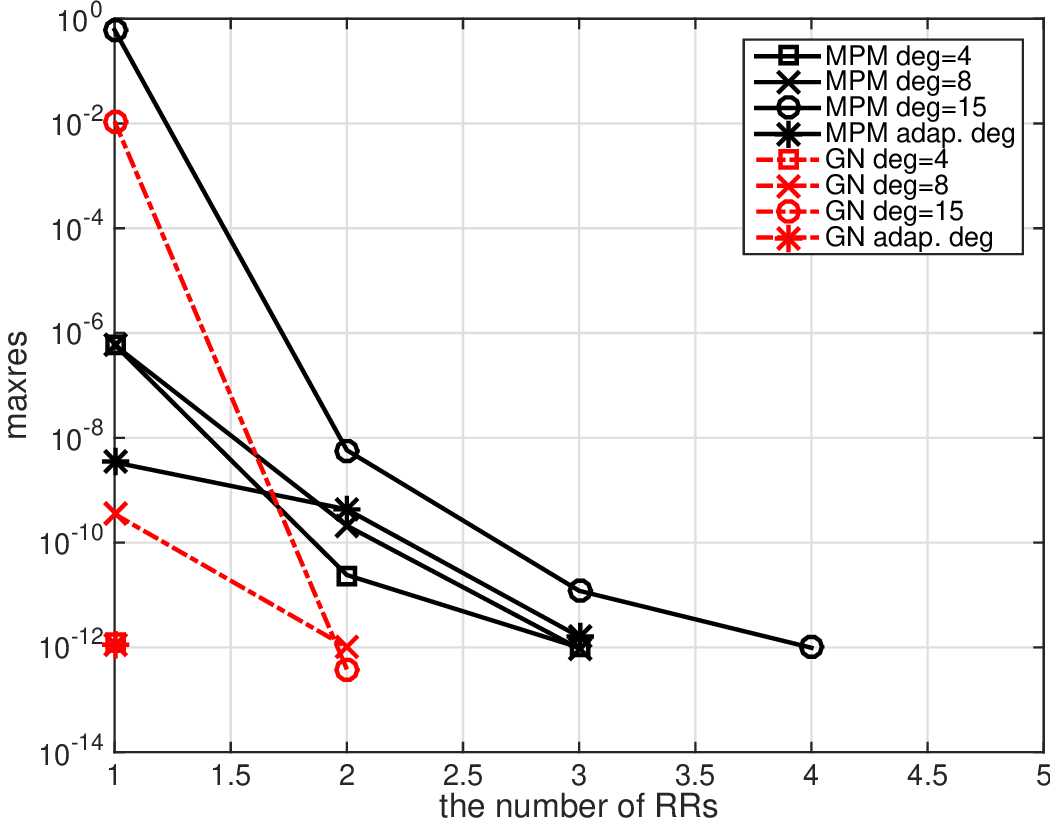}
  }
\caption{ARR: Iteration history of $\maxres$ for computing $k$ largest
eigenpairs using different polynomial degrees}
\label{fig:resi-polydeg-largest}

%\end{figure}
%\vspace{.75cm}
%\begin{figure}[htb]

\centering
 \subfigure[matrix: cfd1]{
    \includegraphics[width=0.48\textwidth,height=0.4\textwidth]
    {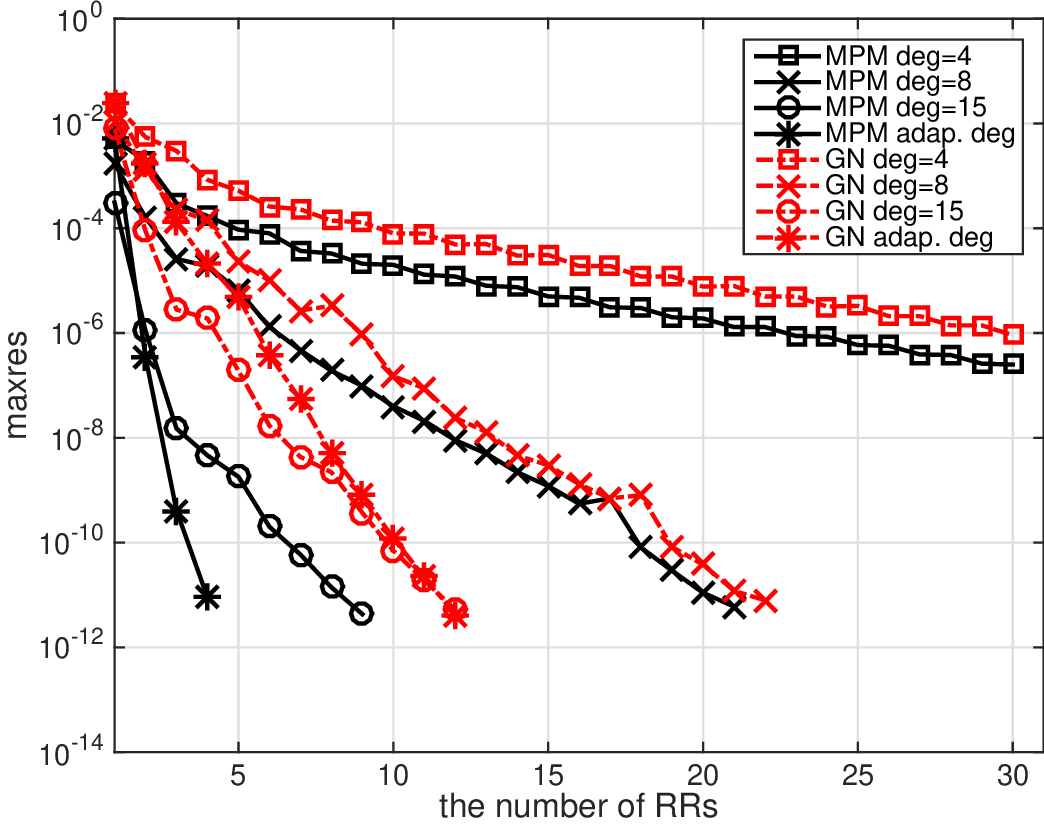}
  }
  \hfill
    \subfigure[matrix: finance]{
    \includegraphics[width=0.48\textwidth,height=0.4\textwidth]
    {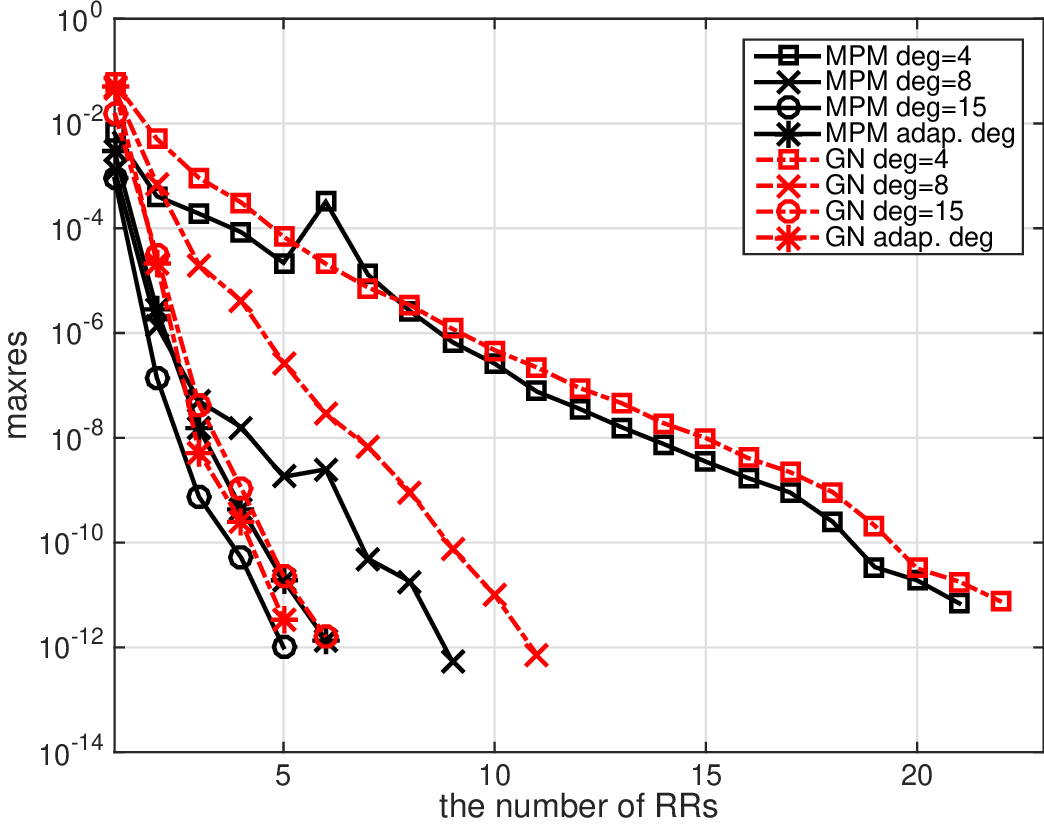}
  }
\caption{ARR: Iteration history of $\maxres$ for computing $k$ smallest
eigenpairs using different polynomial degrees}
\label{fig:resi-polydeg-smallest}
\end{figure}

\begin{table}[htbp]
  \setlength{\tabcolsep}{2pt}
  \centering
  \caption{Comparison results of different polynomial degrees on tol=1e-12} 
  \label{tab:polydeg}
{\footnotesize
  \begin{tabular}{rrrr rrrr rrrr rrrr rrrr}
\hline
& \multicolumn{4}{c}{deg=4} &&
 \multicolumn{4}{c}{deg=8} && \multicolumn{4}{c}{deg=15}  &&
 \multicolumn{4}{c}{adaptive deg}  \\
 \cline{2-5} \cline{7-10} \cline{12-15} \cline{17-20}

\input{tab-deg.tex}

\end{tabular}
  }
\end{table}

Finally, we compare the performance of Algorithm \ref{alg:Arrabit2} either using
Chebyshev interpolates defined in \eqref{eq:chebfun-f} or the Chebyshev
polynomials defined in \eqref{eq:cheby-poly} on the first six matrices in 
Table \ref{tab:matrix}.   The comparison results are given in Table~
\ref{tab:comp-cheby}.   Even though both types of polynomials work well
on these six problems, some performance differences are still observable 
in favor of our polynomials. % \revise{(keep this part or not?)}

\begin{table}[htbp]
  \setlength{\tabcolsep}{2pt}
  \centering
  \caption{Comparison results on Chebyshev interpolates in \eqref{eq:chebfun-f} 
  and Chebyshev polynomials in \eqref{eq:cheby-poly}} 
  \label{tab:comp-cheby}
{\footnotesize
  \begin{tabular}{rrrr rrrr rrrr rrrr rrrr}
\hline
& \multicolumn{4}{c}{\mpm~ } &&
 \multicolumn{4}{c}{\mpm, Cheb. poly.} && \multicolumn{4}{c}{\gn}  &&
 \multicolumn{4}{c}{\gn, Cheb. poly.}  \\
 \cline{2-5} \cline{7-10} \cline{12-15} \cline{17-20}

%%%
\input{tab-stat-cheby.tex}

\end{tabular}
}
\end{table}

%%%%%%%
\subsection{Comparison with \arpack~and \feast} \label{sec:comp-other}

We now compare \mpm~ and \gn~ with  \eigs~ and \feast~ for computing 
both $k$ largest and smallest eigenpairs for all sixteen test matrices 
presented in Tables \ref{tab:matrix} (which also lists the $k$ values).  
Computational results are summarized in Tables \ref{tab:largest} 
and \ref{tab:smallest}, where ``SpMV'' denotes the total number of SpMVs,
counting each operation $AX \in \R^{n\times k}$ as $k$ SpMVs.

In addition, the speedup with respect to the benchmark time of \eigs~ is
measured by the quantity $\log_{2}(\mathrm{time}_{\eigs}/\mathrm{time})$,
as shown in Figures \ref{fig:comp-time-largest} and \ref{fig:comp-time-smallest}
where a positive bar represents a ``speedup'' and a negative one 
a ``slowdown''.   In these two figures, matrices are ordered
from left to right in ascending order of the solution time used by
\eigs; that is, when moving from the left towards the right, problems 
become progressively more and more time-consuming for \eigs~ to solve.
A quick glance at the figures tells us that \mpm~and \gn~ provide clear 
speedups over \eigs~ on most problems, especially on the more time-consuming
problems towards the right. For example, \mpm~and \gn~ deliver
a speedup of about 4 times on each of the seven most time-consuming 
problems in Figure \ref{fig:comp-time-largest}(a), and a speedup of about 
10 times on the most time-consuming problem Ga41As41H72 in 
Figure \ref{fig:comp-time-smallest}(a).  On the other hand, compared to
\eigs, \feast's timing profile looks volatile with both big ``speedups'' and  
``slowdowns''.

%%%%%%%%%%%%
\begin{figure}[htb!]
\centering
 \hfill
 \subfigure[$\tol=10^{-6}$]{
    \includegraphics[width=0.48\textwidth,height=0.4\textwidth]
    {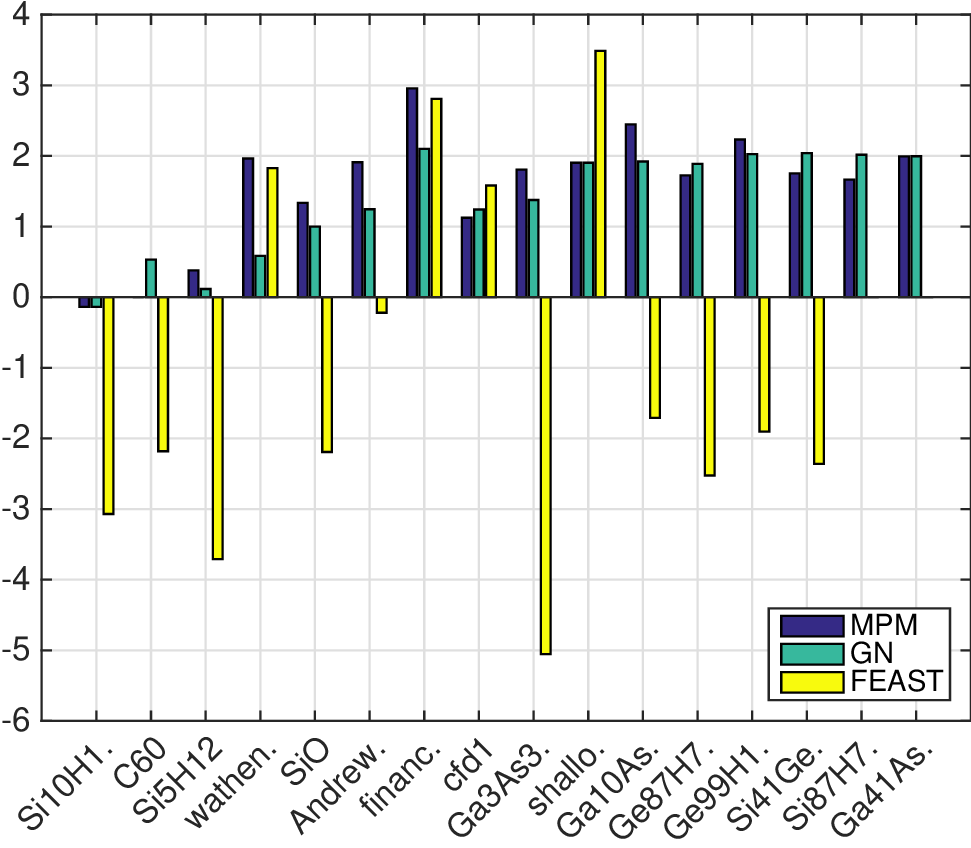}
  }
  \hfill
    \subfigure[$\tol=10^{-12}$]{
    \includegraphics[width=0.48\textwidth,height=0.4\textwidth]
    {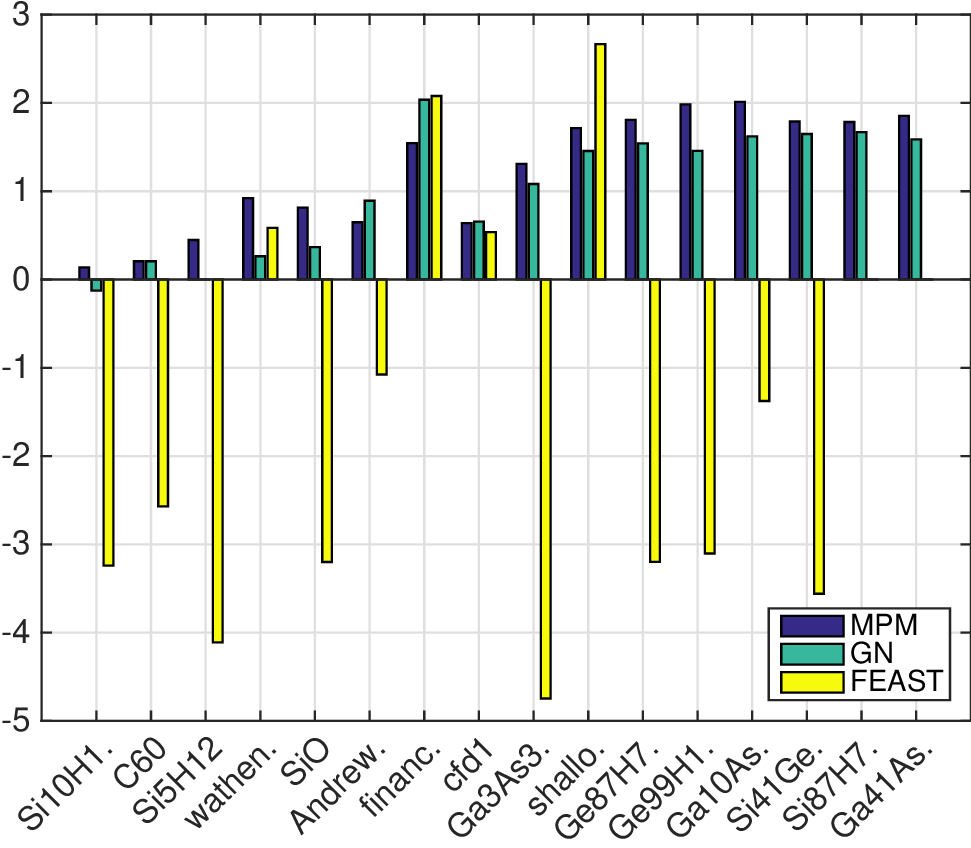}
  }
\caption{Speedup to \eigs: $\log_{2}(\mathrm{time}_{\eigs}/\mathrm{time})$ on
computing $k$ largest eigenpairs}
\label{fig:comp-time-largest}

%\end{figure}
%\vspace{.75cm}
%\begin{figure}[htb]

\centering
 \hfill
    \subfigure[$\tol=10^{-6}$]{
    \includegraphics[width=0.48\textwidth,height=0.4\textwidth]
    {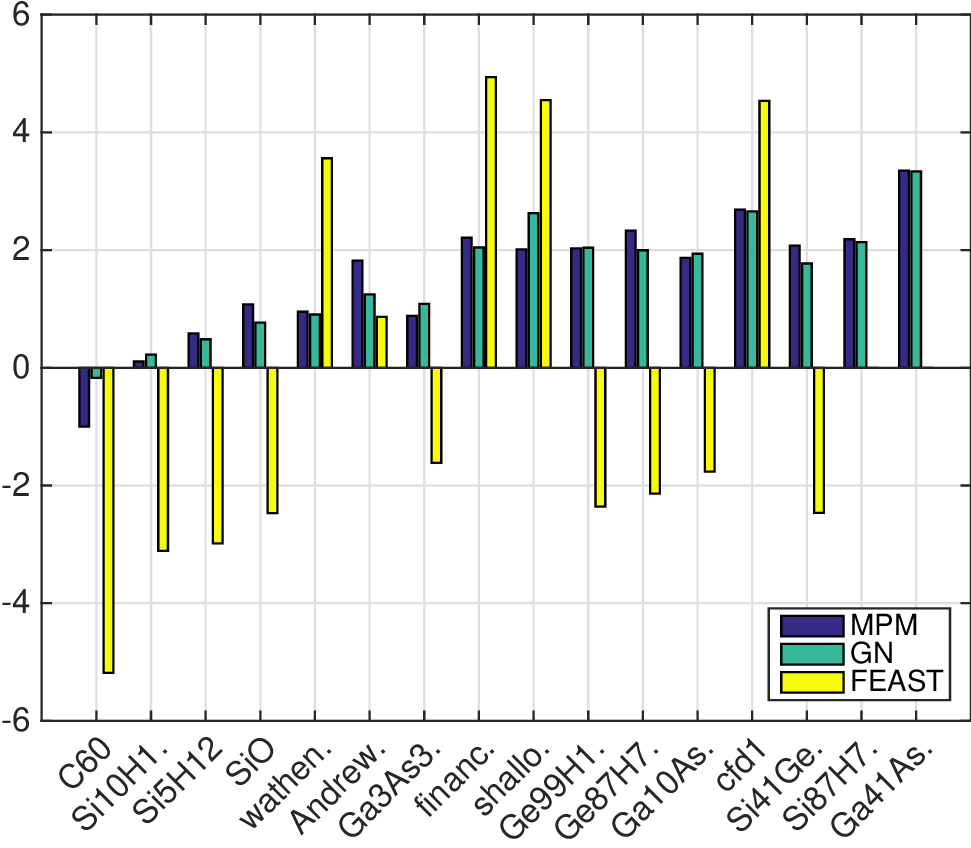}
  }
  \hfill
    \subfigure[$\tol=10^{-12}$]{
    \includegraphics[width=0.48\textwidth,height=0.4\textwidth]
    {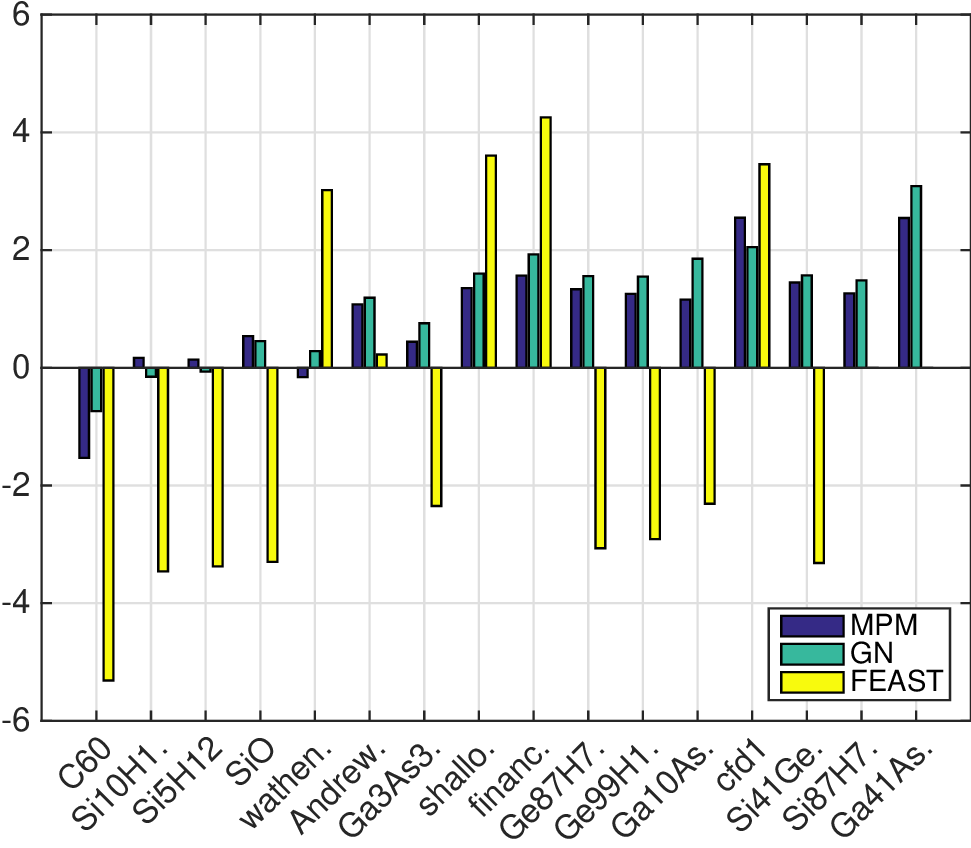}
  }
\caption{Speedup to \eigs: $\log_{2}(\mathrm{time}_{\eigs}/\mathrm{time})$ on
computing $k$ smallest eigenpairs}
\label{fig:comp-time-smallest}
\end{figure}

%%%%%%%%%%
\begin{table}[htbp]
  \centering
    \setlength{\tabcolsep}{1.2pt}
  \caption{Comparison results on computing $k$ largest eigenpairs} 
  \label{tab:largest}
{\footnotesize
  \begin{tabular}{r|rrr rrrr rrrr rrrr rr}
  %\begin{tabular}{|r|r|r|r| r|r|r|r| r|r|r|r| r|r|r|r| r|r|}
\hline
& \multicolumn{3}{c}{\eigs} &&
 \multicolumn{3}{c}{\feast} && \multicolumn{4}{c}{\mpm}  &&
 \multicolumn{4}{c}{\gn}  \\
\cline{2-4} \cline{6-8} \cline{10-13} \cline{15-18}

\input{tab-stat-all-maxeig.tex}

\end{tabular}
  }
\end{table}

\begin{table}[htbp]
  \centering
   \setlength{\tabcolsep}{1.2pt}
  \caption{Comparison results on computing $k$ smallest eigenpairs} 
  \label{tab:smallest}
{\footnotesize

  \begin{tabular}{r|rrr rrrr rrrr rrrr rr}
  %\begin{tabular}{|r|r|r|r| r|r|r|r| r|r|r|r| r|r|r|r| r|r|}
\hline
& \multicolumn{3}{c}{\eigs} &&
 \multicolumn{3}{c}{\feast} && \multicolumn{4}{c}{\mpm}  &&
 \multicolumn{4}{c}{\gn}  \\
\cline{2-4} \cline{6-8} \cline{10-13} \cline{15-18}

\input{tab-stat-all-mineig.tex}

\end{tabular}
  }
\end{table}

%%%%%%%%%%%%
\begin{figure}[htb!]
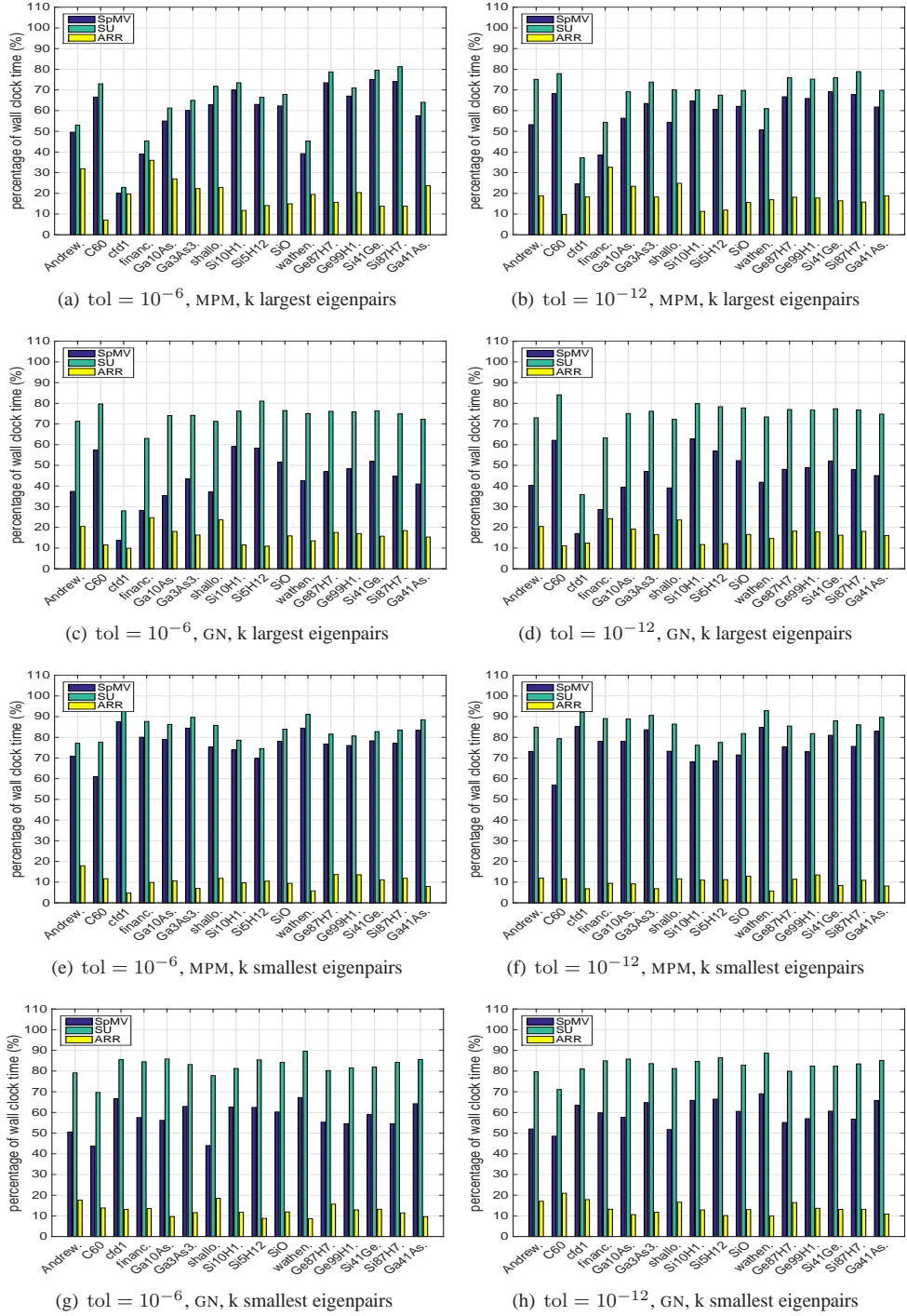

\centering
 \hfill
 \subfigure[$\tol=10^{-6}$, \mpm,  k largest eigenpairs]{
    \includegraphics[width=0.48\textwidth,height=0.3\textwidth]
    {cpu-maxeig-tol6-mpm.eps}
  }
  \hfill
    \subfigure[$\tol=10^{-12}$, \mpm,  k largest eigenpairs]{
    \includegraphics[width=0.48\textwidth,height=0.3\textwidth]
    {cpu-maxeig-tol12-mpm.eps}
  }

 \hfill
    \subfigure[$\tol=10^{-6}$, \gn,  k largest eigenpairs]{
    \includegraphics[width=0.48\textwidth,height=0.3\textwidth]
    {cpu-maxeig-tol6-slrp.eps}
  }
  \hfill
    \subfigure[$\tol=10^{-12}$, \gn, k largest eigenpairs]{
    \includegraphics[width=0.48\textwidth,height=0.3\textwidth]
    {cpu-maxeig-tol12-slrp.eps}
  }

 \hfill
 \subfigure[$\tol=10^{-6}$, \mpm,  k smallest eigenpairs]{
    \includegraphics[width=0.48\textwidth,height=0.3\textwidth]
    {cpu-mineig-tol6-mpm.eps}
  }
  \hfill
    \subfigure[$\tol=10^{-12}$, \mpm,  k smallest eigenpairs]{
    \includegraphics[width=0.48\textwidth,height=0.3\textwidth]
    {cpu-mineig-tol12-mpm.eps}
  }
 \hfill
    \subfigure[$\tol=10^{-6}$, \gn, k smallest eigenpairs]{
    \includegraphics[width=0.48\textwidth,height=0.3\textwidth]
    {cpu-mineig-tol6-slrp.eps}
  }
  \hfill
    \subfigure[$\tol=10^{-12}$, \gn,  k smallest eigenpairs]{
    \includegraphics[width=0.48\textwidth,height=0.3\textwidth]
    {cpu-mineig-tol12-slrp.eps}
  }
\caption{A comparison of timing profile among SpMV, SU and ARR}
\label{fig:cpu-time-percentage}
\end{figure}

The benchmark solver \eigs~ usually, though not always, returns solutions more 
accurate than what is requested by the tolerance value.  In particular, for 
$\tol=10^{-6}$ the accuracy of \eigs~ solutions often reach the order of 
$O(10^{-12})$.  This is due to the fact that \eigs~ need to maintain a 
high working accuracy to ensure proper convergence.

As is observed previously, it is often more time-consuming for \eigs, \mpm~ 
and \gn~ to compute $k$ smallest eigenpairs than $k$ largest ones
on many test matrices.   By examining the spectra of the matrices such as
cfd1 and finance, we believe that this phenomenon is attributable to the 
property that these matrices tend to have a flatter end on the left end of their
spectra.  On the other hand, the behavior of \feast~ appears less affected 
by this property but more by sparsity patterns (see below).

Concerning the performance of \feast, we make the following observations.
\begin{itemize}

\item 
\feast~ solves most problems successfully but fails to correctly solve a few 
cases.  When computing $k$ largest eigenvalues for the matrix Ga10As10H30 
\feast~ returns the warning: 
``No eigenvalue has been found in the proposed search interval''.  
On matrix Ga3As3H12, it seems to exit normally with the output messages 
``Eigensolvers have successfully converged'', but the subsequently computed
maximum relative residual norm in \eqref{eq:out-stop} is way too large at $0.29$. 
On matrices Ga41As41H72 and Si87H76, when computing either $k$ largest or
smallest eigenpairs, \feast~ terminates abnormally after spending a long 
computing time, with the message: 
``Eigensolvers ERROR: Problem from Inner Linear System Solver''.
By examining the density of Cholesky factors for Ga41As41H72 and Si87H76
in Table \ref{tab:matrix}, we speculate that the abnormal termination most likely 
has to do with excessive memory demands encountered by the inner linear 
system solver in Intel Math Kernel Library.

\item  
For $\tol=10^{-12}$, \feast~ is the fastest in solving finance and shallow\_water1s 
for $k$ largest eigenpairs, and in solving cfd1, finance, shallow\_water1s and 
wathen100 for $k$ smallest eigenpairs.  On the other hand, \feast~ can be 
significantly slower than others on matrices such as Ga10As10H30, Ga3As3H12, 
Ge87H76, Ge99H100, Si41Ge41H72, Si87H76 and Ga41As41H72. 
The performance of \feast~ can be at least partly explained from the density of
Cholesky factors $L$ shown in Table \ref{tab:matrix}, since \feast~ uses a direct
linear solver in Intel Math Kernel Library to compute factorizations of matrices 
of the form $(\phi_l I - A)$  in \eqref{eq:feast-rho}.  
We can clearly see the correlation that \feast~ is fast when the density of the 
Cholesky factor is low and Cholesky factorization is fast.

\end{itemize}

With regard to the performance of \mpm~ and \gn, 
we make the following observations.
\begin{itemize}

\item \mpm~ and \gn~ both attain the required accuracy on all test 
problems, and they often return smaller residual errors than what is
required by $\tol$.  Generally speaking, the two variants perform 
quite similarly in terms of both accuracy and timing.
  
\item \mpm~ and \gn~  maintain a clear speed advantage over
\feast~ in most tested cases. They are faster than \feast~ when 
either factorizations of shifted $A$ are expensive, or when spectral 
distributions have a favorable decay (for example, on cfd1 for 
computing $k$ largest eigenpairs).

\item \mpm~ and \gn~  also maintain an overall speed advantage 
over \eigs, especially on those problems more time-consuming for 
\eigs~ (towards the right end of 
Figures \ref{fig:comp-time-largest} and \ref{fig:comp-time-smallest}).
They are faster in spite of taking considerably more matrix-vector 
multiplications than \eigs, as can be seen from Tables \ref{tab:largest} 
and \ref{tab:smallest}, thanks to the benefits of relying on high-concurrency 
operations on many-core computers. 

\item \mpm~ and \gn~ generally require a smaller number ARR calls, 
often only two or three when computing $k$ largest eigenpairs. 
%fewer than or similar to what is required by \feast~ in most cases.  
In quite a number of cases (for example, on finance and wathen100 
for \mpm~and so on), only a single ARR projection is taken which is 
absolutely optimal in order to extract approximate eigenpairs.

\item The number of augmentation blocks used by \mpm~ and \gn~ 
is usually $1$, and the final polynomial degree never reaches the 
maximum degree $15$ except on cfd1, finance and wathen100 when 
computing $k$ smallest eigenpairs. 

\end{itemize}

In Figure \ref{fig:cpu-time-percentage}, we plot runtimes of three categories:
SpMV (i.e., $AX$), SU (lines 10 to 22 of Algorithm \ref{alg:Arrabit2}) and 
ARR (lines 23 to 27 of Algorithm \ref{alg:Arrabit2}). In particular, SpMVs 
are called in both SU and ARR, but overwhelmingly in the former.   
These are the major computational components of \mpm~ and \gn. 
The runtime of each category is measured in the percentage 
of wall-clock time spent in that category over the total wall-clock time.  
We can see, especially from the time-consuming problems on the right, 
that (i) the time of SU dominates that of RR, 
and (ii) the time of SpMVs, always done in batch of $k+q$,
dominates the entire computation in almost all cases.
These trends are much more pronounced (a) for 
\mpm~ than for \gn~ 
(recall that \gn~ requires to solve $k \times k$ linear systems);
and (b) for computing $k$ smallest eigenpairs than for computing 
$k$ largest ones (recall that the former is generally more difficult).
These runtime profiles are favorable to parallel scalability since 
$AX$ operations possess high concurrency for relatively large $k$.

In the final set of experiments, we examine the solvers' scalability with
respect to $k$.  We apply the solvers to matrices cfd1 and Ge87H76,
with $\tol=10^{-12}$, and vary $k$ from $100, 200$ up to $1200$ 
with increment $200$ (there are exceptions for \feast).  
The resulting solution times are plotted in 
Figures \ref{fig:cpu-vark-cfd1} and \ref{fig:cpu-vark-Ge87H76}.
%, while detailed statistics are summarized in 
%Tables \ref{tab:cfd1-varyk} and \ref{tab:Ge87H76-varyk}.  
In both figures, the slopes of the time curves confirm that the three block 
algorithms, \feast, \mpm~ and \gn, clearly scale better with respect 
to $k$ than the Krylov subspace algorithm \eigs.  Although \eigs~can 
be the fastest for $k$ small, its solution time increases at a faster 
pace than the block methods as $k$ increases.

Among the block algorithms themselves, all three provide
comparable performances on cfd1 when computing the 
$k$ largest eigenpairs, while \feast~ is the fastest when 
computing $k$ smallest eigenpairs.   
On Ge87H76, which has a rather dense Cholesky factor, 
\feast~ is much slower in all runs up to $k=1000$ 
(runs for $k>1000$ are skipped to save time).  
%For detailed statistics, see Tables \ref{tab:cfd1-varyk} and \ref{tab:Ge87H76-varyk}.

%%%%%%%%%
\begin{figure}[htb!]
\centering
  \subfigure[$k$ largest eigenpairs, $tol=10^{-12}$]{
    \includegraphics[width=0.48\textwidth,height=0.4\textwidth]
    {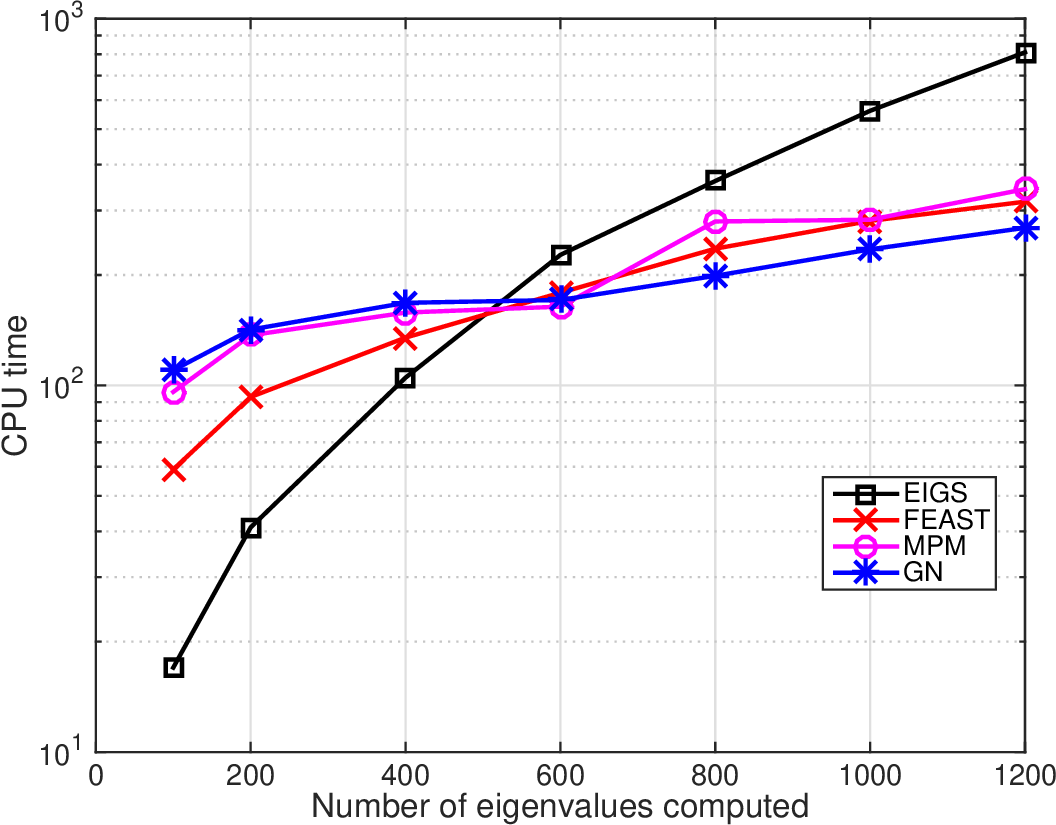}
  }
  \hfill
    \subfigure[$k$ smallest eigenpairs, $tol=10^{-12}$]{
    \includegraphics[width=0.48\textwidth,height=0.4\textwidth]
    {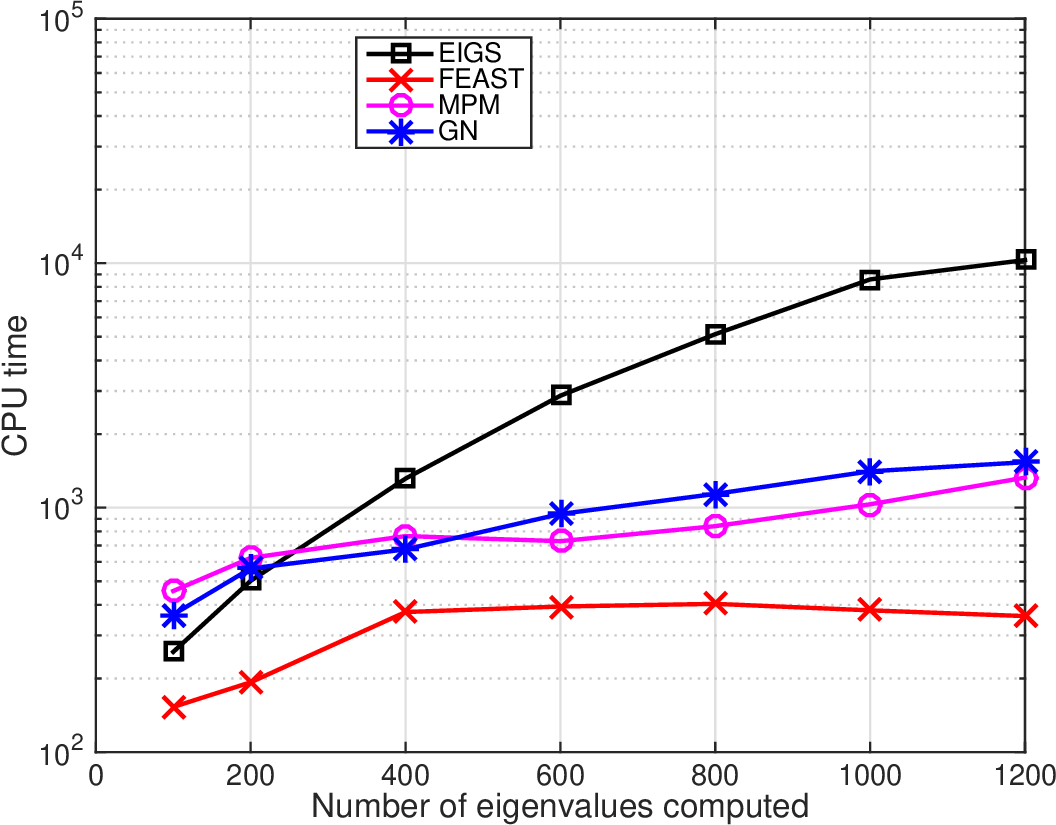}
  }
\caption{Comparison results of solution time for computing $k$ eigenpairs  of
the matrix cfd1}
\label{fig:cpu-vark-cfd1}

%\end{figure}
\vspace{.75cm}
%\begin{figure}[htb]

\centering
  \subfigure[$k$ largest eigenpairs, $tol=10^{-12}$]{
    \includegraphics[width=0.48\textwidth,height=0.4\textwidth]
    {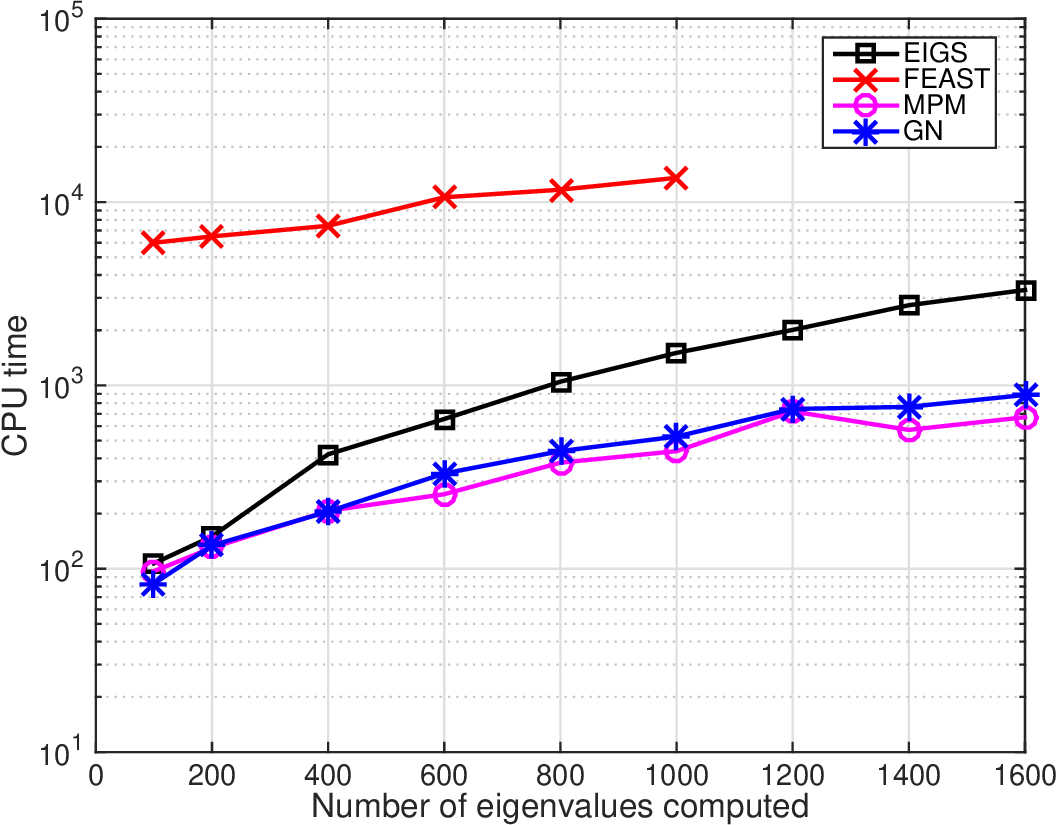}
  }
  \hfill
    \subfigure[$k$ smallest eigenpairs, $tol=10^{-12}$]{
    \includegraphics[width=0.48\textwidth,height=0.4\textwidth]
    {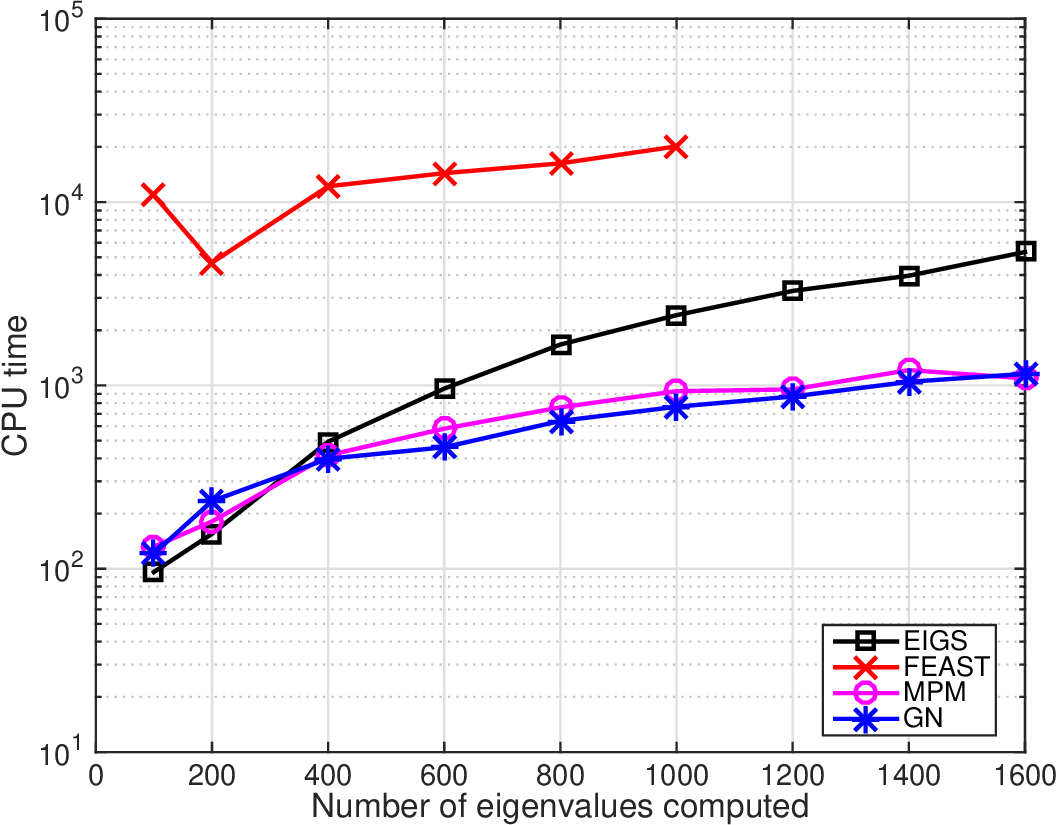}
  }
\caption{Comparison results on solution time for computing $k$ eigenpairs  of
the matrix Ge87H76.}
\label{fig:cpu-vark-Ge87H76}
\end{figure}

%%%%%%%%
%\begin{table}[htbp]
%  \centering
%   \setlength{\tabcolsep}{1.5pt}
%  \caption{Comparison results on cfd1 for varying $k$} 
%  \label{tab:cfd1-varyk}
%{\footnotesize
%\begin{tabular}{r|rrr rrrr rrrr rrrr rr}
%\hline
%& \multicolumn{3}{c}{\eigs} &&
% \multicolumn{3}{c}{\feast} && \multicolumn{4}{c}{\mpm}  &&
% \multicolumn{4}{c}{\gn}  \\
%\cline{2-4} \cline{6-8} \cline{10-13} \cline{15-18}
%%%%
%\input{FIG/cfd1-tab-varyk.tex}
%%%%
%\end{tabular}
%  }
%\end{table}
%
%
%\begin{table}[htbp]
%  \centering
%   \setlength{\tabcolsep}{1.5pt}
%  \caption{Comparison results on Ge87H76 for varying $k$
%  (\feast~ runs skipped for $k>1000$)} 
%  \label{tab:Ge87H76-varyk}
%{\footnotesize
%\begin{tabular}{r|rrr rrrr rrrr rrrr rr}
%\hline
%& \multicolumn{3}{c}{\eigs} &&
% \multicolumn{3}{c}{\feast} && \multicolumn{4}{c}{\mpm}  &&
% \multicolumn{4}{c}{\gn}  \\
%\cline{2-4} \cline{6-8} \cline{10-13} \cline{15-18}
%%%%
%\input{FIG/Ge87H76-tab-varyk.tex}
%%%%
%\end{tabular}
%  }
%\end{table}

\section{Concluding Remarks}
\label{sec:con}

The goal of this paper is to construct a block algorithm of high scalability 
suitable for computing relatively large numbers of exterior eigenpairs for
really large-scale matrices on modern computers.    Our strategy is simple: 
to reduce as much as possible the number of RR calls (Rayleigh-Ritz 
projections) or, in other words, to shift as much as possible computation 
burdens to SU (subspace update) steps.
This strategy is based on the following considerations.
RR steps perform small dense eigenvalue decompositions, as well as 
basis orthogonalizations, thus possessing limited concurrency.  
On the other hand, SU steps can be accomplished by block 
operations like $A$ times $X$, thus more scalable. 

To reach for maximal concurrency, we choose the power iteration
for subspace updating (and also include a Gauss-Newton method 
to test the versatility of our construction).   It is well known that the
convergence of the power method can be intolerably slow, 
preventing it from being used to drive general-purpose eigensolvers.
Therefore, the key to success reduces to whether we could accelerate
the power method sufficiently and reliably to an extent that it can 
compete in speed with Krylov subspace methods in general.
In this work, such an acceleration is accomplished mainly through 
the use of three techniques: 
(1) an augmented Rayleigh-Ritz (ARR) procedure that can provably 
accelerate convergence under mild conditions;
(2) a set of easy-to-control, low-degree polynomial accelerators; and
(3) a bold stoping rule for SU steps that essentially allows an iterate 
matrix to become numerically rank-deficient.  
Of course, the success of our construction also depends greatly on 
a set of carefully integrated algorithmic details.  The resulting 
algorithm is named \arrabit, which uses $A$ only in matrix multiplications.

Numerical experiments in Matlab on sixteen test matrices from the UF 
Sparse Matrix Collection show, convincingly in our view, that the accuracy 
and efficiency of \arrabit~ is indeed competitive to start-of-the-art eigensolvers.
Exceeding our expectations, \arrabit~ can already provide multi-fold speedups
over the benchmark solver \eigs, without explicit code parallelization and 
without running on massively parallel machines, on difficult problems.
In particular,  it often only needs two or three, sometimes just one, ARR 
projections to reach a good solution accuracy.  

There are a number of future directions worth pursuing from this point on.  
For one thing, the robustness and efficiency of \arrabit~can be further 
enhanced by refining its construction and and tuning its parameters.  
Software development and an evaluation of its parallel scalability are 
certainly important.   
The prospective of extending the algorithm to non-Hermitian matrices 
and the generalized eigenvalue problem looks promising.  
Overall, we feel that the present work has laid a solid foundation 
for these and other future activities.

\section*{Acknowledgements} 
%The authors would like to thank Chao Yang for valuable discussions on eigenvalue computation.  
Most of the computational results were obtained at the National Energy Research Scientific Computing Center (NERSC), which is supported by the Director, Office of Advanced Scientific Computing Research of the U.S. Department of Energy under contract number DE-AC02-05CH11232.

% Reference
\bibliographystyle{siam}
\bibliography{SVD,chebfun,optimization}

\end{document}

%% file: tab-ARR.tex
  matrix & maxres   & time & RR & p/d && maxres & time  & RR & p/d &&  maxres  & time  & RR & p/d &&  maxres  & time  & RR & p/d \\ \hline

\multicolumn{20}{c}{ computing k largest eigpair by fix deg = 8} \\ \hline 
        Andrew. & 9.5e-13 &      191 &  4 &  1/ 8&& 1.9e-06 &      250 &  9 &  3/ 8&& 9.0e-12 &      174 &  6 &  1/ 8&& 9.9e-13 &      104 &  2 &  1/ 8 \\ 
            C60 & 4.0e-12 &       45 & 11 &  3/ 8&& 6.3e-12 &       12 &  3 &  1/ 8&& 7.5e-12 &       44 & 22 &  3/ 8&& 1.4e-12 &       16 &  5 &  1/ 8 \\ 
           cfd1 & 9.8e-13 &      381 &  4 &  1/ 8&& 1.0e-12 &      296 &  4 &  1/ 8&& 9.8e-13 &      294 &  4 &  1/ 8&& 9.9e-13 &      206 &  2 &  1/ 8 \\ 
        financ. & 9.9e-13 &      157 &  3 &  1/ 8&& 8.9e-13 &      151 &  3 &  1/ 8&& 1.0e-12 &      196 &  4 &  1/ 8&& 1.0e-12 &      141 &  2 &  1/ 8 \\ 
        Ga10As. & 3.5e-13 &     1218 & 22 &  3/ 8&& 9.9e-13 &     1483 &  8 &  2/ 8&& 6.1e-12 &      910 &  8 &  1/ 8&& 9.9e-13 &      448 &  3 &  1/ 8 \\ 
        Ga3As3. & 9.7e-13 &      467 &  6 &  1/ 8&& 9.8e-13 &      270 &  5 &  1/ 8&& 1.9e-12 &      307 &  8 &  1/ 8&& 9.4e-13 &      179 &  3 &  1/ 8 \\ 
 \hline

\multicolumn{20}{c}{ computing k largest eigpair  with adaptive polynomial degree} \\ \hline 
        Andrew. & 2.0e-11 &      337 &  9 &  3/ 5&& 8.8e-13 &      148 &  5 &  2/ 5&& 5.3e-12 &      319 & 17 &  3/ 5&& 1.0e-12 &      125 &  4 &  1/ 5 \\ 
            C60 & 8.7e-12 &       41 & 10 &  3/ 9&& 2.0e-12 &       13 &  3 &  1/ 9&& 4.2e-12 &       42 & 20 &  3/ 9&& 5.5e-12 &       13 &  3 &  1/ 9 \\ 
           cfd1 & 1.3e-12 &      441 &  5 &  1/ 3&& 9.8e-13 &      190 &  4 &  1/ 3&& 4.1e-12 &      482 & 17 &  3/ 3&& 9.9e-13 &      188 &  3 &  1/ 3 \\ 
        financ. & 9.9e-13 &      256 &  4 &  1/ 3&& 1.3e-12 &       97 &  3 &  2/ 3&& 2.7e-12 &      380 & 14 &  3/ 3&& 1.1e-12 &       69 &  1 &  1/ 3 \\ 
        Ga10As. & 4.7e-12 &     1199 &  6 &  1/ 5&& 9.6e-13 &      442 &  4 &  1/ 5&& 7.1e-12 &     1442 & 19 &  3/ 5&& 9.7e-13 &      580 &  4 &  1/ 6 \\ 
        Ga3As3. & 2.9e-12 &      473 &  7 &  2/ 5&& 1.7e-12 &      169 &  4 &  1/ 5&& 3.9e-12 &      494 & 17 &  3/ 5&& 1.7e-12 &      198 &  4 &  1/ 5 \\ 
 \hline

\multicolumn{20}{c}{ computing k smallest eigpair by fix deg = 8} \\ \hline 
        Andrew. & 4.2e-12 &      465 &  7 &  2/ 8&& 1.5e-13 &      219 &  6 &  2/ 8&& 7.2e-12 &      475 & 19 &  3/ 8&& 1.0e-12 &      199 &  5 &  1/ 8 \\ 
            C60 & 1.7e-12 &       30 &  9 &  3/ 8&& 6.8e-13 &       17 &  6 &  1/ 8&& 5.5e-12 &       24 & 13 &  3/ 8&& 6.7e-12 &       13 &  4 &  1/ 8 \\ 
           cfd1 & 4.1e-05 &     2870 & 30 &  3/ 8&& 6.0e-12 &     1543 & 21 &  3/ 8&& 1.5e-04 &     2505 & 30 &  3/ 8&& 7.9e-12 &     1394 & 22 &  3/ 8 \\ 
        financ. & 3.8e-08 &     1759 & 30 &  3/ 8&& 5.1e-13 &      700 &  9 &  3/ 8&& 3.5e-06 &     1651 & 30 &  3/ 8&& 7.2e-13 &      713 & 11 &  3/ 8 \\ 
        Ga10As. & 8.6e-10 &     2642 & 10 &  3/ 8&& 3.7e-12 &     1372 &  5 &  1/ 8&& 2.1e-02 &     1436 &  6 &  1/ 8&& 2.6e-12 &      961 &  4 &  1/ 8 \\ 
        Ga3As3. & 7.2e-12 &      964 & 11 &  3/ 8&& 2.7e-12 &      489 &  4 &  1/ 8&& 4.2e-12 &      994 & 24 &  3/ 8&& 9.9e-13 &      381 &  4 &  1/ 8 \\ 
 \hline

\multicolumn{20}{c}{ computing k smallest eigpair  with adaptive polynomial degree} \\ \hline 
        Andrew. & 7.3e-12 &      466 &  8 &  3/ 8&& 9.7e-13 &      200 &  4 &  1/ 8&& 8.9e-12 &      505 & 21 &  3/ 8&& 1.1e-12 &      185 &  5 &  1/ 8 \\ 
            C60 & 6.7e-12 &       38 &  9 &  3/ 7&& 2.8e-12 &       26 &  9 &  3/ 6&& 4.0e-12 &       31 & 23 &  3/ 6&& 9.2e-13 &       15 &  8 &  2/ 6 \\ 
           cfd1 & 3.7e-08 &     2869 & 30 &  3/15&& 8.9e-12 &      719 &  4 &  1/15&& 2.3e-06 &     2515 & 30 &  3/15&& 4.2e-12 &     1017 & 12 &  3/15 \\ 
        financ. & 3.7e-12 &     1391 &  9 &  3/15&& 1.4e-12 &      600 &  6 &  1/15&& 5.3e-12 &     1416 & 24 &  3/15&& 3.4e-12 &      467 &  5 &  1/15 \\ 
        Ga10As. & 4.5e-11 &     3261 & 12 &  3/ 8&& 1.1e-12 &     1558 &  6 &  1/ 8&& 2.9e-12 &     3681 & 24 &  3/ 8&& 4.0e-12 &      963 &  3 &  1/ 9 \\ 
        Ga3As3. & 5.9e-12 &     1046 &  8 &  3/ 9&& 9.9e-13 &      420 &  4 &  1/ 9&& 7.7e-12 &     1238 & 24 &  3/ 9&& 9.5e-13 &      338 &  5 &  1/ 9 \\ 
 \hline

%% file: tab-deg.tex
  matrix & maxres   & time & RR & p/d && maxres & time  & RR & p/d &&  maxres  & time  & RR & p/d &&  maxres  & time  & RR & p/d\\ \hline

\multicolumn{20}{c}{MPM  for k largest eigpair} \\ \hline 
        Andrew. & 1.1e-12 &      127 &  5 &  2/ 4&& 1.9e-06 &      250 &  9 &  3/ 8&& 4.3e-12 &      165 &  4 &  1/15&& 8.8e-13 &      148 &  5 &  2/ 5 \\ 
            C60 & 1.6e-12 &       18 &  6 &  3/ 4&& 6.3e-12 &       12 &  3 &  1/ 8&& 9.7e-13 &       24 &  3 &  2/15&& 2.0e-12 &       13 &  3 &  1/ 9 \\ 
           cfd1 & 1.8e-12 &      206 &  3 &  1/ 4&& 1.0e-12 &      296 &  4 &  1/ 8&& 2.8e-12 &      411 &  5 &  2/15&& 9.8e-13 &      190 &  4 &  1/ 3 \\ 
        financ. & 9.9e-13 &      102 &  3 &  1/ 4&& 8.9e-13 &      151 &  3 &  1/ 8&& 9.0e-13 &      175 &  4 &  1/15&& 1.3e-12 &       97 &  3 &  2/ 3 \\ 
        Ga10As. & 1.3e-12 &      906 &  8 &  2/ 4&& 9.9e-13 &     1483 &  8 &  2/ 8&& 2.8e-01 &     5908 &  6 &  1/15&& 9.6e-13 &      442 &  4 &  1/ 5 \\ 
        Ga3As3. & 7.6e-13 &      377 &  7 &  1/ 4&& 9.8e-13 &      270 &  5 &  1/ 8&& 2.8e-01 &     1483 &  6 &  1/15&& 1.7e-12 &      169 &  4 &  1/ 5 \\ 
 \hline 

\multicolumn{20}{c}{SLRP  for k largest eigpair} \\ \hline 
        Andrew. & 1.5e-12 &      116 &  4 &  1/ 4&& 9.9e-13 &      104 &  2 &  1/ 8&& 1.2e-13 &      187 &  2 &  1/15&& 1.0e-12 &      125 &  4 &  1/ 5 \\ 
            C60 & 1.5e-12 &       24 &  9 &  3/ 4&& 1.4e-12 &       16 &  5 &  1/ 8&& 7.1e-13 &       19 &  3 &  1/15&& 5.5e-12 &       13 &  3 &  1/ 9 \\ 
           cfd1 & 9.6e-13 &      185 &  2 &  1/ 4&& 9.9e-13 &      206 &  2 &  1/ 8&& 1.7e-13 &      324 &  2 &  1/15&& 9.9e-13 &      188 &  3 &  1/ 3 \\ 
        financ. & 1.2e-12 &       77 &  1 &  1/ 4&& 1.0e-12 &      141 &  2 &  1/ 8&& 2.7e-13 &      327 &  2 &  1/15&& 1.1e-12 &       69 &  1 &  1/ 3 \\ 
        Ga10As. & 5.9e-13 &      734 &  7 &  2/ 4&& 9.9e-13 &      448 &  3 &  1/ 8&& 2.9e-01 &     1122 &  6 &  1/15&& 9.7e-13 &      580 &  4 &  1/ 6 \\ 
        Ga3As3. & 8.4e-12 &      205 &  4 &  1/ 4&& 9.4e-13 &      179 &  3 &  1/ 8&& 6.4e-02 &      442 &  6 &  1/15&& 1.7e-12 &      198 &  4 &  1/ 5 \\ 
 \hline

\multicolumn{20}{c}{MPM  for k smallest eigpair} \\ \hline 
        Andrew. & 4.1e-13 &      247 &  9 &  3/ 4&& 1.5e-13 &      219 &  6 &  2/ 8&& 9.9e-13 &      448 &  5 &  1/15&& 9.7e-13 &      200 &  4 &  1/ 8 \\ 
            C60 & 1.6e-07 &       20 &  7 &  3/ 4&& 6.8e-13 &       17 &  6 &  1/ 8&& 7.9e-13 &       26 &  5 &  1/15&& 2.8e-12 &       26 &  9 &  3/ 6 \\ 
           cfd1 & 2.5e-07 &     1626 & 30 &  3/ 4&& 6.0e-12 &     1543 & 21 &  3/ 8&& 4.3e-12 &     1340 &  9 &  3/15&& 8.9e-12 &      719 &  4 &  1/15 \\ 
        financ. & 6.9e-12 &     1002 & 21 &  3/ 4&& 5.1e-13 &      700 &  9 &  3/ 8&& 1.0e-12 &      586 &  5 &  1/15&& 1.4e-12 &      600 &  6 &  1/15 \\ 
        Ga10As. & 9.4e-12 &     1893 & 15 &  3/ 4&& 3.7e-12 &     1372 &  5 &  1/ 8&& 1.8e-06 &     2198 &  6 &  2/15&& 1.1e-12 &     1558 &  6 &  1/ 8 \\ 
        Ga3As3. & 4.9e-12 &      569 & 11 &  3/ 4&& 2.7e-12 &      489 &  4 &  1/ 8&& 9.7e-13 &      471 &  4 &  1/15&& 9.9e-13 &      420 &  4 &  1/ 9 \\ 
 \hline 

\multicolumn{20}{c}{SLRP  for k smallest eigpair} \\ \hline 
        Andrew. & 4.6e-12 &      315 & 10 &  3/ 4&& 1.0e-12 &      199 &  5 &  1/ 8&& 9.9e-13 &      208 &  3 &  1/15&& 1.1e-12 &      185 &  5 &  1/ 8 \\ 
            C60 & 1.2e-12 &       16 &  9 &  2/ 4&& 6.7e-12 &       13 &  4 &  1/ 8&& 4.1e-13 &       16 &  3 &  1/15&& 9.2e-13 &       15 &  8 &  2/ 6 \\ 
           cfd1 & 9.1e-07 &     1956 & 30 &  3/ 4&& 7.9e-12 &     1394 & 22 &  3/ 8&& 5.2e-12 &     1121 & 12 &  3/15&& 4.2e-12 &     1017 & 12 &  3/15 \\ 
        financ. & 7.4e-12 &     1223 & 22 &  3/ 4&& 7.2e-13 &      713 & 11 &  3/ 8&& 1.6e-12 &      535 &  6 &  1/15&& 3.4e-12 &      467 &  5 &  1/15 \\ 
        Ga10As. & 1.6e-12 &     1625 &  8 &  3/ 4&& 2.6e-12 &      961 &  4 &  1/ 8&& 1.0e-12 &      999 &  3 &  1/15&& 4.0e-12 &      963 &  3 &  1/ 9 \\ 
        Ga3As3. & 4.8e-12 &      532 & 10 &  3/ 4&& 9.9e-13 &      381 &  4 &  1/ 8&& 9.8e-13 &      374 &  3 &  1/15&& 9.5e-13 &      338 &  5 &  1/ 9 \\ 
 \hline

%% file: tab-stat-cheby.tex
name &  maxres  & time & RR & p/d &&  maxres  & time & RR & p/d &&  maxres  & time & RR & p/d &&  maxres  & time & RR & p/d \\ \hline

\multicolumn{20}{c}{ computing k largest eigpair, tol=1e-6 } \\ \hline 
 Andrew.  & 2.6e-8 &     58 &   2 &   1/  5 && 1.1e-6 &     60 &   2 &   1/  3 && 3.0e-8 &     92 &   2 &   1/  5 && 6.0e-7 &     89 &   2 &   1/  3 \\ 
     C60  & 1.1e-9 &     13 &   2 &   1/  9 && 3.9e-8 &      9 &   2 &   1/  5 && 5.2e-7 &      9 &   2 &   1/  8 && 8.8e-6 &     12 &   3 &   1/  5 \\ 
    cfd1  & 5.6e-9 &    155 &   2 &   1/  3 && 7.4e-7 &    144 &   2 &   1/  2 && 1.5e-7 &    143 &   1 &   1/  3 && 1.5e-7 &    146 &   1 &   1/  3 \\ 
 financ.  & 1.6e-6 &     37 &   1 &   1/  3 && 1.5e-10 &     51 &   1 &   1/  3 && 1.1e-12 &     67 &   1 &   1/  3 && 1.2e-10 &     68 &   1 &   1/  3 \\ 
 Ga10As.  & 5.7e-8 &    264 &   2 &   1/  5 && 4.6e-8 &    550 &   4 &   1/  2 && 9.2e-7 &    380 &   2 &   1/  5 && 2.0e-7 &    484 &   3 &   1/  3 \\ 
 Ga3As3.  & 6.4e-8 &    101 &   2 &   1/  5 && 1.6e-6 &    112 &   3 &   1/  3 && 5.3e-7 &    136 &   2 &   1/  5 && 6.2e-6 &    125 &   2 &   1/  3 \\ 
\hline

\multicolumn{20}{c}{ computing k largest eigpair, tol=1e-12 } \\ \hline 
 Andrew.  & 8.8e-13 &    148 &   5 &   2/  5 && 2.9e-12 &    199 &   7 &   2/ 10 && 1.0e-12 &    125 &   4 &   1/  5 && 1.5e-12 &    160 &   7 &   3/  3 \\ 
     C60  & 2.0e-12 &     13 &   3 &   1/  9 && 4.6e-12 &     16 &   6 &   3/  5 && 5.5e-12 &     13 &   3 &   1/  9 && 4.5e-12 &     23 &  12 &   3/  5 \\ 
    cfd1  & 9.8e-13 &    190 &   4 &   1/  3 && 3.3e-13 &    230 &   6 &   2/  2 && 9.9e-13 &    188 &   3 &   1/  3 && 1.8e-12 &    215 &   4 &   1/  2 \\ 
 financ.  & 1.3e-12 &     97 &   3 &   2/  3 && 6.8e-12 &     87 &   3 &   1/  2 && 1.1e-12 &     69 &   1 &   1/  3 && 9.9e-13 &     93 &   2 &   1/  2 \\ 
 Ga10As.  & 9.6e-13 &    442 &   4 &   1/  5 && 9.0e-12 &    643 &   9 &   3/  3 && 9.7e-13 &    580 &   4 &   1/  6 && 1.3e-12 &    807 &   9 &   3/  3 \\ 
 Ga3As3.  & 1.7e-12 &    169 &   4 &   1/  5 && 2.2e-12 &    239 &   9 &   3/  3 && 1.7e-12 &    198 &   4 &   1/  5 && 4.7e-13 &    285 &   9 &   3/  3 \\ 
\hline

\multicolumn{20}{c}{ computing k smallest eigpair, tol=1e-6 } \\ \hline 
 Andrew.  & 4.2e-7 &    113 &   2 &   1/  8 && 6.1e-7 &    122 &   3 &   1/  5 && 5.2e-9 &    168 &   3 &   1/  8 && 2.6e-6 &    175 &   4 &   1/  5 \\ 
     C60  & 9.6e-7 &     16 &   4 &   2/  6 && 1.3e-6 &     11 &   3 &   1/  4 && 2.4e-6 &      9 &   3 &   1/  3 && 1.4e-6 &     10 &   4 &   1/  4 \\ 
    cfd1  & 3.4e-7 &    601 &   2 &   1/ 15 && 5.0e-6 &    427 &   2 &   1/ 15 && 4.8e-6 &    614 &   5 &   2/ 15 && 2.7e-6 &    607 &   5 &   2/ 15 \\ 
 financ.  & 1.7e-6 &    338 &   2 &   1/ 15 && 3.2e-6 &    310 &   2 &   1/ 10 && 5.3e-9 &    379 &   3 &   1/ 15 && 9.3e-7 &    333 &   3 &   1/ 10 \\ 
 Ga10As.  & 6.2e-6 &    751 &   2 &   1/  8 && 2.9e-6 &    744 &   3 &   1/  5 && 1.8e-6 &    715 &   2 &   1/  7 && 2.8e-6 &    907 &   3 &   1/  5 \\ 
 Ga3As3.  & 6.9e-6 &    325 &   2 &   1/  9 && 4.2e-7 &    269 &   2 &   1/  5 && 1.7e-9 &    282 &   3 &   1/  9 && 1.6e-6 &    369 &   5 &   2/  5 \\ 
\hline

\multicolumn{20}{c}{ computing k smallest eigpair, tol=1e-12 } \\ \hline 
 Andrew.  & 9.7e-13 &    200 &   4 &   1/  8 && 7.3e-12 &    243 &   8 &   3/  5 && 1.1e-12 &    185 &   5 &   1/  8 && 7.0e-12 &    293 &  11 &   3/  5 \\ 
     C60  & 2.8e-12 &     26 &   9 &   3/  6 && 3.4e-12 &     23 &   9 &   3/  4 && 9.2e-13 &     15 &   8 &   2/  6 && 2.1e-12 &     18 &  11 &   3/  4 \\ 
    cfd1  & 8.9e-12 &    719 &   4 &   1/ 15 && 8.7e-12 &   1033 &  11 &   3/ 15 && 4.2e-12 &   1017 &  12 &   3/ 15 && 9.0e-12 &   1471 &  23 &   3/ 15 \\ 
 financ.  & 1.4e-12 &    600 &   6 &   1/ 15 && 8.6e-12 &    587 &   8 &   3/ 10 && 3.4e-12 &    467 &   5 &   1/ 15 && 9.0e-12 &    637 &  10 &   3/ 10 \\ 
 Ga10As.  & 1.1e-12 &   1558 &   6 &   1/  8 && 7.6e-8 &   1629 &   9 &   3/ 15 && 4.0e-12 &    963 &   3 &   1/  9 && 5.2e-12 &   1496 &   9 &   3/  5 \\ 
 Ga3As3.  & 9.9e-13 &    420 &   4 &   1/  9 && 2.0e-12 &    547 &   9 &   3/  5 && 9.5e-13 &    338 &   5 &   1/  9 && 3.6e-12 &    573 &  14 &   3/  5 \\ 
\hline

%% file: tab-stat-all-maxeig.tex
name & maxres & time & SpMV && maxres & time & RR &&  maxres  & time & SpMV & RR/p/d &&  maxres  & time & SpMV & RR/p/d \\ \hline

\multicolumn{13}{c}{tol=1e-6 } \\ \hline 
 Andrew. & 1.0e-7 &    218 & 3e+3 && 1.0e-8 &    254 &   5 && 2.6e-8 &     58 & 6e+4 &   2/  1/  5 && 3.0e-8 &     92 & 6e+4 &   2/  1/  5 \\ 
     C60 & 4.9e-8 &     13 & 2e+3 && 7.9e-9 &     59 &   3 && 1.1e-9 &     13 & 5e+4 &   2/  1/  9 && 5.2e-7 &      9 & 3e+4 &   2/  1/  8 \\ 
    cfd1 & 2.5e-14 &    338 & 3e+3 && 4.2e-8 &    113 &   4 && 5.6e-9 &    155 & 6e+4 &   2/  1/  3 && 1.5e-7 &    143 & 4e+4 &   1/  1/  3 \\ 
 financ. & 3.1e-14 &    287 & 3e+3 && 6.1e-10 &     41 &   3 && 1.6e-6 &     37 & 2e+4 &   1/  1/  3 && 1.1e-12 &     67 & 3e+4 &   1/  1/  3 \\ 
 Ga10As. & 4.2e-14 &   1439 & 8e+3 && 1.6e+0 &   4704 &   2 && 5.7e-8 &    264 & 1e+5 &   2/  1/  5 && 9.2e-7 &    380 & 1e+5 &   2/  1/  5 \\ 
 Ga3As3. & 1.9e-8 &    353 & 5e+3 && 2.9e-1 &  11738 &  21 && 6.4e-8 &    101 & 7e+4 &   2/  1/  5 && 5.3e-7 &    136 & 6e+4 &   2/  1/  5 \\ 
 shallo. & 1.5e-10 &    774 & 8e+3 && 5.2e-9 &     69 &   4 && 4.9e-9 &    207 & 2e+5 &   2/  1/  7 && 9.2e-8 &    207 & 1e+5 &   2/  1/  7 \\ 
 Si10H1. & 5.6e-7 &     10 & 2e+3 && 2.6e-10 &     84 &   3 && 5.2e-9 &     11 & 4e+4 &   2/  1/  9 && 1.2e-10 &     11 & 3e+4 &   2/  1/  9 \\ 
  Si5H12 & 1.5e-12 &     13 & 2e+3 && 1.2e-8 &    170 &   3 && 1.0e-10 &     10 & 3e+4 &   2/  1/  6 && 4.6e-8 &     12 & 3e+4 &   2/  1/  6 \\ 
     SiO & 1.4e-13 &     58 & 3e+3 && 4.1e-7 &    265 &   2 && 1.4e-8 &     23 & 4e+4 &   2/  1/  5 && 4.1e-7 &     29 & 4e+4 &   2/  1/  5 \\ 
 wathen. & 5.5e-14 &     39 & 2e+3 && 6.0e-8 &     11 &   4 && 1.1e-6 &     10 & 2e+4 &   1/  1/  3 && 6.9e-11 &     26 & 4e+4 &   2/  1/  5 \\ 
 Ge87H7. & 1.7e-8 &   1451 & 8e+3 && 5.3e-9 &   8352 &   3 && 6.5e-10 &    439 & 2e+5 &   2/  1/  6 && 1.2e-7 &    392 & 1e+5 &   2/  1/  6 \\ 
 Ge99H1. & 2.5e-14 &   1636 & 8e+3 && 5.6e-7 &   6119 &   2 && 2.3e-9 &    348 & 1e+5 &   2/  1/  6 && 7.4e-8 &    402 & 1e+5 &   2/  1/  6 \\ 
 Si41Ge. & 1.1e-8 &   2909 & 9e+3 && 3.9e-7 &  14929 &   2 && 1.6e-9 &    863 & 2e+5 &   2/  1/  7 && 5.8e-8 &    708 & 1e+5 &   2/  1/  7 \\ 
 Si87H7. & 3.5e-14 &   3568 & 1e+4 && 2.8e-1 &   1702 &   1 && 4.0e-9 &   1126 & 3e+5 &   2/  1/  7 && 1.1e-7 &    882 & 1e+5 &   2/  1/  7 \\ 
 Ga41As. & 7.4e-14 &   4100 & 1e+4 && 8.6e-1 &   1066 &   1 && 1.2e-10 &   1029 & 2e+5 &   3/  1/  5 && 2.1e-7 &   1028 & 1e+5 &   2/  1/  7 \\ 
\hline

name & maxres & time & SpMV && maxres & time & RR &&  maxres  & time & SpMV & RR/p/d &&  maxres  & time & SpMV & RR/p/d \\ \hline

\multicolumn{13}{c}{tol=1e-12 } \\ \hline 
 Andrew. & 5.6e-14 &    232 & 4e+3 && 4.7e-14 &    489 &   9 && 8.8e-13 &    148 & 1e+5 &   5/  2/  5 && 1.0e-12 &    125 & 8e+4 &   4/  1/  5 \\ 
     C60 & 6.3e-13 &     15 & 2e+3 && 2.8e-13 &     89 &   5 && 2.0e-12 &     13 & 5e+4 &   3/  1/  9 && 5.5e-12 &     13 & 4e+4 &   3/  1/  9 \\ 
    cfd1 & 2.5e-14 &    296 & 3e+3 && 7.1e-14 &    204 &   8 && 9.8e-13 &    190 & 8e+4 &   4/  1/  3 && 9.9e-13 &    188 & 6e+4 &   3/  1/  3 \\ 
 financ. & 2.1e-14 &    283 & 3e+3 && 2.1e-14 &     67 &   5 && 1.3e-12 &     97 & 5e+4 &   3/  2/  3 && 1.1e-12 &     69 & 3e+4 &   1/  1/  3 \\ 
 Ga10As. & 4.8e-14 &   1784 & 8e+3 && 1.6e+0 &   4631 &   2 && 9.6e-13 &    442 & 2e+5 &   4/  1/  5 && 9.7e-13 &    580 & 2e+5 &   4/  1/  6 \\ 
 Ga3As3. & 2.1e-14 &    419 & 5e+3 && 2.9e-1 &  11245 &  21 && 1.7e-12 &    169 & 1e+5 &   4/  1/  5 && 1.7e-12 &    198 & 1e+5 &   4/  1/  5 \\ 
 shallo. & 4.6e-13 &    768 & 8e+3 && 1.9e-13 &    121 &   7 && 1.0e-12 &    234 & 2e+5 &   4/  1/  7 && 9.9e-13 &    280 & 2e+5 &   4/  1/  7 \\ 
 Si10H1. & 5.3e-14 &     11 & 2e+3 && 4.0e-13 &    104 &   4 && 6.2e-13 &     10 & 3e+4 &   2/  1/  9 && 3.7e-14 &     12 & 3e+4 &   3/  1/  9 \\ 
  Si5H12 & 1.1e-14 &     15 & 2e+3 && 2.6e-13 &    259 &   5 && 9.5e-13 &     11 & 3e+4 &   2/  1/  6 && 5.3e-12 &     15 & 3e+4 &   3/  1/  6 \\ 
     SiO & 1.4e-14 &     58 & 3e+3 && 4.7e-13 &    533 &   4 && 9.8e-13 &     33 & 5e+4 &   3/  1/  5 && 1.4e-12 &     45 & 6e+4 &   4/  1/  5 \\ 
 wathen. & 4.3e-14 &     36 & 2e+3 && 5.1e-14 &     24 &   8 && 1.1e-12 &     19 & 4e+4 &   2/  1/  5 && 9.8e-13 &     30 & 4e+4 &   3/  1/  5 \\ 
 Ge87H7. & 2.8e-14 &   1524 & 8e+3 && 1.3e-13 &  13993 &   5 && 4.8e-12 &    435 & 2e+5 &   3/  1/  6 && 1.0e-12 &    523 & 2e+5 &   4/  1/  6 \\ 
 Ge99H1. & 8.4e-14 &   1563 & 8e+3 && 2.1e-14 &  13438 &   5 && 3.7e-12 &    395 & 2e+5 &   2/  1/  6 && 9.6e-13 &    569 & 2e+5 &   4/  1/  6 \\ 
 Si41Ge. & 2.6e-14 &   2991 & 9e+3 && 2.5e-14 &  35270 &   5 && 9.9e-13 &    865 & 2e+5 &   3/  1/  7 && 1.1e-12 &    954 & 2e+5 &   3/  1/  7 \\ 
 Si87H7. & 2.8e-14 &   3506 & 1e+4 && 2.8e-1 &   1924 &   1 && 1.0e-12 &   1018 & 2e+5 &   3/  1/  7 && 1.4e-12 &   1102 & 2e+5 &   3/  1/  7 \\ 
 Ga41As. & 7.5e-14 &   4103 & 1e+4 && 8.6e-1 &   1242 &   1 && 7.9e-13 &   1135 & 2e+5 &   3/  1/  7 && 3.7e-12 &   1366 & 2e+5 &   3/  1/  7 \\ 
\hline

%% file: tab-stat-all-mineig.tex
name & maxres & time & SpMV && maxres & time & RR &&  maxres  & time & SpMV & RR/p/d &&  maxres  & time & SpMV & RR/p/d \\ \hline

\multicolumn{13}{c}{tol=1e-6 } \\ \hline 
 Andrew. & 4.9e-7 &    399 & 7e+3 && 8.5e-8 &    219 &   4 && 4.2e-7 &    113 & 1e+5 &   2/  1/  8 && 5.2e-9 &    168 & 1e+5 &   3/  1/  8 \\ 
     C60 & 2.2e-13 &      8 & 2e+3 && 6.4e-5 &    291 &  16 && 9.6e-7 &     16 & 5e+4 &   4/  2/  6 && 2.4e-6 &      9 & 2e+4 &   3/  1/  3 \\ 
    cfd1 & 4.7e-9 &   3871 & 6e+4 && 4.2e-8 &    167 &   7 && 3.4e-7 &    601 & 7e+5 &   2/  1/ 15 && 4.8e-6 &    614 & 6e+5 &   5/  2/ 15 \\ 
 financ. & 1.2e-9 &   1563 & 2e+4 && 4.5e-8 &     51 &   4 && 1.7e-6 &    338 & 4e+5 &   2/  1/ 15 && 5.3e-9 &    379 & 3e+5 &   3/  1/ 15 \\ 
 Ga10As. & 2.9e-12 &   2740 & 2e+4 && 8.9e-9 &   9302 &   4 && 6.2e-6 &    751 & 3e+5 &   2/  1/  8 && 1.8e-6 &    715 & 2e+5 &   2/  1/  7 \\ 
 Ga3As3. & 1.7e-12 &    599 & 8e+3 && 7.3e-8 &   1837 &   3 && 6.9e-6 &    325 & 2e+5 &   2/  1/  9 && 1.7e-9 &    282 & 2e+5 &   3/  1/  9 \\ 
 shallo. & 3.8e-14 &   1614 & 2e+4 && 6.1e-8 &     69 &   4 && 2.0e-8 &    400 & 4e+5 &   2/  1/ 14 && 4.1e-6 &    261 & 2e+5 &   2/  1/  9 \\ 
 Si10H1. & 1.5e-7 &     14 & 2e+3 && 1.2e-7 &    121 &   4 && 2.8e-7 &     13 & 5e+4 &   2/  1/  8 && 7.1e-6 &     12 & 3e+4 &   2/  1/  8 \\ 
  Si5H12 & 5.8e-12 &     21 & 3e+3 && 1.5e-8 &    166 &   3 && 3.3e-7 &     14 & 4e+4 &   2/  1/  8 && 6.5e-6 &     15 & 3e+4 &   2/  1/  8 \\ 
     SiO & 2.7e-13 &     97 & 5e+3 && 5.6e-8 &    537 &   4 && 4.1e-7 &     46 & 9e+4 &   2/  1/  8 && 8.8e-10 &     57 & 9e+4 &   3/  1/  8 \\ 
 wathen. & 1.4e-9 &    118 & 8e+3 && 8.4e-8 &     10 &   4 && 8.2e-6 &     61 & 2e+5 &   2/  1/ 15 && 2.4e-7 &     63 & 1e+5 &   3/  1/ 15 \\ 
 Ge87H7. & 2.0e-13 &   2559 & 1e+4 && 2.7e-8 &  11268 &   4 && 4.8e-7 &    509 & 3e+5 &   2/  1/  9 && 8.1e-10 &    641 & 2e+5 &   3/  1/  9 \\ 
 Ge99H1. & 2.1e-11 &   2319 & 1e+4 && 1.0e-8 &  11892 &   4 && 4.8e-7 &    568 & 3e+5 &   2/  1/  9 && 2.0e-6 &    564 & 2e+5 &   2/  1/  8 \\ 
 Si41Ge. & 4.1e-9 &   4650 & 1e+4 && 1.2e-8 &  25658 &   4 && 6.3e-7 &   1102 & 3e+5 &   2/  1/ 11 && 4.1e-10 &   1361 & 3e+5 &   3/  1/ 11 \\ 
 Si87H7. & 3.0e-13 &   5458 & 2e+4 && 3.3e+0 &   1842 &   1 && 3.2e-6 &   1201 & 3e+5 &   2/  1/ 11 && 7.4e-6 &   1243 & 2e+5 &   2/  1/ 10 \\ 
 Ga41As. & 3.6e-7 &  32279 & 8e+4 && 8.6e-1 &   1095 &   1 && 2.1e-8 &   3166 & 5e+5 &   3/  1/ 11 && 1.3e-6 &   3193 & 4e+5 &   3/  2/ 11 \\ 
\hline

name & maxres & time & SpMV && maxres & time & RR &&  maxres  & time & SpMV & RR/p/d &&  maxres  & time & SpMV & RR/p/d \\ \hline

\multicolumn{13}{c}{tol=1e-12 } \\ \hline 
 Andrew. & 1.2e-13 &    422 & 7e+3 && 4.1e-13 &    361 &   7 && 9.7e-13 &    200 & 2e+5 &   4/  1/  8 && 1.1e-12 &    185 & 2e+5 &   5/  1/  8 \\ 
     C60 & 2.6e-14 &      9 & 2e+3 && 6.4e-6 &    358 &  21 && 2.8e-12 &     26 & 7e+4 &   9/  3/  6 && 9.2e-13 &     15 & 4e+4 &   8/  2/  6 \\ 
    cfd1 & 2.9e-14 &   4209 & 6e+4 && 5.5e-14 &    383 &  16 && 8.9e-12 &    719 & 9e+5 &   4/  1/ 15 && 4.2e-12 &   1017 & 1e+6 &  12/  3/ 15 \\ 
 financ. & 9.7e-13 &   1776 & 2e+4 && 5.5e-14 &     93 &   8 && 1.4e-12 &    600 & 7e+5 &   6/  1/ 15 && 3.4e-12 &    467 & 4e+5 &   5/  1/ 15 \\ 
 Ga10As. & 2.8e-12 &   3479 & 2e+4 && 9.4e-14 &  17251 &   7 && 1.1e-12 &   1558 & 7e+5 &   6/  1/  8 && 4.0e-12 &    963 & 3e+5 &   3/  1/  9 \\ 
 Ga3As3. & 1.2e-12 &    571 & 8e+3 && 3.8e-13 &   2908 &   5 && 9.9e-13 &    420 & 3e+5 &   4/  1/  9 && 9.5e-13 &    338 & 2e+5 &   5/  1/  9 \\ 
 shallo. & 3.9e-14 &   1532 & 2e+4 && 2.7e-13 &    126 &   8 && 3.2e-12 &    600 & 6e+5 &   5/  1/ 12 && 4.0e-13 &    505 & 4e+5 &   5/  1/ 14 \\ 
 Si10H1. & 7.9e-14 &     18 & 2e+3 && 2.1e-12 &    198 &   7 && 2.0e-12 &     16 & 5e+4 &   4/  1/  8 && 3.9e-13 &     20 & 5e+4 &   5/  1/  8 \\ 
  Si5H12 & 1.5e-13 &     22 & 3e+3 && 3.6e-14 &    228 &   5 && 2.1e-12 &     20 & 6e+4 &   4/  1/  8 && 9.6e-12 &     23 & 6e+4 &   4/  1/  8 \\ 
     SiO & 2.7e-13 &     93 & 5e+3 && 2.7e-13 &    915 &   7 && 6.0e-13 &     64 & 1e+5 &   5/  1/  8 && 9.4e-13 &     68 & 1e+5 &   5/  1/  8 \\ 
 wathen. & 8.2e-13 &    146 & 8e+3 && 1.0e-13 &     18 &   7 && 3.1e-12 &    163 & 5e+5 &   6/  2/ 15 && 1.5e-12 &    120 & 3e+5 &   7/  2/ 15 \\ 
 Ge87H7. & 1.8e-13 &   2250 & 1e+4 && 1.5e-13 &  18852 &   7 && 2.6e-13 &    892 & 4e+5 &   5/  1/  9 && 9.9e-13 &    765 & 3e+5 &   5/  1/  9 \\ 
 Ge99H1. & 1.8e-13 &   2353 & 1e+4 && 6.7e-14 &  17683 &   7 && 9.7e-13 &    986 & 5e+5 &   4/  1/  9 && 9.9e-13 &    804 & 3e+5 &   4/  1/  9 \\ 
 Si41Ge. & 3.3e-13 &   4656 & 2e+4 && 1.3e-13 &  46386 &   7 && 9.9e-12 &   1705 & 5e+5 &   4/  1/ 11 && 9.8e-13 &   1568 & 3e+5 &   5/  1/ 11 \\ 
 Si87H7. & 3.0e-13 &   5487 & 2e+4 && 3.3e+0 &   1854 &   1 && 1.1e-12 &   2284 & 6e+5 &   6/  1/ 11 && 1.1e-12 &   1960 & 4e+5 &   5/  1/ 11 \\ 
 Ga41As. & 5.3e-12 &  33254 & 8e+4 && 8.6e-1 &    998 &   1 && 8.8e-13 &   5700 & 1e+6 &   7/  2/ 11 && 1.7e-12 &   3913 & 5e+5 &   5/  2/ 12 \\ 
\hline